\documentclass[preprint]{imsart}

\usepackage{a4wide}        
\RequirePackage{amsthm,amsmath,amsfonts,amssymb}
\RequirePackage[numbers]{natbib}
\RequirePackage[colorlinks,citecolor=blue,urlcolor=blue]{hyperref}
\RequirePackage{graphicx}
\usepackage{mathtools}
\usepackage{algpseudocode}
\usepackage{algorithm}
\usepackage{nicematrix}

\DeclareMathOperator*{\argmin}{arg\!\min}
\DeclareMathOperator*{\argmax}{arg\!\max}

\startlocaldefs
\numberwithin{equation}{section}
\theoremstyle{plain}
\newtheorem{theorem}{Theorem}[section]
\newtheorem{lemma}[theorem]{Lemma}
\newtheorem{proposition}[theorem]{Proposition}
\newtheorem{corollary}[theorem]{Corollary}
\theoremstyle{remark}

\newtheorem{example}[theorem]{Example}
\newtheorem{remark}[theorem]{Remark}

\endlocaldefs

\begin{document}

\begin{frontmatter}
\title{Early stopping for \( L^{2} \)-boosting in high-dimensional
       linear models}
\runtitle{Early stopping for \( L^{2} \)-boosting}

\begin{aug}
\author[A]{\fnms{Bernhard} \snm{Stankewitz}\ead[label=e1]{stankebe@math.hu-berlin.de}},
\address[A]{Department of Mathematics, Humboldt-University of Berlin, \printead{e1}}
\end{aug}

\begin{abstract}
  Increasingly high-dimensional data sets require that estimation methods do
  not only satisfy statistical guarantees but also remain computationally
  feasible.
  In this context, we consider \( L^{2} \)-boosting via orthogonal matching
  pursuit in a high-dimensional linear model and analyze a data-driven early
  stopping time \( \tau \) of the algorithm, which is sequential in
  the sense that its computation is based on the first \( \tau \)
  iterations only.
  This approach is much less costly than established model selection criteria,
  that require the computation of the full boosting path.
  We prove that sequential early stopping preserves statistical optimality in
  this setting in terms of a fully general oracle inequality for the empirical 
  risk and recently established optimal convergence rates for the population 
  risk.
  Finally, an extensive simulation study shows that at an immensely reduced
  computational cost, the performance of these type of methods is on par with
  other state of the art algorithms such as the cross-validated Lasso or model
  selection via a high dimensional Akaike criterion based on the full boosting
  path.
\end{abstract}



\end{frontmatter}


\section{Introduction}
\label{sec_Introduction}

Iterative estimation procedures typically have to be combined with a data-driven
choice \( \widehat{m} \) of the effectively selected iteration in order to avoid
under- as well as over-fitting.
In the context of increasingly high-dimensional data sets, which require that
estimation methods do not only provide statistical guarantees but also ensure
computational feasibility, established model selection criteria for \(
\widehat{m} \) such as \emph{cross-validation}, \emph{unbiased risk estimation},
\emph{Akaike's information criterion} or \emph{Lepski's balancing principle}
suffer from a disadvantage:
They involve computing the full iteration path up to some large 
\( m_{ \max } \), which is computationally costly, even if the final choice 
\( \widehat{m} \) is much smaller than \( m_{ \max } \).
In comparison, \emph{sequential early stopping}, i.e., halting the procedure at
an iteration \( \widehat{m} \) depending only on the iterates 
\( m \le \widehat{m} \), can substantially reduce computational complexity while
maintaining guarantees in terms of adaptivity.
For inverse problems, results were established in Blanchard and Mathé
\cite{BlanchardMathe2012}, Blanchard et al.
\cite{BlanchardEtal2018a,BlanchardEtal2018b}, Stankewitz
\cite{Stankewitz2020Smoothed} and Jahn \cite{Jahn2021Discrepancy}. 
A Poisson inverse problem was treated in Mika and Szkutnik
\cite{MikaSzkutnik2021DiscrepancyPrinciple} and general kernel learning in
Celisse and Wahl \cite{CelisseWahl2020Discrepancy}. 

In this work, we analyze sequential early stopping for an iterative boosting
algorithm applied to data \( 
  Y = ( Y_{i} )_{ i \le n }
\) from a high-dimensional linear model
\begin{align}
  \label{eq_1_HDLinearModel}
  Y_{i} =   f^{*}( X_{i} ) + \varepsilon_{i} 
        =   \sum_{ j = 1 }^{p} \beta_{j}^{*} X_{i}^{ (j) } 
          +
            \varepsilon_{i}, 
  \qquad i = 1, \dots, n,
\end{align}
where \( 
  f^{*}(x) = \sum_{ j = 1 }^{p} \beta^{*}_{j} x^{ (j) }, 
             x \in \mathbb{R}^{p}
\), is a linear function of the columns of the design matrix, \( 
  \varepsilon: = ( \varepsilon_{i} )_{ i \le n }
\) is the vector of centered noise terms in our observations and the parameter
size \( p \) is potentially much larger than the sample size \( n \). 
A large body of research has focused on developing methods that, given
reasonable assumptions on the design \( 
  \mathbf{X}: = ( X_{i}^{ (j) } )_{ i \le n, j \le p }
\) and the sparsity of the coefficients \( \beta^{*} \), consistently estimate
\( f^{*} \) despite the fact that \( p \gg n \).

\begin{minipage}{0.5\textwidth}
    \centering
    \begin{figure}[H]
      \includegraphics[width=0.99\linewidth]{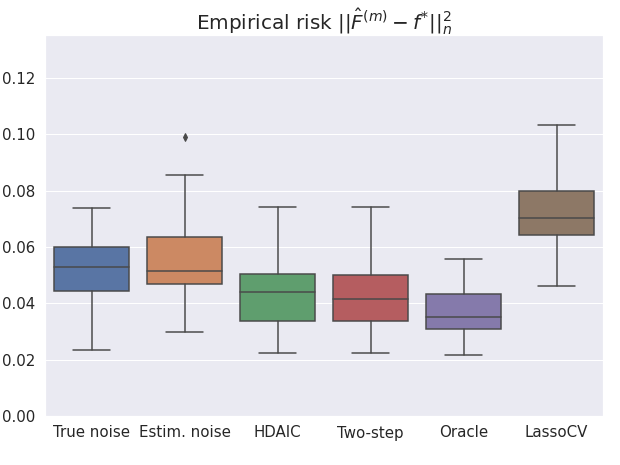}
      \\[-2.0ex]
      \caption{Empirical risk for different methods.}
      \label{fig_MethodComparison}
    \end{figure}
\end{minipage}
\begin{minipage}{0.5\textwidth}
  \centering
  \begin{table}[H]
    \centering
    \begin{tabular}{lr}
      \\
      \\
      \\
      \\
      \\
      \\[-1.8ex]\hline
      \\[-1.8ex]\hline\\[-1.8ex]
      True noise      &  19.8 sec \\
      Estimated noise &  32.0 sec \\ 
      Two-step        &  49.6 sec \\ 
      HDAIC           & 411.6 sec \\
      Lasso CV        & 164.3 sec \\
      \\[-1.8ex]\hline
      \hline\\[-1.8ex]
      \\
      \\
      \\
      \\
    \end{tabular}
    \caption{Computation times for different methods.}
    \label{tab_MethodComparisonComputationTimes} 
  \end{table}
\end{minipage}

\noindent Typically, approaches rely either on penalized least squares
estimation such as the Lasso, see e.g., B\"uhlmann and van der Geer
\cite{BuehlmannVDGeer2011HDData}, or on boosting type algorithms which
iteratively aggregate ``weak'' estimators with low accuracy to ``strong''
estimators with high accuracy, see Schapire and Freund
\cite{SchapireFreund2012Boosting} and B\"uhlmann
\cite{Buehlmann2006BoostingForHDLinearModels}.
Here, we focus on \( L^{2} \)-boosting based on \emph{orthogonal matching
pursuit} (OMP), which is one of the standard algorithms, particularly in signal
processing, see e.g., Tropp and Gilbert
\cite{TroppGilbert2007OrthogonalMatchingPursuit} or Needell and Vershynin
\cite{NeedellVershynin2010RegularizedOMP}. 
Temlyakov \cite{Temlyakov2000WeakGreedyAlgorithms} provided one of the first
deterministic analyses of OMP under the term \emph{orthogonal greedy algorithm}
(OGA).
In a statistical setting, where the non-linearity of OMP further complicates the
analysis, optimal convergence rates based on a high-dimensional Akaike criterion
have only been derived recently in Ing and Lai
\cite{IngLai2011ConsistentModelSelection} and Ing \cite{Ing2020ModelSelection}.

We sequentially stop OMP at \( \widehat{m} = \tau \) with
\begin{align}
  \label{eq_1_SequentialEarlyStoppingTime}
  \tau: = \inf \{ m \ge 0: r_{m}^{2} \le \kappa \} 
  \qquad \text{ and } \qquad 
  r_{m}^{2}: = \| Y - \widehat{F}^{ (m) } \|_{n}^{2}, 
  \quad m \ge 0,
\end{align}
where, at iteration \( m \), \( \widehat{F}^{ (m) } \) is the
OMP-estimator of \( f^{*} \) and \( r_{m}^{2} \) is the squared
\emph{empirical residual norm}.
\( \kappa > 0 \) is a critical value chosen by the user.
We consciously switch the notation from \( \widehat{m} \), for a general
data-driven selection criterion, to \( \tau \), indicating that the sequential
early stopping time is in fact a \emph{stopping time} in the sense of stochastic
process theory.
It is closely related to the \emph{discrepancy principle} which has been studied
in the analysis of inverse problems, see Engl et al.
\cite{EnglEtal1996InverseProblems}.
Most important to our analysis are the ideas developed in Blanchard et al.
\cite{BlanchardEtal2018a} and Ing \cite{Ing2020ModelSelection}. 

As an initial impression, Figure \ref{fig_MethodComparison} displays boxplots of
the empirical risk for five methods of stopping OMP:

\begin{itemize}

  \item[(i)]
    Early stopping with the choice \( \kappa \) equal to the true empirical
    noise level \( 
      \| \varepsilon \|_{n}^{2} 
      =
      n^{-1} \sum_{ i = 1 }^{n} \varepsilon_{i}^{2},
    \) which will be justified later;

  \item[(ii)]
    Early stopping with \( \kappa \) equal to an estimated noise level 
    \( \widehat{ \sigma }^{2} \), which approximates
    \( \| \varepsilon \|_{n}^{2} \); 

  \item[(iii)]
    OMP based on the full high-dimensional Akaike selection (HDAIC) from Ing
    \cite{Ing2020ModelSelection};

  \item[(iv)]
    A two-step procedure that combines early stopping based on an estimated
    noise level with an additional Akaike model selection step performed only
    over the iterations \( m \le \tau \). 

\end{itemize}

\noindent The plots are based on Monte Carlo simulations from model
\eqref{eq_1_HDLinearModel} with \( p = n = 1000 \) and a signal \( f^{*} \), the
sparsity of which is unknown to the methods.
As benchmarks, we additionally provide the values of the risk at the classical
oracle iteration \(
    m^{ \mathfrak{o} }:
  =
    \argmin_{ m \ge 0 } \| \widehat{F}^{ (m) } - f^{*} \|_{n}^{2}
\) and the default method {\tt LassoCV} from the python library 
{\tt scikit-learn} \cite{PedregosaEtal2011ScikitLearn} based on 5-fold
cross-validation.
The exact specifications of the simulation are in Section
\ref{sec_NumericalSimulationsAndATwoStepProcedure}. 
Table \ref{tab_MethodComparisonComputationTimes} contains the computation times
for the different methods.
The results suggest that sequential early stopping performs as well as
established exhaustive model selection criteria at an immensely reduced
computational cost, requiring only the computation of \( \tau \) iterations of
OMP.

The contribution of this paper is to provide rigorous theoretical guarantees
that justify this statement.
In the remainder of Section \ref{sec_Introduction}, we present our main results,
which are a fully general oracle inequality for the empirical risk at \( \tau \)
and an optimal adaptation guarantee for the population risk in terms of the
rates from Ing \cite{Ing2020ModelSelection}.
In Section \ref{sec_EmpiricalRisk}, we study the stopped empirical risk in
detail and provide precise bounds for important elementary quantities, which are
used to extend our results to the population risk in Section
\ref{sec_PopulationRisk}.
The analysis, which is conducted \( \omega \)-pointwise on the underlying
probability space, is able to avoid some of the saturation phenomena which
occurred in previous works, see Blanchard et al. \cite{BlanchardEtal2018a} and
Celisse and Wahl \cite{CelisseWahl2020Discrepancy} in particular.
Both of our main theorems require access to a rate-optimal estimator 
\( \widehat{ \sigma }^{2} \) of \( \| \varepsilon \|_{n}^{2} \).
Section \ref{sec_EstimationOfTheEmpiricalNoiseLevel} presents a noise estimation
result which shows that such estimators do exist and can be computed
efficiently.
Section \ref{sec_NumericalSimulationsAndATwoStepProcedure} provides a simulation
study, which illustrates our main findings numerically.
Finally, in the two-step procedure from (iv), we combine early stopping with a
second model selection step over the iterations \( m \le \tau \).
This procedure, which empirically outperforms the others, inherits the
guarantees for early stopping from our main results, while robustifying the
methodology against deviations in the stopping time. 

\subsection{A general oracle inequality for the empirical risk}
\label{ssec_AGeneralOracleInequalityForTheEmpiricalRisk}

In order to state results for sequential early stopping of OMP in model
\eqref{eq_1_HDLinearModel}, as minimal assumptions, we require that the rows 
\( ( X_{i} )_{ i \le n } \) of the design matrix \( 
  \mathbf{X} = ( X_{i}^{ (j) } )_{ i \le n, j \le p } 
\) are independently and identically distributed such that \( \mathbf{X} \)
has full rank \( n \) almost surely.
We also require that the noise terms \( ( \varepsilon_{i} )_{ i \le n } \) are
independently and identically distributed and assume that, conditional on the
design, a joint subgaussian parameter for the noise terms exists.

\begin{enumerate}

  \item [{\color{blue} (A1)}] \textbf{{\color{blue} (SubGE)}:} 
    \label{ass_SubGaussianErrors} 
    Conditional on the design, the noise terms are centered subgaussians with a
    joint parameter \( \overline{ \sigma }^{2} > 0 \), i.e., for all 
    \( i \le n \) and \( u \in \mathbb{R} \), 
    \begin{align*}
              \mathbb{E} ( e^{ u \varepsilon_{i} } | X_{i} ) 
      & \le
              e^{ \frac{ u^{2} \overline{ \sigma }^{2} }{2} }
      \qquad \text{almost surely.}
    \end{align*}
    Complementary to \( \overline{ \sigma }^{2} \), we set \( 
      \underline{ \sigma }^{2}: = \text{Var}( \varepsilon_{1} )
    \).

\end{enumerate}

\noindent By conditioning, we have \( 
      \underline{ \sigma }^{2} 
  =
      \mathbb{E} ( \mathbb{E} ( \varepsilon_{1}^{2} | X_{1} ) )
  \le
      \overline{ \sigma }^{2}
\).
Assumption
\hyperref[ass_SubGaussianErrors]{\normalfont \textbf{{\color{blue} (SubGE)}}}
permits heteroscedastic error terms \( ( \varepsilon_{i} )_{ i \le n } \),
allowing us to treat both regression and classification.

\begin{example}
  \label{expl_RegressionAndClassification}
  \
  \begin{enumerate}

    \item [(a)] \textbf{(Gaussian Regression):} 
      For \( 
        \varepsilon_{1}, \dots \varepsilon_{n} \sim N( 0, \sigma^{2} )
      \) i.i.d., we have \(
        \underline{ \sigma }^{2} = \sigma^{2} = \overline{ \sigma }^{2}  
      \).

    \item [(b)] \textbf{(Classification):} 
      For classification, we consider i.i.d. observations
      \begin{align}
        Y_{i} \sim \text{Ber}( f^{*}( X_{i} ) ), \qquad i = 1, \dots, n. 
      \end{align}
      Then, the noise terms are given by \( 
        \varepsilon_{i} = Y_{i} - f^{*}( X_{i} )
      \) with 
      \begin{align}
              \mathbb{E} ( \varepsilon_{i} | X_{i} ) 
        & = 
              f^{*}( X_{i} )         ( 1 - f^{*}( X_{i} ) ) 
            +
              ( 1 - f^{*}( X_{i} ) ) ( - f^{*}( X_{i} ) ) 
        =
            0.
      \end{align}
      Conditional on the design, the noise is bounded by one.
      This implies that \( \overline{ \sigma }^{2} \le 1 \). 

  \end{enumerate}
\end{example}

\noindent For the asymptotic analysis, we assume that the observations stem from
a sequence of models of the form \eqref{eq_1_HDLinearModel}, where 
\( p = p^{ (n) } \to \infty \) and \( \log ( p^{ (n) } ) / n \to 0 \) for 
\( n \to \infty \). 
We allow the quantities \(
               \mathbf{X} = \mathbf{X}^{ (n) }, \beta^{*} = ( \beta^{*} )^{ (n) } 
  \text{ and } \varepsilon = \varepsilon^{ (n) } 
\) to vary in \( n \). 
For notational convenience, we keep this dependence implicit.

In this setting, \( L^{2} \)-boosting based on OMP is used to estimate 
\( f^{*} \) and perform variable selection at the same time.
Empirical correlations between data vectors are measured via the \emph{empirical
inner product} \( 
  \langle a, b \rangle_{n}: = n^{-1} \sum_{ i = 1 }^{n} a_{i} b_{i}
\) with norm \( 
  \| a \|_{n}: = \langle a, a \rangle_{n}^{ 1 / 2 }, 
\) for \( a, b \in \mathbb{R}^{n} \).
By \( 
  \widehat{ \Pi }_{J}: \mathbb{R}^{n} \to \mathbb{R}^{n}, 
\) we denote the orthogonal projection with respect to 
\( \langle \cdot, \cdot \rangle_{n} \) onto the span of the columns 
\( \{ X^{ (j) }: j \in J \} \) of the design matrix.
OMP is initialized at \( \widehat{F}^{ (0) }: = 0 \) and then iteratively
selects the covariates \( X^{ (j) }, j \le p \), which maximize the empirical
correlation with the residuals \( Y - \widehat{F}^{ (m) } \) at the current
iteration \( m \).  
The estimator is updated by projecting onto the subspace spanned by the selected
covariates.
Explicitly, the procedure is given by the following algorithm:

\begin{center}
\begin{minipage}{.8\linewidth}
  \begin{algorithm}[H]
    \caption{Orthogonal matching pursuit (OMP)}
    \label{alg_OMP}
    \begin{algorithmic}[1]
      \vspace{2px}
      \State \( 
               \widehat{F}^{ (0) } \gets 0,
               \widehat{J}_{0} \gets \emptyset 
             \) 
      \vspace{2px}
      \For{ \( m = 0, 1, 2, \dots \) }
        \State \( 
                        \widehat{j}_{m + 1}
                 \gets 
                        \argmax_{ j \le p } 
                        \Big| 
                          \Big\langle 
                            Y - \widehat{F}^{ (m) }, 
                            \frac{ X^{ (j) } }
                                 { \| X^{ (j) } \|_{n} }
                          \Big\rangle_{n}
                        \Big| 
              \) 
        \vspace{2px}
        \State \( 
                       \widehat{J}_{ m + 1 } 
                 \gets
                       \widehat{J}_{m} \cup \big\{ \widehat{j}_{ m + 1 } \big\}
               \)
        \vspace{2px}
        \State \( 
                        \widehat{F}^{ ( m + 1 ) }
                  \gets
                        \widehat{ \Pi }_{ \widehat{J}_{ m + 1 } } Y
               \) 
        \vspace{2px}
        \vspace{2px}
      \EndFor
      \vspace{2px}
  \end{algorithmic}
  \end{algorithm}
\end{minipage}
\end{center}

\noindent Maximizing the empirical correlation between the residuals 
\( Y - \widehat{F}^{ (m) } \) and \( X^{ ( \widehat{j}_{ m + 1 } ) } \) at
iteration \( m \) is equivalent to minimizing \( 
  \| Y - \widehat{F}^{ ( m + 1 ) } \|_{n}^{2}
\), i.e., OMP performs \emph{greedy optimization} for the residual norm.
It is therefore natural to stop this procedure at \( \tau \) from Equation
\eqref{eq_1_SequentialEarlyStoppingTime} when the residual norm reaches a
critical value.

From a statistical perspective, we are interested in the risk of the estimators
\( \widehat{F}^{ (m) }, m \ge 0 \).
Initially, we consider the \emph{empirical risk}
\begin{align}
  \label{eq_EmipircalBiasVarianceDecomposition}
          \| \widehat{F}^{ (m) } - f^{*} \|_{n}^{2} 
  & = 
          \| ( I - \widehat{ \Pi }_{m} ) f^{*} \|_{n}^{2} 
        + 
          \| \widehat{ \Pi }_{m} \varepsilon \|_{n}^{2}  
  =
          b_{m}^{2} + s_{m}, 
\end{align}
where we introduce the notation \( 
  \widehat{ \Pi }_{m}: = \widehat{ \Pi }_{ \widehat{J}_{m} },
  b_{m}^{2}: = \| ( I - \widehat{ \Pi }_{m} ) f^{*} \|_{n}^{2} 
\) for the squared \emph{empirical bias} and \( 
  s_{m}: = \| \widehat{ \Pi }_{m} \varepsilon \|_{n}^{2}
\) for the \emph{empirical stochastic error}.\footnote{
  In the term \( b_{m}^{2} \), we use the standard overloading of notation,
  letting \( f^{*} \) denote \( ( f^{*}( X_{i} ) )_{ i \le n } \), 
  see Section \ref{ssec_FurtherNotation}.
}
Note that at this point, we cannot simply take expectations, due to the
non-linear, stochastic choice of \( \widehat{ J }_{ m } \).
The definition of the orthogonal projections 
\( ( \widehat{ \Pi }_{m} )_{ m \ge 0 } \) with respect to 
\( \langle \cdot, \cdot \rangle_{n} \) guarantees that the mappings 
\( m \mapsto b_{m}^{2} \) and \( m \mapsto s_{m} \) are monotonously decreasing
and increasing, respectively.

This reveals the fundamental problem of selecting an iteration of the procedure
in Algorithm \ref{alg_OMP}.  
We need to iterate far enough to sufficiently reduce the bias, yet not too far
as to blow up the stochastic error.
For \( m \ge n \), we have \(
  b_{m}^{2} = 0 \text{ but also } 
  s_{m}     = \| \varepsilon \|_{n}^{2}
\), which converges to \( \underline{ \sigma }^{2} \) by the law of large
numbers.
In particular, this means iterating Algorithm \ref{alg_OMP} indefinitely will
not produce a consistent estimator of the unknown signal \( f^{*} \).
Since the decay of the bias depends on \( f^{*} \), no a priori, i.e., data
independent, choice of the iteration will perform well in terms of the risk
uniformly over different realizations of \( f^{*} \). 
Therefore, Algorithm \ref{alg_OMP} needs to be combined with a data-driven
choice \( \widehat{m} \) of the effectively selected iteration, which is
\emph{adaptive}.
This means either, without prior knowledge of \( f^{*} \), the choice 
\( \widehat{m} \) satisfies an \emph{oracle inequality} relating its performance
to that of the ideal oracle iteration
\begin{align}
  \label{eq_1_ClassicalOracle}
      m^{ \mathfrak{o} } 
  =
      m^{ \mathfrak{o} }( f^{*} ):
  =
      \argmin_{ m \ge 0 } \| \widehat{F}^{ (m) } - f^{*} \|_{n}^{2} 
\end{align}
or, in terms of convergence rates for the risk, \( \widehat{m} \) performs
optimally for multiple classes of signals without prior knowledge of the class
to which the true signal \( f^{*} \) belongs. 

Our analysis in Section \ref{sec_EmpiricalRisk} shows that in order to derive
such an adaptation result for the sequential early stopping time \( \tau \) in
Equation \eqref{eq_1_SequentialEarlyStoppingTime}, ideally, the critical value
\( \kappa \) should be chosen depending on the iteration as
\begin{align}
      \kappa 
  = 
      \kappa_{m}
  = 
      \| \varepsilon \|_{n}^{2} + \frac{ C_{ \tau } m \log p }{n},
  \qquad m \ge 0,
\end{align}
where \( C_{ \tau } \ge  0 \) is a non-negative constant.
Since the \emph{empirical noise level} \( \| \varepsilon \|_{n}^{2} \) is
unknown, it has to be replaced by an estimator \( \widehat{ \sigma }^{2} \) and
we redefine 
\begin{align}
  \label{eq_1_SequentialEarlyStoppingTimeKappaM}
    \tau: 
  =
    \inf \{ m \ge 0: r_{m}^{2} \le \kappa_{m} \} 
  \qquad \text{ with } \qquad 
    \kappa_{m}: 
  = 
    \widehat{ \sigma }^{2} +  \frac{ C_{ \tau } m \log p }{n},
  \qquad m \ge 0.
\end{align}
Our first main result is an oracle inequality for the stopped empirical risk at
\( \tau \).

\begin{theorem}[Oracle inequality for the empirical risk]
  \label{thm_OracleInequalityForTheEmpiricalRisk}
  Under Assumption
  \hyperref[ass_SubGaussianErrors]{\normalfont \textbf{{\color{blue} (SubGE)}}},
  the empirical risk at the stopping time \( \tau \) in Equation
  \eqref{eq_1_SequentialEarlyStoppingTimeKappaM} with 
  \(
    C_{ \tau } \ge 8 \overline{ \sigma }^{2}
  \)
  satisfies
  \begin{align*}
            \| \widehat{F}^{ ( \tau ) } - f^{*} \|_{n}^{2} 
    & \le 
            \min_{ m \ge 0 } 
            \Big( 
              7 \| \widehat{F}^{ (m) } - f^{*} \|_{n}^{2} 
            + 
              \frac{ ( 8 \overline{ \sigma }^{2} + C_{ \tau } ) m \log p }{n}
            \Big) 
          + 
            | \widehat{ \sigma }^{2} - \| \varepsilon \|_{n}^{2} |
    \\
    & \le 
            7 \| \widehat{F}^{ ( m^{ \mathfrak{o} } ) } - f^{*} \|_{n}^{2} 
          + 
            \frac{
                   ( 8 \overline{ \sigma }^{2} + C_{ \tau } ) 
                   m^{ \mathfrak{o} } \log p 
                 }{n}
          + 
            | \widehat{ \sigma }^{2} - \| \varepsilon \|_{n}^{2} |
  \end{align*}
  with probability converging to one.
\end{theorem}

\noindent The oracle inequality is completely general in the sense that no
assumption on \( f^{*} \) is required.
In particular, the result also holds for non-sparse \( f^{*} \).
The first term on the right-hand side involving the iteration \( m^{
  \mathfrak{o} } \) from Equation \eqref{eq_1_ClassicalOracle} is of optimal
order and the second term matches the upper bound for the empirical stochastic
error at iteration \( m^{ \mathfrak{o} } \) we derive in Lemma
\ref{lem_BoundForTheEmpiricalStochasticError}.
The last term is the  absolute estimation error of \( \widehat{ \sigma }^{2} \)
for the empirical noise level.
The result is closely related to Theorem 3.3 in Blanchard et al.
\cite{BlanchardEtal2018a}.
Whereas they state their oracle inequality in expectation, ours is formulated \(
\omega \)-pointwise on the underlying probability space, which is slightly
stronger. 
In particular, this leads to the term \( 
  | \widehat{ \sigma }^{2} - \| \varepsilon \|_{n}^{2} |
\)
in the inequality, which will be essential for the noise estimation problem,
see Section \ref{sec_EstimationOfTheEmpiricalNoiseLevel}.


\subsection{Optimal adaptation for the population risk}
\label{ssec_OptimalAdaptationForThePopulationRisk}

The \emph{population} counterpart of the empirical inner product is \( 
  \langle f, g \rangle_{ L^2 }: = \mathbb{E} ( f( X_{1} ) g( X_{1} ) )
\) with norm \( 
  \| f \|_{ L^2 }: = \langle f, f \rangle_{ L^2 }^{ 1 / 2 }
\) for functions \( 
  f, g \in L^{2}( \mathbb{P}^{ X_{1} } ),
\) where \( \mathbb{P}^{ X_{1} } \) denotes the distribution of one observation
of the covariates.
Identifying \( \widehat{F}^{ (m) } \) with its corresponding function in the
covariates, the \emph{population risk} of the estimators is given by \( 
  \| \widehat{F}^{ (m) } - f^{*} \|_{ L^{2} }^{2},
  m \ge 0
\).
Assuming that all of the covariates are square-integrable, for \(
  J \subset \{ 1, \dots, p \}
\), let \( 
  \Pi_{J}: L^{2}( \mathbb{P}^{ X_{1} } ) \to  L^{2}( \mathbb{P}^{ X_{1} } )
\) denote the orthogonal projection with respect to \(
  \langle \cdot, \cdot \rangle_{ L^{2} }
\) onto the span of the covariates \(
  \{ X_{1}^{ (j) }: j \in J \} 
\).
Setting \(
  \Pi_{m}: = \Pi_{ \widehat{J}_{m} } 
\), the population risk decomposes into
\begin{align}
  \label{eq_PopulationBiasVarianceDecomposition}
          \| \widehat{F}^{ (m) } - f^{*} \|_{ L^{2} }^{2} 
  & =  
          \| ( I - \Pi_{m} ) f^{*} \|_{ L^{2} }^{2} 
        + 
          \| \widehat{F}^{ (m) } - \Pi_{m} f^{*} \|_{ L^{2} }^{2}
  = 
          B_{m}^{2} + S_{m},
\end{align}
where \( 
    B_{m}^{2}:a
  =
    \| ( I - \Pi_{m} ) f^{*} \|_{ L^{2} }^{2}
\) is the squared \emph{population bias} and \( 
    S_{m}:
  =
    \| \widehat{F}^{ (m) } - \Pi_{m} f^{*} \|_{ L^{2} }^{2}
\) is the \emph{population stochastic error}.

Note that \( B_{m}^{2} \) and \( S_{m} \) are not the exact population
counterparts of the empirical quantities \( b_{m}^{2} \) and \( s_{m} \), since
we have to account for the difference between \( \widehat{ \Pi }_{m} \) and \(
\Pi_{m} \).
The challenge of selecting the iteration in Algorithm \ref{alg_OMP} discussed in
the previous section is the same for the population risk.
The mapping \( m \mapsto B_{m}^{2} \) is monotonously decreasing, \( S_{0} = 0
\) and \( S_{m} \) approaches \( \text{Var}( \varepsilon_{1} ) \) for \( m \to
\infty \) assuming that the difference between \( \widehat{ \Pi }_{m} \) and \(
\Pi_{m} \) becomes negligible.
Due to this difference, however, the mapping \( m \mapsto S_{m} \) is no longer
guaranteed to be monotonous.
Both \( B_{m} \) and \( S_{m} \) are still random quantities due to the
randomness of \( \widehat{J}_{m} \).

In order to derive guarantees for the population risk, additional assumptions
are required.
We quantify the sparsity of the coefficients \( \beta^{*} \) of \( f^{*} \):

\begin{enumerate}

  \item [{\color{blue} (A2)}] \textbf{{\color{blue} (Sparse)}:} 
    \label{ass_Sparse} 
    We assume one of the two following assumptions holds.
    \begin{enumerate}

      \item[(i)] 
        \( \beta^{*} \) is \( s \)-sparse for some \( s \in \mathbb{N}_{0} \),
        i.e., \(
          \| \beta^{*} \|_{0} \le s 
        \), where \( \| \beta^{*} \|_{0} \) is the cardinality of the support \( 
          S: = \{ j \le p: | \beta^{*}_{j} | \ne 0 \}
        \). 
        Additionally, we require that
        \begin{align*}
              s \| \beta^{*} \|_{1}^{2}
          = 
              s \Big( \sum_{ j = 1 }^{p} | \beta^{*}_{j} | \Big)^{2} 
          = 
              o \Big( \frac{n}{ \log p } \Big),
          \quad 
          \| f^{*} \|_{ L^{2} }^{2} \le C_{ f^{*} } 
          \quad \text{ and } \quad 
              \min_{ j \in S } | \beta^{*}_{j} | 
          \ge 
              \underline{ \beta },
        \end{align*}
        where \( C_{ f^{*} }, \underline{ \beta } > 0 \) are numerical
        constants.

      \item[(ii)]
        \( \beta^{*} \) is $ \gamma $-sparse for some \(
          \gamma \in [ 1, \infty ) 
        \), i.e.,  \( 
          \| \beta^{*} \|_{2} \le C_{ \ell^{2} }
        \) and 
        \begin{align*}
                  \sum_{ j \in J } | \beta_{j}^{*} | 
          & \le
                  C_{ \gamma } 
                  \Big( 
                    \sum_{ j \in J } | \beta_{j}^{*} |^{2}
                  \Big)^{ \frac{ \gamma - 1 }{ 2 \gamma - 1 } } 
          \qquad \text{ for all } J \subset \{ 1, \dots, p \},
        \end{align*}
        where $ C_{ \ell^{2} }, C_{ \gamma } > 0 $ are numerical constants.

    \end{enumerate}

\end{enumerate}

\noindent Assumptions like \hyperref[ass_Sparse]{\normalfont
\textbf{{\color{blue} (Sparse)}}} (i) are standard in the literature on high
dimensional models, see e.g., Bühlmann and van de Geer
\cite{BuehlmannVDGeer2011HDData}.  
Note that the conditions in (i) imply that \( 
  s = o( ( n / \log p )^{ 1 / 3 } ) 
\). 
\hyperref[ass_Sparse]{\normalfont \textbf{{\color{blue} (Sparse)}}} (ii) encodes
a decay of the coefficients \( \beta^{*} \). 
It includes several well known settings as special cases.

\begin{example}
  \label{expl_Sparsity}
  \
  \begin{enumerate}

    \item[(a)] \textbf{(\( \ell^{ 1 / \gamma } \)-boundedness):}
      For \(
        \gamma \in [ 1, \infty ) 
      \) and some \(
        C_{ \ell^{ 1 / \gamma } } >
      0 
      \), let the coefficients satisfy \(
            \sum_{ j = 1 }^{p} | \beta^{*}_{j} |^{ 1 / \gamma } 
        \le
            C_{ \ell^{ 1 / \gamma } }
      \).
      Then, Hoelder's inequality yields
      \begin{align}
                \sum_{ j \in J } 
                | \beta^{*}_{j} | 
        & = 
                \sum_{ j \in J } 
                | \beta^{*}_{j} |^{ \frac{1}{ 2 \gamma - 1 } } 
                | \beta^{*}_{j} |^{ \frac{ 2 \gamma - 2 }{ 2 \gamma - 1 } } 
        \le 
                \Big( 
                  \sum_{ j \in J } 
                  | \beta^{*}_{j} |^{ \frac{1}{ \gamma } } 
                \Big)^{ \frac{ \gamma }{ 2 \gamma - 1 }  } 
                \Big( 
                  \sum_{ j \in J } 
                  | \beta^{*}_{j} |^{2} 
                \Big)^{ \frac{ \gamma - 1 }{ 2 \gamma - 1 } } 
        \\
        & \le 
                ( C_{ \ell^{ 1 / \gamma } } )^{ \frac{ \gamma }{ 2 \gamma - 1 } } 
                \Big( 
                  \sum_{ j \in J } 
                  | \beta^{*}_{j} |^{2} 
                \Big)^{ \frac{ \gamma - 1 }{ 2 \gamma - 1 } } 
        \qquad \text{ for all } J \subset \{ 1, \dots, p \},
        \notag
      \end{align}
      i.e., Assumption
      \hyperref[ass_Sparse]{\normalfont \textbf{{\color{blue} (Sparse)}}} (ii)
      is satisfied with \(
          C_{ \gamma }
        =  
          ( C_{ \ell^{ 1 / \gamma } } )^{ \frac{ \gamma }{ 2 \gamma - 1 } } 
      \). 
      For \( \gamma \to \infty \), this approaches the setting in (i), in which
      the support \( S \) is finite.

    \item[(b)] \textbf{(Polynomial decay):} 
      For \(
        \gamma \in ( 1, \infty ) 
      \) and \( 
        C_{ \gamma }' \ge c_{ \gamma }' > 0 
      \), let the coefficients satisfy
      \begin{align}
        \label{eq_1_PolyDecay}
            c_{ \gamma }' j^{ - \gamma } 
        \le 
            | \beta^{*}_{ (j) } | 
        \le 
            C_{ \gamma }' j^{ - \gamma } 
        \qquad 
        \qquad \text{ for all } j \le p,
      \end{align}
      where \(
        ( \beta^{*}_{ (j) } )_{ j \le p } 
      \) is a reordering of the \(
        ( \beta^{*}_{j} )_{ j \le p } 
      \) with decreasing absolute values.
      Then, Assumption
      \hyperref[ass_Sparse]{\normalfont \textbf{{\color{blue} (Sparse)}}} (ii)
      is satisfied with \( C_{ \gamma } \) proportional to \( 
        C_{ \gamma }' 
        ( c_{ \gamma }' )^{ - ( 2 \gamma - 2 ) / ( 2 \gamma - 1 ) }, 
      \) see Lemma A1.2 of Ing \cite{Ing2020ModelSelection}. 

  \end{enumerate}

\end{example}

For the covariance structure of the design, we assume subgaussianity and some
additional boundedness conditions.

\begin{enumerate}

  \item [{\color{blue} (A3)}] \textbf{{\color{blue} (SubGD)}:} 
    \label{ass_SubGaussianDesign} 
    The design variables are centered subgaussians in \( \mathbb{R}^{p} \) with
    unit variance, i.e., there exists some \( \rho > 0 \) such that for all \(
      x \in \mathbb{R}^{p} \text{ with } \| x \| = 1,
    \)
    \begin{align*}
              \mathbb{E} e^{ u \langle x, X_{1} \rangle }
      & \le 
              e^{ \frac{ u^{2} \rho^{2} }{2} }, 
              \quad u \in \mathbb{R} 
      \qquad \text{ and } \qquad 
              \text{Var} ( X_{1}^{ (j) } ) = 1
              \quad \text{ for all }  j \le p.
    \end{align*}
\end{enumerate}

\begin{remark}[Inclusion of an Intercept]
  \label{rem_InclusionOfAnInterceptercept}
  Assumption
  \hyperref[ass_SubGaussianDesign]{\normalfont \textbf{{\color{blue} (SubGD)}}}
  still allows to include an intercept additional to the design variables.
  If the intercept is selected, we have just applied Algorithm \ref{alg_OMP} to
  the data \( Y \) centered at their empirical mean for which
  \hyperref[ass_SubGaussianDesign]{\normalfont \textbf{{\color{blue} (SubGD)}}}
  is satisfied up to a negligible term.
  If it is not selected, the result is identical to applying the Algorithm
  without an intercept.
\end{remark}

\begin{enumerate}

  \item [{\color{blue} (A4)}] \textbf{{\color{blue} (CovB)}:} 
    \label{ass_CovB} 
    The complete covariance matrix \( \Gamma: = \text{Cov}( X_{1} ) \) of one
    design observation is bounded from below, i.e., there exists some 
    \( c_{ \lambda } > 0 \) such that the smallest eigenvalue of \( \Gamma \)
    satisfies
    \begin{align}
      \lambda_{ \min }( \Gamma ) \ge c_{ \lambda } > 0. 
    \end{align}
    Further, we assume that there exists \(
      C_{ \text{Cov} } > 0 
    \) such that the partial population covariance matrices \( 
      \Gamma_{J}: = ( \Gamma_{ j k } )_{ j, k \in J }, 
    \) for \(
      J \subset \{ 1, \dots, p \}, 
    \) satisfy
    \begin{align}
      \label{eq_1_CovB_2}
          \sup_{ | J | \le M_{n}, k \not \in J } 
          \| \Gamma^{-1}_{J} v_{k} \|_{1} 
      <
          C_{ \text{Cov} }
    \end{align}
    with \( 
        M_{n}: 
      =
        \sqrt{ n / ( ( \bar \sigma^{2} + \rho^{4} ) \log p ) } 
    \), where \(
      v_{k}: =   ( \text{Cov}( X_{1}^{ (k) }, X_{1}^{ (j) } ))_{ j \in J } 
             \in \mathbb{R}^{ | J | } 
    \) is the vector of covariances between the \( k \)-th covariate and the
    covariates from the set \( J \).

\end{enumerate}
\(
  ( \Gamma_{J}^{-1} v_{k} )_{ j \in J } 
\) is the vector of coefficients for the \(
  X_{1}^{ (j) }, j \in J 
\), in the conditional expectation \(
  \mathbb{E} ( X_{1}^{ (k) } | X_{1}^{ (j) },\\ j \in J ) 
\).
\( M_{ n } \) will be the largest iteration of Algorithm \ref{alg_OMP} for which
we need control over the covariance structure.
Condition \eqref{eq_1_CovB_2} imposes a restriction on the correlation between
the covariates.

\begin{example}[]
  \label{expl_BoundedCovariance}
  \ 
  \begin{enumerate}

    \item[(a)] \textbf{(Uncorrelated design):} 
      For \( \Gamma = I_{p} \), condition \eqref{eq_1_CovB_2} is satisfied for
      any choice \( C_{ \text{Cov} } > 0 \), since the left-hand side of the
      condition is zero.

    \item[(b)] \textbf{(Bounded cumulative coherence):} 
      For \( m \ge 0 \) and \( J \subset \{ 1, \dots, p \} \), let
      \begin{align}
            \mu_{1}(m): 
        = 
            \max_{ k \le p } 
            \max_{ | J | \le m, k \not \in J } 
            \sum_{ j \in J } \text{Cov}( X_{1}^{ (k) }, X_{1}^{ (j) } ) 
      \end{align}
      be the \emph{cumulative coherence function}. 
      Then, 
      \begin{align}
            \sup_{ | J | \le M_{ n },  k \not \in J } 
            \| \Gamma^{-1}_{J} v_{k} \|_{1} 
      & \le 
            \sup_{ | J | \le M_{ n }, k \not \in J } 
            \| \Gamma_{J}^{-1} \|_{1} 
            \mu_{1}( M_{ n } ),
      \end{align}
      where \( \| \Gamma_{J}^{-1} \|_{1} \) denotes the column sum norm.
      Under the assumption that both quantities on the right-hand side are
      bounded, condition \eqref{eq_1_CovB_2} is satisfied.
      Under the stronger assumption that \( \mu_{1}( M_{ n } ) < 1 / 2, \) it
      can be shown that condition \eqref{eq_1_CovB_2} is satisfied with 
      \( C_{ \text{Cov} } = 1 \). 
      This is the \emph{exact recovery condition} in Theorem 3.5 of Tropp
      \cite{Tropp2004Greed}.
      
  \end{enumerate}

\end{example}

\noindent As in Ing \cite{Ing2020ModelSelection}, condition \eqref{eq_1_CovB_2}
guarantees that the coefficients \(
  \beta( ( I - \Pi_{J} ) f^{*} ) 
\) of the population residual term \(
  ( I - \Pi_{J} ) f^{*} 
\) satisfy
\begin{align}
  \label{eq_1_UniformBaxtersInequality}
       \| \beta( ( I - \Pi_{J} ) f^{*} ) \|_{1} 
    = 
       \| \beta^{*} - \beta( \Pi_{J} f^{*} ) \|_{1} 
  \le 
        ( C_{ \text{Cov} } + 1 ) 
        \sum_{ j \not \in J } | \beta^{*}_{j} |
  \\
  \qquad \qquad \qquad \qquad \qquad \qquad \qquad 
  \text{ for all } | J | \le M_{ n }, 
  \notag 
\end{align}
where \( J \) ranges over all subsets of \(  \{ 1, \dots, p \} \).
A derivation is stated in Lemma \ref{lem_UniformBaxtersInequality}.
Equation \eqref{eq_1_UniformBaxtersInequality} provides a uniform bound on the
vector difference of the finite time predictor coefficients 
\( \beta( \Pi_{J} f^{*} ) \) of \( \Pi_{J} f^{*} \) and the infinite time
predictor coefficients \( \beta^{*} \).
In the literature on autoregressive modeling, such an inequality is referred to
as a \emph{uniform Baxter's inequality}, see Ing \cite{Ing2020ModelSelection}
and the references therein, Baxter \cite{Baxter1962FinitePredictor} and Meyer et
al. \cite{MeyerEtal2015Baxter}.  

Under Assumptions 
\hyperref[ass_Sparse]{\normalfont \textbf{{\color{blue} (A1)}}} -
\hyperref[ass_CovB]{\normalfont \textbf{{\color{blue} (A4)}}},
explicit bounds for the population bias and the stochastic error are available.
In the formulation of the results, the postpositioned ``with probability
converging to one'' always refers to the whole statement including
quantification over \( m \ge 0 \), see also Section \ref{ssec_FurtherNotation}.

\begin{lemma}[Bound for the population stochastic error, Ing \cite{Ing2020ModelSelection}]
  \label{lem_BoundForThePopulationStochasticError}
  Under Assumptions
  \hyperref[ass_SubGaussianErrors]{\normalfont \textbf{{\color{blue} (SubGE)}}},
  \hyperref[ass_Sparse]{\normalfont \textbf{{\color{blue} (Sparse)}}},
  \hyperref[ass_SubGaussianDesign]{\normalfont \textbf{{\color{blue} (SubGD)}}} 
  and
  \hyperref[ass_CovB]{\normalfont \textbf{{\color{blue} (CovB)}}},
  there is a constant \( C_{ \text{Stoch} } > 0 \) such that 
  \begin{align*}
            S_{m}
    \le 
            C_{ \text{Stoch} } 
            \begin{dcases}
              \frac{ 
                (
                  \overline{ \sigma }^{2} 
                +
                  \| \beta^{*} \|_{1}^{2} \rho^{4} \mathbf{1} \{ m \le \tilde m \}
                )
                m \log p 
              }{n},
              & \beta^{*} \ s \text{-sparse}, \\
              \frac{ ( \overline{ \sigma }^{2} + \rho^{4} ) m \log p }{n},
              & \beta^{*} \ \gamma \text{-sparse}
            \end{dcases}
    \\
    \qquad \text{ for all } m \ge 0 
  \end{align*}
  with probability converging to one, where \(
    \tilde m = \inf \{ m \ge 0: S \subset \widehat{J}_{m} \} 
  \).
\end{lemma}

\noindent The stochastic error grows linearly in \( m \), whereas, up to lower
order terms, the bias decays exponentially when \( \beta^{*} \) is 
\( s \)-sparse and with a rate \( m^{ 1 - 2 \gamma } \) when \( \beta^{*} \) is
\( \gamma \)-sparse.

\begin{proposition}[Bound for the population bias, Ing \cite{Ing2020ModelSelection}]
  \label{prp_BoundForThePopulationBias}
  Under Assumptions 
  \hyperref[ass_SubGaussianErrors]{\normalfont \textbf{{\color{blue} (SubGE)}}},
  \hyperref[ass_Sparse]{\normalfont \textbf{{\color{blue} (Sparse)}}},
  \hyperref[ass_SubGaussianDesign]{\normalfont \textbf{{\color{blue} (SubGD)}}}
  and
  \hyperref[ass_CovB]{\normalfont \textbf{{\color{blue} (CovB)}}},
  there are constants \(
    c_{ \text{Bias} }, C_{ \text{Bias} } > 0 
  \) such that 
  \begin{align*}
            B_{m}^{2}
    & \le  
            C_{ \text{Bias} }
            \begin{dcases}
                \| f^{*} \|_{ L^{2} }^{2}
                \exp \Big( \frac{ - c_{ \text{Bias} } m }{s} \Big) 
              + 
                \| \beta^{*} \|_{1}^{2}  
                \frac{ s \log p }{n}, 
              & \beta^{*} \ s\text{-sparse}, \\
                m^{ 1 - 2 \gamma } 
              + 
                \Big(
                  \frac{ ( \overline{ \sigma }^{2} + \rho^{4} ) \log p }{n}
                \Big)^{ 1 - \frac{1}{ 2 \gamma } },
              & \beta^{*} \ \gamma \text{-sparse} \\
            \end{dcases}
    \\
    & \qquad \qquad \qquad \qquad \qquad \qquad 
      \qquad \qquad \qquad \qquad \ \ \ \text{ for all } m \ge 0 
  \end{align*}
  with probability converging to one.
  When \( \beta^{*} \) is \( s \)-sparse, on the corresponding event, there is a
  constant \( C_{ \text{supp} } > 0 \) such that \( 
    S \subset \widehat{J}_{ C_{ \text{supp} } s } 
  \).
\end{proposition}

\noindent Lemma \ref{lem_BoundForThePopulationStochasticError} and Proposition
\ref{prp_BoundForThePopulationBias} are essentially proven in Ing
\cite{Ing2020ModelSelection} but not stated explicitly.
We include derivations in Appendix \ref{sec_ProofsForSupplementaryResults} to
keep this paper self-contained.

Under \( s \)-sparsity, the definition of the population bias guarantees that
\begin{align}
  B_{m}^{2} = 0 \qquad \text{ for all } m \ge C_{ \text{supp} } s
\end{align}
with probability converging to one and under \( \gamma \)-sparsity, the upper
bounds from Lemma \ref{lem_BoundForThePopulationStochasticError} and Proposition
\ref{prp_BoundForThePopulationBias} balance at an iteration of size \( 
  (
    n / ( ( \overline{ \sigma }^{2} + \rho^{4} ) \log p ) 
  )^{ 1 / ( 2 \gamma ) }
\). 
We obtain that for
\begin{align}
  \label{eq_1_BalancedBoundOracleIndex}
      m^{*}_{ s, \gamma }: 
  = 
      \begin{dcases}
        C_{ \text{supp} } s,
        & \beta^{*} \ s \text{-sparse}, \\ 
        \Big(
          \frac{n}{ ( \overline{ \sigma }^{2} + \rho^{4} ) \log p }
        \Big)^{ \frac{1}{ 2 \gamma } }, 
        & \beta^{*} \ \gamma \text{-sparse}, \\ 
      \end{dcases}
\end{align}
there exists a constant \( C_{ \text{Risk} } > 0 \) such that with probability
converging to one, the population risk satisfies
\begin{align}
  \label{eq_1_RatesPopulationRisk}
          \| \widehat{F}^{ ( m^{*}_{ s, \gamma } ) } - f^{*} \|_{ L^2 }^{2} 
  & \le 
          C_{ \text{Risk} } 
          \mathcal{R}( s, \gamma )
\end{align}
with the rates
\begin{align}
  \label{eq_1_MinimaxRates}
      \mathcal{R}( s, \gamma ):
  = 
      \begin{dcases}
        \frac{ \overline{ \sigma }^{2} s \log p }{n},
        & \beta^{*} \ s \text{-sparse}, \\ 
          \Big(
            \frac{ ( \overline{ \sigma }^{2} + \rho^{4} ) \log p }{n}
          \Big)^{ 1 - \frac{1}{ 2 \gamma } },
        & \beta^{*} \ \gamma \text{-sparse}.
      \end{dcases}
\end{align}
In Lemmas \ref{lem_BoundForTheEmpiricalBias} and
\ref{lem_BoundForTheEmpiricalStochasticError}, we show that the empirical
quantities \( b_{m}^{2} \) and \( s_{m} \) satisfy  bounds analogous to those
stated in Proposition \ref{prp_BoundForThePopulationBias} and Lemma
\ref{lem_BoundForThePopulationStochasticError}, such that also 
\begin{align}
  \label{eq_1_RatesEmpiricalRisk}
          \| \widehat{F}^{ ( m^{*}_{ s, \gamma } ) } - f^{*} \|_{n}^{2} 
  & \le 
          C_{ \text{Risk} } 
          \mathcal{R}( s, \gamma )
\end{align}
under the respective assumptions.

In general, we cannot expect to improve the rates \( \mathcal{R}( s, \gamma ) \)
neither for the population nor for the empirical risk, see Ing
\cite{Ing2020ModelSelection} and our discussion in Section
\ref{ssec_ExplicitBoundsForEmpiricalQuantities}. 
Consequently, under \( s \)-sparsity, we call a data-driven selection criterion
\( \widehat{m} \) adaptive to a parameter set \( T \subset \mathbb{N}_{0} \) for
one of the two risks, if the choice \( \widehat{m} \) attains the rate 
\( \mathcal{R}( s, \gamma ) \) simultaneously over all \( s \)-sparse signals 
\( f^{*} \) with \( s \in T \), without any prior knowledge of \( s \).
We call \( \widehat{m} \) optimally adaptive, if the above holds for 
\( T = \mathbb{N}_{0} \).
Under \( \gamma \)-sparsity, we define adaptivity analogously with parametersets
\( T \subset [ 1, \infty ) \) instead.
Ideally, \( \widehat{m} \) would be optimally adaptive both under \( s \)- and
\( \gamma \)-sparsity even without any prior knowledge about what class of
sparsity assumption is true for a given signal.
Ing \cite{Ing2020ModelSelection} proposes to determine \( \widehat{m} \) via a
high-dimensional Akaike criterion, which is in fact optimally adaptive for the
population risk under both sparsity assumptions.
In order to compute \( \widehat{m} \), however, the full iteration path of
Algorithm \ref{alg_OMP} has to be computed as well.

Our second main result states that optimal adaptation is also achievable by a
computationally efficient procedure, given by the early stopping rule in
Equation \eqref{eq_1_SequentialEarlyStoppingTimeKappaM}.
The proof of Theorem \ref{thm_OptimalAdaptationForThePopulationRisk} is
developed in Section \ref{sec_PopulationRisk}.

\begin{theorem}[Optimal adaptation for the population risk]
  \label{thm_OptimalAdaptationForThePopulationRisk}
  Under Assumptions 
  \hyperref[ass_SubGaussianErrors]{\normalfont \textbf{{\color{blue} (SubGE)}}},
  \hyperref[ass_Sparse]{\normalfont \textbf{{\color{blue} (Sparse)}}},
  \hyperref[ass_SubGaussianDesign]{\normalfont \textbf{{\color{blue} (SubGD)}}}
  and
  \hyperref[ass_CovB]{\normalfont \textbf{{\color{blue} (CovB)}}},
  choose \( \widehat{ \sigma }^{2} \) in Equation
  \eqref{eq_1_SequentialEarlyStoppingTimeKappaM} such that there is a constant
  \( C_{ \text{Noise} } > 0 \) for which
  \begin{align*}
          | \widehat{ \sigma }^{2} - \| \varepsilon \|_{n}^{2} | 
    & \le 
          C_{ \text{Noise} } \mathcal{R}( s, \gamma ) 
  \end{align*}
  with probability converging to one.
  Then, the population risk at the stopping time in Equation
  \eqref{eq_1_SequentialEarlyStoppingTimeKappaM} with 
  \(
    C_{ \tau } = c ( \overline{ \sigma }^{2} + \rho^{4} ) 
  \)
  for any \( c > 0 \) satisfies
  \begin{align*}
          \| \widehat{F}^{ ( \tau ) } - f^{*} \|_{ L^2 }^{2} 
    & \le 
          C_{ \text{PopRisk} } \mathcal{R}( s, \gamma ) 
  \end{align*}
  with probability converging to one for a constant 
  \( C_{ \text{PopRisk} } > 0 \).
\end{theorem}

Under the additional Assumptions
\hyperref[ass_Sparse]{\normalfont \textbf{{\color{blue} (Sparse)}}},
\hyperref[ass_SubGaussianDesign]{\normalfont \textbf{{\color{blue} (SubGD)}}}
and
\hyperref[ass_CovB]{\normalfont \textbf{{\color{blue} (CovB)}}},
the bounds in Lemmas \ref{lem_BoundForTheEmpiricalBias} and
\ref{lem_BoundForTheEmpiricalStochasticError} also allow to translate Theorem
\ref{thm_OracleInequalityForTheEmpiricalRisk} into optimal convergence rates by
setting \( m = m^{*}_{ s, \gamma } \) from Equation
\eqref{eq_1_BalancedBoundOracleIndex}:

\begin{corollary}[Optimal adaptation for the empirical risk]
  \label{cor_OptimalAdaptationForTheEmpiricalRisk}
  Under Assumptions 
  \hyperref[ass_SubGaussianErrors]{\normalfont \textbf{{\color{blue} (SubGE)}}},
  \hyperref[ass_Sparse]{\normalfont \textbf{{\color{blue} (Sparse)}}},
  \hyperref[ass_SubGaussianDesign]{\normalfont \textbf{{\color{blue} (SubGD)}}}
  and
  \hyperref[ass_CovB]{\normalfont \textbf{{\color{blue} (CovB)}}},
  the empirical risk at the stopping time in Equation
  \eqref{eq_1_SequentialEarlyStoppingTimeKappaM} with 
  \( C_{ \tau } = C ( \overline{ \sigma }^{2} + \rho^{4} ) \)
  with \( C \ge 12 \) satisfies
  \begin{align*}
    \label{cor_RatesForTheEmpiricalRisk}
              \| \widehat{F}^{ ( \tau ) } - f^{*} \|_{n}^{2} 
    & \le 
              C_{ \text{EmpRisk} } \mathcal{R}( s, \gamma ) 
            + 
              | \widehat{ \sigma }^{2} - \| \varepsilon \|_{n}^{2} |,
  \end{align*}
  with probability converging to one for a constant 
  \( C_{ \text{EmpRisk} } > 0 \).
\end{corollary}

In order for sequential early stopping to be adaptive over a parameter subset 
\( T \) from \( \mathbb{N}_{0} \) or \( [ 1, \infty ) \), all of our results
above require an estimator of the empirical noise level that attains the rates
\( \mathcal{R}( s, \gamma ) \) for the absolute loss.
In Proposition \ref{prp_FastNoiseEstimation}, we show that such an estimator
does in fact exist, even for \( T \) equal to \( \mathbb{N}_{0} \) and 
\( [ 1, \infty ) \). 
Together, this establishes that an optimally adaptive, fully sequential choice
of the iteration in Algorithm \ref{alg_OMP} is possible.
This is a strong positive result, given the fact that in previous settings
adaptations has only been possible for restricted subsets of parameters, see
Blanchard et al. \cite{BlanchardEtal2018a} and Celisse and Wahl
\cite{CelisseWahl2020Discrepancy}.
The two-step procedure, which we analyze in detail in Section
\ref{sec_NumericalSimulationsAndATwoStepProcedure}, further robustifies this
method against deviations in the stopping time and reduces the assumptions
necessary for the noise estimation.


\subsection{Further notation}
\label{ssec_FurtherNotation}

We overload both the notation of the empirical and the population inner products
with functions and vectors respectively, i.e., for \( 
  f, g: \mathbb{R}^{p} \to \mathbb{R}, 
\) we set
\begin{align}
    \| f \|_{n}^{2}:
  =
    \frac{1}{n} \sum_{i = 1}^{n} f(X_{i})^{2}  
  \qquad \text{ and } \qquad 
    \langle f, g \rangle_{n}: 
  = 
    \frac{1}{n} \sum_{i = 1}^{n} f(X_{i}) g(X_{i})
\end{align}
and also, e.g., 
\begin{align}
    \langle \varepsilon, f \rangle_{n}: 
  = 
    \frac{1}{n} \sum_{ i = 1 }^{n} \varepsilon_{i} f( X_{i} ) 
  \qquad
  \text{ or }
  \qquad 
    \| Y \|_{ L^2 }^{2}: 
  =
    \mathbb{E} Y_{1}^{2}. 
\end{align}
Further, as in B\"uhlmann \cite{Buehlmann2006BoostingForHDLinearModels}, for 
$ j \le p $, we denote the $ j $-th coordinate projection as \( 
  g_{j}(x): = x^{(j)}, x \in \mathbb{R}^{p},
\) and vectors of dot products via
\begin{align}
    \langle \cdot, g_{J} \rangle_{n}: 
  = 
    ( \langle \cdot, g_{j} \rangle_{n} )_{ j \in J } 
    \in \mathbb{R}^{ | J | } 
  \qquad
  \text{ and }
  \qquad 
    \langle \cdot, g_{J} \rangle_{ L^2 }: 
  = 
    ( \langle \cdot, g_{j} \rangle_{ L^2 } )_{ j \in J } 
    \in \mathbb{R}^{ | J | } 
\end{align}
for \( J \subset \{ 1, \dots, p \} \).
This way, Equation \eqref{eq_1_CovB_2} in Assumption 
\hyperref[ass_CovB]{\normalfont \textbf{{\color{blue} (CovB)}}}
can be restated as 
\begin{align}
      \sup_{ | J | \le M_{n}, k \not \in J } 
      \| \Gamma_{J}^{-1} \langle g_{k}, g_{J} \rangle_{ L^{2} } \|_{1} 
  \le 
      C_{ \text{Cov} }.
\end{align}

Analogously to the population covariance matrix \( 
  \Gamma = ( \langle g_{j}, g_{k} \rangle_{ L^2 } )_{ j, k \le p } 
\), we define the empirical covariance matrix \( 
    \widehat{ \Gamma }: 
  = 
    ( \langle g_{j}, g_{k} \rangle_{n} )_{ j, k \le p }
\).
Using the same notation for partial matrices as in Assumption
\hyperref[ass_CovB]{\normalfont \textbf{{\color{blue} (CovB)}}} and \( 
      X^{ (J) } 
  =
      ( X_{i}^{ (j) } )_{ i \le n, j \in J } 
  \in
      \mathbb{R}^{ n \times | J | }
\), the projections \( \widehat{ \Pi }_{J} \) and \( \Pi_{J} \) can be written
as
\begin{align}
  \label{eq_1_Projections}
  \widehat{ \Pi }_{J}:   \mathbb{R}^{n} \to \mathbb{R}^{n},
  \qquad \qquad \qquad \ \ \
      \widehat{ \Pi }_{J} y:
  & =
      ( X^{ (J) } )^{ \top } 
      \widehat{ \Gamma }_{J}^{-1} 
      \langle y, g_{J} \rangle_{n},
  \\
  \Pi_{J}: L^{2}( \mathbb{P}^{ X_{1} } ) \to L^{2}( \mathbb{P}^{ X_{1} } ),
  \qquad 
      \Pi_{J} h:
  & =
      g_{J}^{ \top } \Gamma_{J}^{-1} \langle h, g_{J} \rangle_{ L^2 }
  \notag 
\end{align}
for \( J \subset \{ 1, \dots, p \} \).

At points where we switch between a linear combination \( f \) of the columns
of the design and its coefficients, we introduce the notation \( \beta(f) \). 
We use this, e.g., for the coefficients of the population residual function \( 
  \beta( ( I - \Pi_{m} ) f^{*} )
\) as in Equation \eqref{eq_1_UniformBaxtersInequality}.
For coefficients \( \beta \in \mathbb{R}^{p} \), we also use the general set
notation \( 
  \beta_{J}: = ( \beta_{j} )_{ j \in J } \in \mathbb{R}^{ | J | }
\) for \(
  J \subset \{ 1, \dots, p \}
\).

Throughout the paper, variables \( c > 0 \) and \( C > 0 \) denote small and
large constants respectively.
They may change from line to line and can depend on constants defined in our
assumptions. 
They are, however, independent of \( n, \overline{ \sigma }^{2} \) and 
\( \rho^{2} \). 

Many statements in our results are formulated with a postpositioned ``with
probability converging to one''.
This always refers to the whole statement including quantifiers.
E.g. in Lemma \ref{lem_BoundForTheEmpiricalStochasticError}, the result is to be
read as:
There exists an event with probability converging to one on which for all
iterations \( m \ge 0 \), the inequality \( 
  s_{m} \le C \overline{ \sigma }^{2} m \log(p) / n
\) is satisfied.



\section{Empirical risk analysis}
\label{sec_EmpiricalRisk}

Since the stopping time \( \tau \) in Equation
\eqref{eq_1_SequentialEarlyStoppingTime} is defined in terms of the squared
empirical residual norm \(
  r_{m}^{2}  m \ge 0 
\), its functioning principles are initially best explained by analyzing the
stopped empirical risk \(
  \| \widehat{F}^{ ( \tau ) } - f^{*} \|_{n}^{2}
\). 
We begin by formulating an intuition for why \( \tau \) is adaptive.

\subsection{An intuition for sequential early stopping}
\label{ssec_AnIntuitionForSequentialEarlyStopping}

Ideally, an adaptive choice \( \widehat{m} \) of the iteration in Algorithm
\ref{alg_OMP} would approximate the classical oracle iteration 
\( m^{ \mathfrak{o} } \) from Equation \eqref{eq_1_ClassicalOracle}, which
minimizes the empirical risk.
The sequential stopping time \( \tau \), however, does not have a direct
connection to \( m^{ \mathfrak{o} } \). 
In fact, its sequential definition guarantees that \( \tau \) does
not incorporate information about the squared bias \( b_{m}^{2} \)
for iterations \( m > \tau \).
Instead, \( \tau \) mimics the \emph{balanced oracle iteration}
\begin{align}
      m^{ \mathfrak{b} }
  = 
      m^{ \mathfrak{b} }(f^{*}): 
  = 
      \inf \{ m \ge 0: b_{m}^{2} \le s_{m} \}.
\end{align}
Fortunately, the empirical risk at \( m^{ \mathfrak{b} } \) is essentially
optimal up to a small discretization error, which opens up the possibility of 
sequential adaptation in the first place.

\begin{lemma}[Optimality of the balanced oracle]
  \label{lem_OptimalityOfTheBalancedOracle}
  The empirical risk at the balanced oracle iteration \( m^{ \mathfrak{b} } \) 
  satisfies
  \begin{align*}
            \| \widehat{F}^{ ( m^{ \mathfrak{b} } ) } - f^{*} \|_{n}^{2} 
    & \le 
            2 \| \widehat{F}^{ ( m^{ \mathfrak{o} } ) } - f^{*} \|_{n}^{2} 
          + 
            \Delta( s_{ m^{ \mathfrak{b} } } ) 
    = 
            2 \min_{ m \ge 0 } 
            \| \widehat{F}^{ (m) } - f^{*} \|_{n}^{2} 
            + 
            \Delta( s_{ m^{ \mathfrak{b} } } ),
  \end{align*}
  where \( 
      \Delta( s_{m} ): = s_{m} - s_{ m - 1 }
  \) is the discretization error of the empirical stochastic error at \( m \).
\end{lemma}

\begin{proof}[Proof]
  If \(
    m^{ \mathfrak{b} } > m^{ \mathfrak{o} } 
  \), then the definition of \( m^{ \mathfrak{b} } \) and the monotonicity of 
  \( m \mapsto b_{m}^{2} \) yield
  \begin{align}
            \| \widehat{F}^{ ( m^{ \mathfrak{b} } ) } - f^{*} \|_{n}^{2} 
    & = 
            b_{ m^{ \mathfrak{b} } }^{2}
          + 
            s_{ m^{ \mathfrak{b} } } 
    \le 
            2 b_{ m^{ \mathfrak{b} } }^{2}
          + 
            \Delta( s_{ m^{ \mathfrak{b} } } ) 
    \le 
            2 b_{ m^{ \mathfrak{o} } }^{2} + \Delta( s_{ m^{ \mathfrak{b} } } ) 
    \\
    & \le 
            2 \| \widehat{F}^{ ( m^{ \mathfrak{o} } ) } - f^{*} \|_{n}^{2} 
          + 
            \Delta( s_{ m^{ \mathfrak{b} } } ).
    \notag
  \end{align}
  Otherwise, if \(
    m^{ \mathfrak{b} } \le m^{ \mathfrak{o} } 
  \), then analogously, the monotonicity of \( m \mapsto s_{m} \) yields \( 
        \| \widehat{F}^{ ( m^{ \mathfrak{b} } ) } - f^{*} \|_{n}^{2} 
    \le 2 s_{ m^{ \mathfrak{b} } } 
    \le 2 \| \widehat{F}^{ ( m^{ \mathfrak{o} } ) } - f^{*} \|_{n}^{2} 
  \).
\end{proof}

The connection between \( \tau \) and \( m^{ \mathfrak{b} } \) can be seen by
decomposing the squared residual norm \( r_{m}^{2} \) into
\begin{align}
  \label{eq_2_DecompositionEmpiricalResiduals}
          r_{m}^{2} 
  & = 
          \| ( I - \widehat{ \Pi }_{m} ) f^{*} \|_{n}^{2}
        +
          2 \langle ( I - \widehat{ \Pi }_{m} ) f^{*}, \varepsilon \rangle_{n} 
        +
          \| ( I - \widehat{ \Pi }_{m} ) \varepsilon \|_{n}^{2} 
  \\
  & = 
          b_{m}^{2} + 2 c_{m} + \| \varepsilon \|_{n}^{2} - s_{m} 
  \notag
\end{align}
with the \emph{cross term} 
\begin{align}
  \label{eq_2_CrossTerm}
  c_{m}: = \langle ( I - \widehat{ \Pi }_{m} ) f^{*}, \varepsilon \rangle_{n}, 
  \qquad m \ge 0. 
\end{align}
Indeed, Equation \eqref{eq_2_DecompositionEmpiricalResiduals} yields that the
stopping condition \( r_{m}^{2} \le \kappa \) is equivalent to
\begin{align}
  \label{eq_2_ExplicitStoppingCondition}
          b_{m}^{2} + 2 c_{m}
  & \le 
          s_{m} + \kappa - \| \varepsilon \|_{n}^{2}. 
\end{align}
Assuming that \( c_{m} \) can be treated as a lower order term, this implies
that, up to the difference \( \kappa - \| \varepsilon \|_{n}^{2} \), 
\( \tau \) behaves like \( m^{ \mathfrak{b} } \).

The connection between a discrepancy-type stopping rule and a balanced oracle
was initially drawn in Blanchard et al.
\cite{BlanchardEtal2018a,BlanchardEtal2018b}.
Whereas their oracle quantities were defined in terms of non-random population
versions of bias and variance, ours have to be defined \( \omega \)-pointwise on
the underlying probability space.
This is owed to the fact that, even conditional on the design \( \mathbf{X} \),
the squared bias \( b_{m}^{2} \) is still a random quantity due to the random
selection of \( \widehat{J}_{m} \) in Algorithm \ref{alg_OMP}. 
This is a subtle but important distinction, which leads to a substantially
different analysis.


\subsection{A general oracle inequality}
\label{ssec_AGeneralOracleInequality}

In this section, we derive the first main result in Theorem
\ref{thm_OracleInequalityForTheEmpiricalRisk}.
As in Blanchard et al. \cite{BlanchardEtal2018a}, the key ingredient is that 
via the squared residual norm \( r_{m}^{2}, m \ge 0 \), the stopped estimator 
\( \widehat{F}^{ ( \tau ) } \) can be compared with any other estimator 
\( \widehat{F}^{ (m) } \) in empirical norm.
Note that the statement in Lemma \ref{lem_EmpiricalNormComparison} is
completely deterministic.

\begin{lemma}[Empirical norm comparison]
  \label{lem_EmpiricalNormComparison}
  For any \( m \ge 0 \), 
  the stopped estimator \( \widehat{F}^{ ( \tau ) } \) 
  with \( \tau \) from Equation \eqref{eq_1_SequentialEarlyStoppingTime}
  satisfies
  \begin{align*}
            \| \widehat{F}^{ ( \tau ) } - \widehat{F}^{ (m) } \|_{n}^{2} 
    & \le 
            \| \widehat{F}^{ (m) } - f^{*} \|_{n}^{2} 
          + 
            2 | c_{m} | 
            \\
    &     + 
            ( \kappa - \| \varepsilon \|_{n}^{2} )
            \mathbf{1} \{ \tau < m \} 
          + 
            ( \| \varepsilon \|_{n}^{2} + \Delta( r_{ \tau }^{2} ) - \kappa )
            \mathbf{1} \{ \tau > m \}, 
    \notag
  \end{align*}
  where 
  \( c_{m} \) is the cross term from Equation
  \eqref{eq_2_CrossTerm} 
  and
  \( 
    \Delta( r_{m}^{2} ): = r_{m}^{2} - r_{ m - 1 }^{2}
  \) 
  is the discretization error of the squared residual norm at \( m \).
\end{lemma}

\begin{proof}[Proof]
  Fix \( m \ge 0 \).
  We have
  \begin{align}
    \label{eq_ComparingTauWithADeterministicIndex_Decomposition}
            \|
              \widehat{F}^{ ( \tau ) } 
            - \widehat{F}^{ (m) }
            \|_{n}^{2} 
    & = 
            \| 
              Y - \widehat{F}^{ ( \tau ) } 
            + \widehat{F}^{ (m) } - Y
            \|_{n}^{2} 
    = 
            r_{ \tau }^{2} 
          - 
            2 \langle 
              ( I - \widehat{ \Pi }_{ \tau } ) Y, 
              ( I - \widehat{ \Pi }_{ m    } ) Y 
            \rangle_{n} 
          + 
            r_{m}^{2} 
    \notag 
    \\
    & = 
            ( r_{m}^{2} - r_{ \tau }^{2} ) 
            \mathbf{1} \{ \tau > m \} 
          + 
            ( r_{ \tau }^{2} - r_{m}^{2} ) 
            \mathbf{1} \{ \tau < m \}.
  \end{align}
  On \( \{ \tau > m \} \), we use the definition of \( \tau \) in 
  Equation \eqref{eq_1_SequentialEarlyStoppingTime}
  to estimate 
  \begin{align}
            r_{m}^{2} - r_{ \tau }^{2} 
    & \le 
            r_{m}^{2} - \kappa + \Delta( r_{ \tau }^{2} ) 
    =  
            b_{m}^{2} + 2 c_{m} + \| \varepsilon \|_{n}^{2} - s_{m} 
          - \kappa + \Delta( r_{ \tau }^{2} ) 
    \\
    & \le 
            \| \widehat{F}^{ (m) } - f^{*} \|_{n}^{2} 
          + 
            2 c_{m} 
          + 
            \| \varepsilon \|_{n}^{2} - \kappa + \Delta( r_{ \tau }^{2} ).
    \notag
  \end{align}
  On \( \{ \tau \le m \} \), analogously, we obtain \( 
          r_{ \tau }^{2} - r_{m}^{2} 
    \le 
          \| \widehat{F}^{ (m) } - f^{*} \|_{n}^{2} 
        -  
          2 c_{m} 
        + 
          \kappa - \| \varepsilon \|_{n}^{2},
  \) which finishes the proof.
\end{proof}

In order to translate this norm comparison to an oracle inequality, it suffices
to control the cross term and the discretization error of the residual norm.
This is already possible under Assumption  
\hyperref[ass_SubGaussianErrors]{\normalfont \textbf{{\color{blue} (SubGE)}}}.
The proof of Lemma \ref{lem_BoundsForTheCrossTermAndTheDiscretizationError} is deferred to Appendix \ref{sec_ProofsForAuxiliaryResults}. 

\begin{lemma}[Bounds for the cross term and the discretization error]
  \label{lem_BoundsForTheCrossTermAndTheDiscretizationError}
  Under Assumption
  \hyperref[ass_SubGaussianErrors]{\normalfont \textbf{{\color{blue} (SubGE)}}}, 
  the following statements hold:
  \begin{enumerate}

    \item [(i)] With probability converging to one, the cross term satisfies
      \begin{align*}
              | c_{m} |
        & \le 
              b_{m}
              \sqrt{ \frac{ 4 \overline{ \sigma }^{2} ( m + 1 ) \log p }{n} }
        \qquad \text{ for all } m \ge 0.
      \end{align*}

    \item [(ii)] With probability converging to one, the discretization error
      of the squared residual norm satisfies
      \begin{align*}
        \Delta( r_{m}^{2} )
        & \le 
            2 b_{ m - 1 }^{2}
          + 
            \frac{ 8 \overline{ \sigma }^{2} m \log p }{n}
        \qquad \text{ for all } m \ge 1.
      \end{align*}

  \end{enumerate}
\end{lemma}

Together, Lemmas \ref{lem_EmpiricalNormComparison} and
\ref{lem_BoundsForTheCrossTermAndTheDiscretizationError} motivate the choice 
\( \kappa = \kappa_{m} \) in Equation
\eqref{eq_1_SequentialEarlyStoppingTimeKappaM}, where the additional term 
\( C_{ \tau } m \log(p) / n \) accounts for the discretization error of
the residuals norm.
With this choice of \( \kappa \), Lemma \ref{lem_EmpiricalNormComparison}
yields for any fixed \( m \ge 0 \) that
\begin{align}
          \| \widehat{F}^{ ( \tau ) } - f^{*} \|_{n}^{2} 
  & \le 
          2
          ( 
            \| \widehat{F}^{ ( \tau ) } - \widehat{F}^{ (m) } \|_{n}^{2} 
          + 
            \| \widehat{F}^{ (m) } - f^{*} \|_{n}^{2} 
          ) 
  \\
  & \le 
          2
          \big( 
            2 \| \widehat{F}^{ (m) } - f^{*} \|_{n}^{2} 
          + 
            2 | c_{m} | 
        + 
            ( \kappa_{ \tau } - \| \varepsilon \|_{n}^{2} )
            \mathbf{1} \{ \tau < m \}. 
  \notag
  \\
  & \qquad \ \ \quad \qquad \qquad \qquad \qquad \ \ + 
            (
              \| \varepsilon \|_{n}^{2} 
            + 
               \Delta( r_{ \tau }^{2} ) - \kappa_{ \tau } 
            )
            \mathbf{1} \{ \tau >  m \} 
          \big).
  \notag
\end{align}
Under Assumption 
\hyperref[ass_SubGaussianErrors]{\normalfont \textbf{{\color{blue} (SubGE)}}},
with probability converging to one, the estimates from Lemma
\ref{lem_BoundsForTheCrossTermAndTheDiscretizationError} then imply that
on \( \{ \tau < m \} \),
\begin{align}
            \| \widehat{F}^{ ( \tau ) } - f^{*} \|_{n}^{2} 
  & \le 
          6 \| \widehat{F}^{ (m) } - f^{*} \|_{n}^{2} 
        + 
          \frac{ ( 8 \overline{ \sigma }^{2} + C_{ \tau } )  m \log p }{n}
        + 
          \widehat{ \sigma }^{2} - \| \varepsilon \|_{n}^{2}, 
\end{align}
using that \( 4 ( m + 1 ) \le 8 m \).
Analogously, on \( \{ \tau > m \} \), we obtain
\begin{align}
          \| \widehat{F}^{ ( \tau ) } - f^{*} \|_{n}^{2} 
  & \le 
          7 \| \widehat{F}^{ (m) } - f^{*} \|_{n}^{2} 
        + 
          \frac{ 8 \overline{ \sigma }^{2} m \log p }{n} 
        + 
          \frac{ ( 8 \overline{ \sigma }^{2} - C_{ \tau } )  \tau \log p }{n} 
        + 
          \| \varepsilon \|_{n}^{2} - \widehat{ \sigma }^{2}, 
\end{align}
where we have used that \( b_{ \tau - 1 }^{2} \le b_{m}^{2} \). 
Combining the events and taking the infimum over \( m \ge 0 \) 
yields the result in Theorem \ref{thm_OracleInequalityForTheEmpiricalRisk}. 
We reiterate that here, it is the \( \omega \)-pointwise analysis that preserves the term \( 
  | \widehat{ \sigma }^{2} - \| \varepsilon \|_{n}^{2} |
\) in the result.


\subsection{Explicit bounds for empirical quantities}
\label{ssec_ExplicitBoundsForEmpiricalQuantities}

In order to derive a convergence rate from Theorem
\ref{thm_OracleInequalityForTheEmpiricalRisk}, we need explicit bounds for the
empirical quantities involved.
These will also be essential for the analysis of the stopped population risk.
We begin by establishing control over the most basic quantities.

\begin{lemma}[Uniform bounds in high probability]
  \label{lem_UniformBoundsInHighProbability}
  Under Assumptions
  \hyperref[ass_SubGaussianErrors]{\normalfont \textbf{{\color{blue} (SubGE)}}}
  and
  \hyperref[ass_SubGaussianDesign]{\normalfont \textbf{{\color{blue} (SubGD)}}},
  the following statements hold:
  \begin{enumerate}

    \item[(i)]
      There exists some $ C_{g} > 0 $ such that with probability converging to
      one,
      \begin{align*}
            \sup_{ j, k \le p } 
            | 
              \langle g_{j}, g_{k} \rangle_{n} 
            - \langle g_{j}, g_{k} \rangle_{ L^{2} } 
            | 
        \le 
            C_{g} 
            \sqrt{ \frac{ \rho^{4} \log p }{n} }. 
      \end{align*}

    \item[(ii)]
      There exists some $ C_{ \varepsilon } > 0 $ such that with probability
      converging to one,
      \begin{align*}
            \sup_{ j \le p } 
            | \langle \varepsilon, g_{j} \rangle_{n} | 
        \le
            C_{ \varepsilon } 
            \sqrt{ \frac{ \overline{ \sigma }^{2} \log p }{n} }. 
      \end{align*}

    \item[(iii)] 
      There exists some \( C_{ \Gamma } > 0 \) such that 
      for any fixed \( c_{ \text{Cov} } > 0 \), 
      \begin{align*}
            \sup_{ | J | \le c_{ \text{Cov} } n / \log p } 
            \frac{ \| \widehat{ \Gamma }_{J} - \Gamma_{J} \|_{ \text{op} } }
                 { \rho^{2} }
        \le
            c_{ \text{Cov} } C_{ \Gamma }
      \end{align*}
      with probability converging to one.

    \item[(iv)] 
      There exist \( c_{ \text{Cov} }, C_{ \Gamma^{-1} } > 0 \) such that with
      probability converging to one,
      \begin{align*}
            \sup_{ | J | \le c_{ \text{Cov} } n / \log p } 
            \| \widehat{ \Gamma }_{J}^{-1} \|_{ \text{op} }
        \le
            C_{ \Gamma^{-1} }.
      \end{align*}

  \end{enumerate}
\end{lemma}

\noindent Some version of this is needed in all results for \( L^{2} \)-boosting
in high-dimensional models, see Lemma 1 in Bühlmann
\cite{Buehlmann2006BoostingForHDLinearModels}, Lemma A.2 in Ing and Lai
\cite{IngLai2011ConsistentModelSelection} or Assumptions (A1) and (A2) in Ing
\cite{Ing2020ModelSelection}. 
A proof for our setting is  detailed in Appendix
\ref{sec_ProofsForAuxiliaryResults}.
Note that Lemma \ref{lem_UniformBoundsInHighProbability} (iii) and (iv) improve
the control to subsets \( J \) with cardinality of order up to \( n / \log p \)
from Lemma A.2 in Ing and Lai \cite{IngLai2011ConsistentModelSelection}, where
only subsets of order \( \sqrt{ n / \log p } \) could be handled.
For our results, we only need that \( 
  M_{ n } \le c_{ \text{Cov} } n / \log p
\) for \( n \) sufficiently large, however, this could open up further research
into the setting where \( \gamma \in ( 1 / 2, 1 ) \),
i.e., when \( m^{*}_{ s, \gamma } \) can be of order \( n / \log p \), see also
Barron et al. \cite{BarronEtal2008Greedy}. 

From Lemma \ref{lem_UniformBoundsInHighProbability}, we obtain that the
empirical stochastic error \( s_{m} \) satisfies a similar upper bound as its
population counterpart \( S_{m} \). 

\begin{lemma}[Bound for the empirical stochastic error]
  \label{lem_BoundForTheEmpiricalStochasticError}
  Under Assumptions 
  \hyperref[ass_SubGaussianErrors]{\normalfont \textbf{{\color{blue} (SubGE)}}}
  and
  \hyperref[ass_SubGaussianDesign]{\normalfont \textbf{{\color{blue} (SubGD)}}},
  the empirical stochastic error satisfies 
  \begin{align*}
    s_{m} & \le C \frac{ \overline{ \sigma }^{2} m \log p }{n} 
    \qquad \text{ for all } m \ge 0
  \end{align*}
  with probability converging to one.
\end{lemma}

\begin{proof}[Proof]
  For \( m \le c_{ \text{iter} } n / \log p \), using the notation from Equation 
  \eqref{eq_1_Projections}, we can write
  \begin{align}
          s_{m} 
    & = 
          \langle \varepsilon, \widehat{ \Pi }_{m} \varepsilon \rangle_{n} 
    = 
          \langle 
            \varepsilon, 
            g_{ \widehat{J}_{m} }^{ \top } 
            \widehat{ \Gamma }_{ \widehat{J}_{m} }^{-1} 
            \langle \varepsilon, g_{ \widehat{J}_{m} } \rangle_{n}
          \rangle_{n} 
    \le 
          \| 
            \widehat{ \Gamma }_{ \widehat{J}_{m} }^{-1} 
            \langle \varepsilon, g_{ \widehat{J}_{m} } \rangle_{n}
          \|_{1} 
          \sup_{ j \le p } 
          | \langle \varepsilon, g_{j} \rangle_{n} | 
    \\
    & \le 
          \sqrt{m} 
          \| 
            \widehat{ \Gamma }_{ \widehat{J}_{m} }^{-1} 
            \langle \varepsilon, g_{ \widehat{J}_{m} } \rangle_{n}
          \|_{2} 
          \sup_{ j \le p } 
          | \langle \varepsilon, g_{j} \rangle_{n} | 
    \le 
          \| \widehat{ \Gamma }_{ \widehat{J}_{m} }^{-1} \|_{ \text{op} } 
          m
          \sup_{ j \le p } 
          | \langle \varepsilon, g_{j} \rangle_{n} |^{2}. 
    \notag
  \end{align}
  This yields the result for \( m \le M_{n}^{2} \) by taking the supremum and
  applying the bounds from Lemma \ref{lem_UniformBoundsInHighProbability} 
  (ii) and (iv). 
  For \( m > c_{ \text{iter} } n / \log p \), the bound is also satisfied, since
  \( 
    s_{m} \le \| \varepsilon \|_{n}^{2} 
          \le C \text{Var}( \varepsilon_{1} ) 
          =   C \underline{ \sigma }^{2} 
          \le C \overline{  \sigma }^{2}
  \) with probability one.
\end{proof}

In order to relate the empirical bias to the population bias, we use a norm
change inequality from Ing \cite{Ing2020ModelSelection}, which we extend to the
\( s \)-sparse setting.
A complete derivation, which is based on the uniform Baxter inequality in
\eqref{eq_1_UniformBaxtersInequality}, is stated in Appendix
\ref{sec_ProofsForSupplementaryResults}. 

\begin{proposition}[Fast norm change for the bias]
  \label{prp_FastNormChangeForTheBias}
  Under Assumptions
  \hyperref[ass_Sparse]{\normalfont \textbf{{\color{blue} (Sparse)}}} and
  \hyperref[ass_CovB]{\normalfont \textbf{{\color{blue} (CovB)}}},
  the squared population bias \( 
    B_{m}^{2} = \| ( I - \Pi_{m} ) f^{*} \|_{ L^2 }^{2}
  \) satisfies the norm change inequality
  \begin{align*}
    & \ \ \ \
            | 
              \| ( I - \Pi_{m} ) f^{*} \|_{n}^{2} 
              - 
              \| ( I - \Pi_{m} ) f^{*} \|_{ L^2 }^{2} 
            | 
    \\
    & \le 
          C
          \begin{dcases}
            ( s + m )
            \| ( I - \Pi_{m} ) f^{*} \|_{ L^2 }^{2} 
            \sup_{ j, k \le p } 
            | 
              \langle g_{j}, g_{k} \rangle_{n} 
              - 
              \langle g_{j}, g_{k} \rangle_{ L^2 }
            |,
            & \beta^{*} \ s \text{-sparse}, \\
            \big( 
              \| ( I - \Pi_{m} ) f^{*} \|_{ L^2 }^{2} 
            \big)^{ \frac{ 2 \gamma - 2 }{ 2 \gamma - 1 } }
            \sup_{ j, k \le p } 
            | 
              \langle g_{j}, g_{k} \rangle_{n} 
              - 
              \langle g_{j}, g_{k} \rangle_{ L^2 }
            |,
            & \beta^{*} \ \gamma \text{-sparse} \\
          \end{dcases}
  \end{align*}
  for any \( m \le M_{n} \) and \( n \) large enough.
\end{proposition}

\noindent Proposition \ref{prp_FastNormChangeForTheBias} will appear again in
analyzing the stopping condition \eqref{eq_2_ExplicitStoppingCondition} in
Section \ref{sec_PopulationRisk}. 
Initially, it guarantees that the squared empirical bias \( b_{m}^{2} \)
satisfies the same bound as its population counterpart \( B_{m}^{2} \).

\begin{lemma}[Bound for the empirical bias]
  \label{lem_BoundForTheEmpiricalBias}
  Under assumptions 
  \hyperref[ass_SubGaussianErrors]{\normalfont \textbf{{\color{blue} (SubGE)}}},
  \hyperref[ass_Sparse]{\normalfont \textbf{{\color{blue} (Sparse)}}},
  \hyperref[ass_SubGaussianDesign]{\normalfont \textbf{{\color{blue} (SubGD)}}}
  and
  \hyperref[ass_CovB]{\normalfont \textbf{{\color{blue} (CovB)}}},
  the squared empirical bias satisfies
  \begin{align*}
            b_{m}^{2} 
    & \le
          C
          \begin{dcases}
              \| f^{*} \|_{ L^{2} }^{2}
              \exp \Big( \frac{ - c_{ \text{Bias} } m }{s} \Big) 
            + 
              \| \beta^{*} \|_{1}^{2}  
              \frac{ s \log p }{n}, 
            & \beta^{*} \ s \text{-sparse}, \\
              m^{ 1 - 2 \gamma } 
            + 
              \Big(
                \frac{ ( \overline{ \sigma }^{2} + \rho^{4} ) \log p }{n} 
              \Big)^{ 1 - 1 / ( 2 \gamma ) }, 
            & \beta^{*} \ \gamma \text{-sparse} \\
          \end{dcases}
    \\
    & \qquad \qquad \qquad \qquad \qquad \qquad \qquad \qquad 
      \qquad \qquad \text{ for all } m \ge 0 
  \end{align*}
  with probability converging to one.
\end{lemma}

\begin{proof}[Proof]
  For a fixed \( m \le M_{n} \) and \( n \) large enough,
  Proposition \ref{prp_FastNormChangeForTheBias} yields the estimate
  \begin{align}
    \label{eq_2_DerivationEmpiricalBiasBound}
          b_{m}^{2} 
    & = 
          \| ( I - \widehat{ \Pi }_{m} ) f^{*} \|_{n}^{2} 
    \le 
          \| ( I - \Pi_{m} ) f^{*} \|_{n}^{2} 
    \\
    & \le 
          \| ( I - \Pi_{m} ) f^{*} \|_{ L^{2} }^{2} 
    \notag 
    \\
    & + 
          C
          \begin{dcases}
            ( s + m ) 
            \| ( I - \Pi_{m} ) f^{*} \|_{ L^{2} }^{2}
            \sup_{ j, k \le p } 
            | 
              \langle g_{j}, g_{k} \rangle_{n} 
            - \langle g_{j}, g_{k} \rangle_{ L^2 }
            |, 
            & \beta^{*} \ s \text{-sparse}, \\
              \big(
                \| ( I - \Pi_{m} ) f^{*} \|_{ L^2 }^{2} 
              \big)^{ \frac{ 2 \gamma - 2 }{ 2 \gamma - 1 } } 
              \sup_{ j, k \le p } 
              | 
                \langle g_{j}, g_{k} \rangle_{n} 
              - \langle g_{j}, g_{k} \rangle_{ L^2 }
              |,
              & \beta^{*} \ \gamma \text{-sparse}. \\
          \end{dcases}
    \notag
  \end{align}
  Applying Lemma \ref{lem_UniformBoundsInHighProbability} (i)
  and Proposition \ref{prp_BoundForThePopulationBias} then yields
  the result for \( m \le M_{n} \). 
  The monotonicity of \( m \mapsto b_{m}^{2} \) implies that the
  claim is also true for any \( m > M_{n} \) under 
  \( \gamma \)-sparsity.
  Under \( s \)-sparsity, \( b_{m}^{2} = 0 \) for all \(
    m \ge C_{ \text{supp} } s 
  \)  with \( s = o( ( n / \log p )^{ 1 / 3 } ) \). 
  This finishes the proof.
\end{proof}

Analogous to  Equation \eqref{eq_1_RatesPopulationRisk}, Lemmas
\ref{lem_BoundForTheEmpiricalStochasticError} and
\ref{lem_BoundForTheEmpiricalBias} imply that at iteration 
\( m^{*}_{ s, \gamma } \) from Equation \eqref{eq_1_BalancedBoundOracleIndex},
the empirical risk satisfies the bound
\begin{align}
        \| \widehat{F}^{ ( m^{*}_{ s, \gamma } ) } - f^{*} \|_{n}^{2} 
  & \le 
        C \mathcal{R}( s, \gamma ) 
\end{align}
with probability converging to one.
This yields the result Corollary \ref{cor_RatesForTheEmpiricalRisk}.
For the empirical risk, we can also argue precisely that such a result cannot be
improved upon:

\begin{remark}[Optimality of the rates]
  \label{rem_OptimalityOfTheRates}
  For simplicity, we consider \( p = n \), a fixed, orthogonal
  (with respect to \( \langle \cdot, \cdot \rangle_{n} \) )
  design matrix \( \mathbf{X} \) and \(
    \varepsilon \sim N( 0, \sigma^{2} I_{n} ) 
  \).
  Conceptually, \( \rho^{ 2 } = 0 \) in this setting.
  When \( \beta^{*} \) is \( s \)-sparse, the squared empirical bias satisfies
  \begin{align}
        b_{m}^{2} 
    & = 
        \| ( I - \widehat{ \Pi }_{m} ) f^{*} \|_{n}^{2}
    \ge 
        \| \beta( \widehat{ \Pi }_{m} f^{*} ) - \beta^{*} \|_{2}^{2}  
    \ge 
        \underline{ \beta }^{2} 
  \end{align}
  for any \( m \le s \). 
  Similarly, when \( \beta^{*} \) is \( \gamma \)-sparse,
  \begin{align}
    \label{eq_2_MTermLowerBound}
        b_{m}^{2} 
    & = 
        \| ( I - \widehat{ \Pi }_{m} ) f^{*} \|_{n}^{2}
    \ge 
        \| \beta^{*}_{ m \text{-term} } - \beta^{*} \|_{2}^{2},  
  \end{align}
  where \( \beta^{*}_{ m \text{-term} } \) is the best \( m \)-term
  approximation of \( \beta^{*} \) with respect to the euclidean norm.
  For \( \beta^{*} \) with polynomial decay as in Equation
  \eqref{eq_1_PolyDecay}, the right-hand side in Equation
  \eqref{eq_2_MTermLowerBound} is larger than \(
    c m^{ 1 - 2 \gamma } 
  \), see e.g., Lemma A.3 in Ing \cite{Ing2020ModelSelection}.

  Conversely, for \( f^{*} = 0 \), the greedy procedure in Algorithm
  \ref{alg_OMP} guarantees that
  \begin{align}
        s_{m} 
    & = 
        \| \widehat{ \Pi }_{m} \varepsilon \|_{n}^{2} 
      = 
        \frac{1}{n} 
        \sum_{ j = 1 }^{m} 
        Z_{ ( n - j + 1 ) }^{2} 
    \ge 
        \frac{ m Z_{ ( p - m + 1 ) }^{2} }{n},
  \end{align}
  where \( Z_{ (j) } \) denotes the \( j \)-th order statistic of the \(
    Z_{j}: = \langle X^{ (j) }, \varepsilon \rangle_{n} 
  \), \( j \le p \), which are again independent, identically distributed
  Gaussians with variance \( \sigma^{2} \). 
  Noting that \(
    s = o( ( n / \log p )^{ 1 / 3 } ) 
  \), for both \( m = s \) and \(
    m = ( n / ( \sigma^{2} \log p ) )^{ 1 / ( 2 \gamma ) } 
  \), the order statistic \(
    Z_{ p - m + 1 } 
  \) is larger than \(
    c \sqrt{ \sigma^{2} \log p } 
  \) with probability converging to one, see Lemma
  \ref{lem_LowerBoundForOrderStatistics}.  
  Consequently, by distinguishing the cases where \( m \) is smaller or greater
  than \( s \)  under \( s \)-sparsity and the cases where \( m \) is smaller
  or greater than \( 
    ( n / ( \sigma^{2} \log p ) )^{ 1 / ( 2 \gamma ) }
  \) under \( \gamma \)-sparsity, we obtain
  \begin{align}
    \inf_{ m \ge 0 } 
    \sup_{ f^{*} } 
    \| \widehat{F}^{ (m) } - f^{*} \|_{n}^{2} 
    \ge 
    c
    \begin{dcases}
      \frac{ \sigma^{2} s \log p }{ n }, 
      & \beta^{*} \ s \text{-sparse}, \\
      \Big( \frac{ \sigma^{2} \log p }{ n } \Big)^{ 1 - \frac{1}{ 2 \gamma } }, 
      & \beta^{*} \ \gamma \text{-sparse}
    \end{dcases}
  \end{align}
  with probability converging to one, where the infimum is taken over 
  either all \( f^{*} \) satisfying 
  \hyperref[ass_Sparse]{\normalfont \textbf{{\color{blue} (Sparse)}}} (i)
  or over all \( f^{*} \) satisfying
  \hyperref[ass_Sparse]{\normalfont \textbf{{\color{blue} (Sparse)}}} (ii).
\end{remark}



\section{Population risk analysis}
\label{sec_PopulationRisk}

In this section, we analyze the stopped population risk \( 
  \| \widehat{F}^{ ( \tau ) } - f^{*} \|_{ L^2 }^{2}
\) with \( \tau \) from Equation \eqref{eq_1_SequentialEarlyStoppingTimeKappaM}. 
Unlike the empirical risk, the population risk cannot be expressed in terms of
the residuals straight away.
Instead, we examine the stopping condition \( r_{m}^{2} \le \kappa_{m} \), i.e., 
\begin{align}
  \label{eq_3_ExplicitStoppingCondition}
            b_{m}^{2} + 2 c_{m} 
  & \le 
            \widehat{ \sigma }^{2} - \| \varepsilon \|_{n}^{2} 
          + 
            \frac{ C_{ \tau } m \log p }{n} + s_{m}.
\end{align}
We show separately that for a suitable choice of \( \widehat{ \sigma }^{2} \),
condition \eqref{eq_3_ExplicitStoppingCondition} guarantees that \( \tau \)
stops neither too early nor too late.
In combination, this yields Theorem \ref{thm_OptimalAdaptationForThePopulationRisk}.
For the analysis, it becomes essential that we have access to the fast norm
change for the population residual term \(
  ( I - \Pi_{m} ) f^{*} 
\) from Proposition \ref{prp_FastNormChangeForTheBias}, which guarantees that
empirical and population norm remain of the same size until the squared
population bias reaches the optimal rate \( \mathcal{R}( s, \gamma ) \). 
This control is not already readily available by standard tools, e.g.,
Wainwright \cite{Wainwright2019HDStatistics}.

\subsection{No stopping too early}
\label{ssec_NoStoppingTooEarly}

The sequential procedure stops too early if the squared population bias 
\( 
  B_{m}^{2} = \| ( I - \Pi_{m} ) f^{*} \|_{ L^2 }^{2} 
\)
has not reached the optimal rate of convergence yet, i.e., 
\( \tau < \tilde m_{ s, \gamma, G } \), where
\begin{align}
  \label{eq_3_TildeMG}
      \tilde m_{ s, \gamma, G}:
  = 
      \begin{dcases}
        \inf \{ m \ge 0: S \subset \widehat{J}_{m} \},
        & \beta^{*} \ s\text{-sparse}, \\
      \inf \Big\{ 
        m \ge 0:     \| ( I - \Pi_{m} ) f^{*} \|_{ L^2 }^{2} 
                 \le 
                     G 
                     \Big( 
                       \frac{ ( \overline{ \sigma }^{2} + \rho^{4} ) \log p }{n}
                     \Big)^{ 1 - \frac{1}{ 2 \gamma } }
      \Big\},
        & \beta^{*} \ \gamma \text{-sparse}
      \end{dcases}
\end{align}
for any constant \( G > 0 \).
Note that under 
\hyperref[ass_CovB]{\normalfont \textbf{{\color{blue} (CovB)}}},
for \( s \)-sparse \( \beta^{*} \), 
\begin{align}
  B_{m}^{2} 
  & = 
  \| ( I - \Pi_{m} ) f^{*} \|_{ L^{2} }^{2} 
  \begin{cases}
    \ge c_{ \lambda } \| \tilde \beta \|_{2}^{2}
    \ge c_{ \lambda } \underline{ \beta }^{2},
    & m < \tilde m_{ s, \gamma, G }, \\ 
    = 0, & m \ge \tilde m_{ s, \gamma, G }.
  \end{cases}
\end{align}
Therefore, a condition for stopping too early is given by
\begin{align}
  \label{eq_3_StoppingTooEarly}
  \exists m < \tilde m_{ s, \gamma, G }: 
          b_{m}^{2} + 2 c_{m} 
  & \le 
          \widehat{ \sigma }^{2} - \| \varepsilon \|_{n}^{2} 
        + 
          \frac{ C_{ \tau } m \log p }{n} + s_{m},
\end{align}
where we may vary \( G > 0 \).

Using the norm change inequality for the bias in Proposition
\ref{prp_FastNormChangeForTheBias}, we can derive that the left-hand side  of
condition \eqref{eq_3_StoppingTooEarly} is of the same order as \( B_{m}^{2} \),
i.e., 
\begin{align}
          b_{m}^{2} + 2 c_{m} 
  & \ge 
          \begin{dcases}
            \frac{ c_{ \lambda } }{8} \underline{ \beta }^{2},
            & \beta^{*} \ s \text{-sparse}, \\
            \frac{G}{8} 
            \Big( 
              \frac{ ( \overline{ \sigma }^{2} + \rho^{4} ) \log p }{n}
            \Big)^{ 1 - \frac{1}{ 2 \gamma } },
            & \beta^{*} \ \gamma \text{-sparse}
          \end{dcases}
\end{align}
with probability converging to one.
At the same time, Proposition \ref{prp_BoundForThePopulationBias}
guarantees that 
\( 
  \tilde m_{ s, \gamma, G} \le m^{*}_{ s, \gamma } 
\) 
from Equation \eqref{eq_1_BalancedBoundOracleIndex} with probability converging
to one for \( G \) large enough.
Therefore, if \( \widehat{ \sigma }^{2} \) does not substantially overestimate
the empirical noise level and \( C_{ \tau } \) is chosen proportional to 
\( ( \overline{ \sigma }^{2} + \rho^{4} ) \), 
Lemma \ref{lem_BoundForTheEmpiricalStochasticError}
implies that the right-hand side of condition \eqref{eq_3_StoppingTooEarly}
satisfies
\begin{align}
      \widehat{ \sigma }^{2} - \| \varepsilon \|_{n}^{2} 
      + 
      \frac{ C_{ \tau } m \log p }{n} 
      + 
      s_{m}
  \le 
      C \mathcal{R}( s, y ) 
\end{align}
with probability converging to one
for a constant \( C > 0 \) independent of \( G \). 
For \( G \) large enough, condition \eqref{eq_3_StoppingTooEarly} can therefore
only be satisfied on an event with probability converging to zero.
Together, this yields the following result:

\begin{proposition}[No stopping too early]
  \label{prp_NoStoppingTooEarly}
  Under Assumptions 
  \hyperref[ass_SubGaussianErrors]{\normalfont \textbf{{\color{blue} (SubGE)}}},
  \hyperref[ass_Sparse]{\normalfont \textbf{{\color{blue} (Sparse)}}}, \\
  \hyperref[ass_SubGaussianDesign]{\normalfont \textbf{{\color{blue} (SubGD)}}}
  and
  \hyperref[ass_CovB]{\normalfont \textbf{{\color{blue} (CovB)}}},
  choose \( \widehat{ \sigma }^{2} \) in Equation
  \eqref{eq_1_SequentialEarlyStoppingTimeKappaM} such that 
  \begin{align*}
        \widehat{ \sigma }^{2}  
    \le 
        \| \varepsilon \|_{n}^{2} + C \mathcal{R}( s, \gamma ) 
  \end{align*}
  with probability converging to one.
  Then, for \( G > 0 \) large enough and any choice 
  \(
    C_{ \tau } = c ( \overline{ \sigma }^{2} + \rho^{4} ) 
  \) 
  in \eqref{eq_1_SequentialEarlyStoppingTimeKappaM}
  with \( c \ge 0 \), the sequential stopping time satisfies 
  \( \tilde m_{ s, \gamma, G} \le \tau < \infty \), with probability converging to one.
  On the corresponding event, it holds that
  \begin{align*}
    B_{ \tau }^{2} 
    = 
    \| ( I - \Pi_{ \tau } ) f^{*} \|_{ L^2 }^{2} 
    \le 
        \begin{dcases}
          0,
          & \beta^{*} \ s \text{-sparse}, \\
          G 
          \Big( 
            \frac{ ( \overline{ \sigma }^{2} + \rho^{4} ) \log p }{n}
          \Big)^{ 1 - \frac{1}{ 2 \gamma } }, 
          & \beta^{*} \ \gamma \text{-sparse}.
        \end{dcases}
  \end{align*}
\end{proposition}

\noindent The technical details of the proof are provided in Appendix
\ref{sec_ProofsForTheMainResults}. 
Proposition \ref{prp_NoStoppingTooEarly} guarantees that \( \tau \)
controls the population bias on an event with probability converging to one.
It is noteworthy that in order to do so, it is only required that 
\( \widehat{ \sigma }^{2} \) is smaller than the empirical noise level up to
a lower order term and also the choice \( C_{ \tau } = 0 \) in 
Equation \eqref{eq_1_SequentialEarlyStoppingTimeKappaM} is allowed.
We will further discuss this in Section
\ref{sec_NumericalSimulationsAndATwoStepProcedure}. 


\subsection{No stopping too late}
\label{ssec_NoStoppingTooLate}

The sequential procedure potentially stops too late when the bound in Lemma
\ref{lem_BoundForThePopulationStochasticError} no longer guarantees that the
population stochastic error \( S_{ \tau } \) is of optimal order, i.e., when
there is no constant \( H > 0 \) such that \( \tau \) can be bounded by 
\( H m^{*}_{ s, \gamma } \) on a large event for \( m^{*}_{ s, \gamma } \) from
Equation \eqref{eq_1_BalancedBoundOracleIndex}.
For stopping too late, we therefore consider the condition 
\( r_{m}^{2} > \kappa_{m} \), i.e.,
\begin{align}
  \label{eq_3_ConditionForLateStopping}
      b_{m}^{2} + 2 c_{m} 
  > 
      \widehat{ \sigma }^{2} - \| \varepsilon \|_{n}^{2}
    + 
      \frac{ C_{ \tau } m \log p }{n} 
    + 
      s_{m} 
    \qquad \text{ for } \qquad 
  m = H m^{*}_{ s, \gamma }. 
\end{align}
For \( s \)-sparse \( \beta^{*} \) and \( H > 1 \), the left-hand side vanishes
with probability converging to one due to Proposition
\ref{prp_BoundForThePopulationBias}.
For \( \gamma \)-sparse \( \beta^{*} \), the results in Lemma
\ref{lem_BoundForTheEmpiricalBias} and Lemma
\ref{lem_BoundsForTheCrossTermAndTheDiscretizationError} (i) yield that the
left-hand side of condition \eqref{eq_3_ConditionForLateStopping} is at most of
order \( 
  \sqrt{H} (
               ( \overline{ \sigma }^{2} + \rho^{4} ) \log(p) /  n 
             )^{ 1 - 1 / ( 2 \gamma ) }
\)
on an event with probability converging to one.
At the same time, for \( \widehat{ \sigma }^{2} \) large enough
and a choice \( C_{ \tau } = c ( \overline{ \sigma }^{2} + \rho^{4} ) \) with 
\( c > 0 \), the right-hand
side is at least of order  \( H \mathcal{R}( s, \gamma ) \)
also on an event with probability converging to one.
Note that this requires a choice \( c > 0 \), since 
Lemma \ref{lem_BoundForTheEmpiricalStochasticError} only provides an upper bound
for \( s_{m} \).
For \( H > 0 \) sufficiently large, this yields that condition
\eqref{eq_3_ConditionForLateStopping} can only be satisfied on an event with
probability converging to zero.  We obtain the following result:

\begin{proposition}[No stopping too late]
  \label{prp_NoStoppingTooLate}
  Under Assumptions 
  \hyperref[ass_SubGaussianErrors]{\normalfont \textbf{{\color{blue} (SubGE)}}},
  \hyperref[ass_Sparse]{\normalfont \textbf{{\color{blue} (Sparse)}}},
  \hyperref[ass_SubGaussianDesign]{\normalfont \textbf{{\color{blue} (SubGD)}}}
  and
  \hyperref[ass_CovB]{\normalfont \textbf{{\color{blue} (CovB)}}},
  choose \( \widehat{ \sigma }^{2} \) in Equation
  \eqref{eq_1_SequentialEarlyStoppingTimeKappaM}  such that 
  \begin{align*}
        \widehat{ \sigma }^{2}  
    \ge 
        \| \varepsilon \|_{n}^{2} - C \mathcal{R}( s, \gamma ) 
  \end{align*}
  with probability converging to one.
  Then, for any choice 
  \(
    C_{ \tau } = c ( \overline{ \sigma }^{2} + \rho^{4} ) 
  \) 
  in \eqref{eq_1_SequentialEarlyStoppingTimeKappaM}
  with \( c > 0 \), the sequential stopping time satisfies
  \( 
    \tau \le H m^{*}_{ s, \gamma }
  \) 
  with probability converging to one for some \( H > 0 \) large enough.
  On the corresponding event, it holds that
  \begin{align*}
        S_{ \tau } 
    & = 
        \| \widehat{F}^{ ( \tau ) } - \Pi_{ \tau } f^{*} \|_{ L^{2} }^{2} 
    \le 
        C H \mathcal{R}( s, \gamma ).
  \end{align*}
\end{proposition}

\noindent The details of the proof can be found in Appendix
\ref{sec_ProofsForTheMainResults}. 
Proposition \ref{prp_NoStoppingTooLate} complements Proposition
\ref{prp_NoStoppingTooEarly} in that it guarantees that \( \tau \) controls the
stochastic error on an event with probability converging to one. 
Together, the two results imply Theorem
\ref{thm_OptimalAdaptationForThePopulationRisk}. 
As in Theorem \ref{thm_OracleInequalityForTheEmpiricalRisk}, it is the 
\( \omega \)-pointwise analysis of the stopping condition preserves the term \( 
  | \widehat{ \sigma }^{2} - \| \varepsilon \|_{n}^{2} |
\) 
in the condition of the result.



\section{Estimation of the empirical noise level}
\label{sec_EstimationOfTheEmpiricalNoiseLevel}

For any real application, the results in Theorem
\ref{thm_OracleInequalityForTheEmpiricalRisk} and Theorem
\ref{thm_OptimalAdaptationForThePopulationRisk} require access to a suitable
estimator \( \widehat{ \sigma }^{2} \) of the empirical noise level 
\( \| \varepsilon \|_{n}^{2} \).
In this section, we demonstrate that under reasonable assumptions, such
estimators do in fact exist.
In particular, we analyze the \emph{Scaled Lasso} noise estimate
\( \widehat{ \sigma }^{2} \) from Sun and Zhang \cite{SunZhang2012ScaledLasso}
in our setting.
It is noteworthy that our estimation target is the empirical noise level 
\( \| \varepsilon \|_{n}^{2} \) rather than its almost sure limit
\( \text{Var}( \varepsilon_{1} ) \).

\begin{remark}[Estimating \( \| \varepsilon \|_{ n }^{2} \) vs. \( \text{Var}( \varepsilon_{1} ) \)]
  \label{rem_EstimatingEmpiricalNoiseVSPopulationNoise}
  In general, it is easier to estimate \( \| \varepsilon \|_{n}^{2} \) than 
  \( \text{Var}( \varepsilon_{1} ) \).
  We illustrate this fact in the simple location-scale model
  \begin{align}
    Y_{i} = \mu + \varepsilon_{i}, \qquad i = 1, \dots, n, 
  \end{align}
  where \( \varepsilon_{i} \sim N( 0, \sigma^{2} ) \) i.i.d.
  Simple calculations yield that the standard noise estimator \( 
      \widehat{ \sigma }^{2}:
    =
      \| Y - \widehat{ \mu } ( 1, \dots, 1 )^{ \top } \|_{n}^{2}
  \) with \(
    \widehat{ \mu } = n^{-1} \sum_{ i = 1 }^{n} Y_{i} 
  \) satisfies 
  \begin{align}
          \mathbb{E}_{ \sigma^{2} }
          | \widehat{ \sigma }^{2} - \| \varepsilon \|_{n}^{2} | 
    \le
          C \frac{ \sigma^{2} }{n}
    \qquad \text{ for all } \sigma^{2} > 0, 
  \end{align}
  where the subscript \( \sigma^{2} \) denotes the expectation with respect to
  \( N( 0, \sigma^{2} ) \).
  Conversely, for \( \text{Var}( \varepsilon_{1} ) = \sigma^{2} \), a
  convergence rate of \( n^{-1} \) can only be reached for the squared risk.
  Indeed, it follows from an application of van-Trees's inequality, see Gill and
  Levit \cite{GillLevit1995vanTrees}, that for any \( \delta^{2} > 0 \),
  \begin{align}
            \inf_{ \widehat{ \sigma }^{2} } 
            \sup_{ \sigma^{2} > \delta^{2} } 
            \mathbb{E}_{ \sigma^{2} } ( \widehat{ \sigma }^{2} - \sigma^{2} )^{2}
    & \ge 
            c \frac{ \delta^{4} }{n} 
  \end{align}
  for \( n \in \mathbb{N} \) sufficiently large, where the infimum is taken over
  all measurable functions \( \widehat{ \sigma }^{2} \).
  This indicates that for the absolute risk we cannot expect a rate faster than
  \( n^{ - 1 / 2 } \).
\end{remark}

\noindent The fact that in general, \( \| \varepsilon \|_{n}^{2} \) can be
estimated with a faster rate than \( \text{Var}( \varepsilon_{1} ) \) together
with the \( \omega \)-pointwise analysis is essential in circumventing a lower
bound restriction as in Corollary 2.5 of Blanchard et al.
\cite{BlanchardEtal2018a}.

We briefly recall the approach in Sun and Zhang \cite{SunZhang2012ScaledLasso}. 
The authors consider the joint minimizer \(
  ( \widehat{ \beta }, \widehat{ \sigma } ) 
\) of the \emph{Scaled Lasso}
objective 
\begin{align}
  \label{eq_4_ScaledLassoObjective}
      L_{ \lambda_{0} }( \beta, \sigma ): 
  = 
      \frac{ \| Y - X \beta \|_{2}^{2} }{ 2 n \sigma } 
    + 
      \frac{ \sigma }{2} 
    + 
      \lambda_{0} \| \beta \|_{1}, 
  \qquad \beta \in \mathbb{R}^{p}, \sigma > 0,
\end{align}
where \( \lambda_{0} \) is a penalty parameter chosen by the user.
Since \( L_{ \lambda_{0} } \) is jointly convex in \( ( \beta, \sigma ) \), the minimizer can be implemented efficiently.
For \( \lambda > 0 \) and \( \xi > 1 \), they set
\begin{align}
      \mu( \lambda, \xi ):
  = 
      ( \xi + 1 )
      \min_{ J \subset \{ 1, \dots, p \} } 
      \inf_{ \nu \in ( 0, 1 ) } 
      \max \Big( 
        \frac{ \| \beta^{*}_{ J^{c} } \|_{1} }{ \nu },
        \frac{ \lambda | J | / ( 2 ( 1 - \nu ) ) }
             { \kappa^{2}( ( \xi + \nu ) /  ( 1 - \nu ), J ) } 
      \Big), 
\end{align}
with the \emph{compatibility factor} 
\begin{align}
  \label{eq_4_CompatibilityFactor}
  \kappa^{2}( \xi, J ): 
  = 
  \min \Big\{ 
    \frac{ \| X \Delta \|_{2}^{2} | J | }
         { n \| \Delta_{J} \|_{1}^{2} }: 
    \| \Delta_{ J^{c}  } \|_{1} \le \xi \| \Delta_{J} \|_{1}
  \Big\},
\end{align}
see Bühlmann and van de Geer \cite{BuehlmannVDGeer2011HDData}. 
In Theorem 2 of \cite{SunZhang2012ScaledLasso}, Sun and Zhang then show that
\begin{align}
        \max \Big( 
          1 - \frac{ \widehat{ \sigma } }{ \| \varepsilon \|_{n} }, 
          1 - \frac{ \| \varepsilon \|_{n} }{ \widehat{ \sigma } }
        \Big) 
  & \le 
        \alpha^{*}: =  \frac{ 
                              \lambda_{0} 
                              \mu( \| \varepsilon \|_{n} \lambda_{0}, \xi)
                            }
                            { \| \varepsilon \|_{n} } 
\end{align}
on the event
\begin{align}
      \Omega_{ \text{Noise} }: 
  = 
      \Big\{ 
      \sup_{ j \le p } \langle g_{j}, \varepsilon \rangle_{n} 
  \le 
      ( 1 - \alpha^{*} ) 
      \frac{ \xi - 1 }{ \xi + 1 } 
      \| \varepsilon \|_{n} \lambda_{0}
      \Big\}.
\end{align}

In our setting, due to Lemma \ref{lem_UniformBoundsInHighProbability} (i), 
\( \Omega_{ \text{Noise} } \) is an event with probability converging to one
when \( \lambda_{0} \) is of order \( \sqrt{ \log(p) / n } \).
Further, it can be shown that in this case,
\begin{align}
        \mu( \| \varepsilon \|_{n} \lambda_{0}, \xi ) 
  & \le 
        C
        \begin{dcases}
          s \sqrt{ \frac{ \| \varepsilon \|_{n}^{2} \log p }{n} },
          & \beta^{*} \ s \text{-sparse}, \\
          \Big( 
            \frac{ \| \varepsilon \|_{n}^{2} \log p }{n}
          \Big)^{ \frac{1}{2} - \frac{1}{ 2 \gamma } },
          & \beta^{*} \ \gamma \text{-sparse}, \\
        \end{dcases}
\end{align}
as long as the compatibility factor from Equation
\eqref{eq_4_CompatibilityFactor}  is strictly positive. 
Usually this can be guaranteed on an event with high probability:
When the rows of design matrix \( \mathbf{X} \) are given by \( 
  ( X_{i} )_{ i \le n } \sim N( 0, \Gamma ) 
\) i.i.d., e.g., Theorem 7.16 in Wainwright \cite{Wainwright2019HDStatistics}
guarantees that the compatibility factor satisfies 
\begin{align}
  \kappa^{2}( \xi, J )
  & \ge 
  \frac{ \lambda_{ \min }( \Gamma ) }{ 16 } 
  \qquad \text{whenever }
  | J | \le 
            \frac{ ( 1 + \xi )^{ - 2 } }{ 800 } 
            \frac{ \lambda_{ \min }( \Gamma )  }{ \max_{ j \le p } \Gamma_{ j j } } 
            \frac{n}{ \log p } 
\end{align}
with probability larger than 
\( 
  1 - e^{ - n / 32 } / ( 1 - e^{ - n / 32 } ) 
\).
The combination of these results allow to formulate Proposition
\ref{prp_FastNoiseEstimation}, which is applicable to the setting in Corollary
\ref{cor_RatesForTheEmpiricalRisk}  and Theorem
\ref{thm_OptimalAdaptationForThePopulationRisk}. 
The proof is given Appendix \ref{sec_ProofsForTheMainResults}. 

\begin{proposition}[Fast noise estimation]
  \label{prp_FastNoiseEstimation}
  Under Assumptions 
  \hyperref[ass_SubGaussianErrors]{\normalfont \textbf{{\color{blue} (SubGE)}}},
  \hyperref[ass_Sparse]{\normalfont \textbf{{\color{blue} (Sparse)}}}
  and 
  \hyperref[ass_CovB]{\normalfont \textbf{{\color{blue} (CovB)}}} with 
  Gaussian design 
  \( 
    ( X_{i} )_{ i \le n } \sim N( 0, \Gamma )
  \) 
  i.i.d., 
  set \( \xi > 1 \) and
  \(
    \lambda_{0} 
    =
    C_{ \lambda_{0} } ( \xi + 1 ) / ( \xi - 1 ) \sqrt{ \log(p) / n } 
  \)
  with 
  \(
    C_{ \lambda_{0} } 
    \ge
    2 C_{ \varepsilon } \overline{ \sigma } / \underline{ \sigma } 
  \).
  Then, with probability converging to one, the Scaled lasso noise estimator 
  \( \widehat{ \sigma }^{2} \) satisfies
  \begin{align*}
          | \widehat{ \sigma }^{2} - \| \varepsilon \|_{n}^{2} | 
    & \le 
          \frac{ 2 \| \varepsilon \|_{n}^{2} \alpha^{*} }
               { ( 1 - \alpha^{*} )^{2} }
    \qquad \text{ with } \qquad 
          \alpha^{*}
    =
          \frac{ 
                 \lambda_{0} 
                 \mu( \| \varepsilon \|_{n} \lambda_{0}, \xi )
               }
               { \| \varepsilon \|_{n} }.
  \end{align*}
  In particular, for any fixed choice \( \xi > 1 \), this implies that 
  with probability converging to one,
  \begin{align*}
            | \widehat{ \sigma }^{2} - \| \varepsilon \|_{n}^{2} | 
    & \le 
            C
            \begin{dcases}
              \frac{ \overline{ \sigma }^{2} s \log p }{n},
              & \beta^{*} \ s \text{-sparse}, \\
              \Big( 
                \frac{ \overline{ \sigma }^{2} \log p }{n}
              \Big)^{ 1 - 1 / ( 2 \gamma ) },
              & \beta^{*} \ s \text{-sparse}.
            \end{dcases}
  \end{align*}
\end{proposition}

\noindent The rates in Proposition \ref{prp_FastNoiseEstimation} match
the rates in Corollary \ref{cor_RatesForTheEmpiricalRisk} and Theorem
\ref{thm_OptimalAdaptationForThePopulationRisk}.
Under the respective assumptions, the combination of the stopping time \( \tau
\) from Equation \eqref{eq_1_SequentialEarlyStoppingTimeKappaM} with the
estimator \( \widehat{ \sigma }^{2} \) therefore provides a fully data-driven
sequential procedure which guarantees optimal adaptation to the unknown
sparsity parameters \( s \in \mathbb{N}_{0} \) or \( \gamma \in [ 1, \infty )
\).


\section{Numerical simulations and a two-step procedure}
\label{sec_NumericalSimulationsAndATwoStepProcedure}

In this section, we illustrate our main results by numerical simulations.
Here, we focus on the noise estimation aspects of our results in the regression
setting with uncorrelated design. 
A more extensive simulation study, including correlated design and the
classification setting from Example \ref{expl_RegressionAndClassification} (b)
with heteroscedastic error terms is provided in Appendix
\ref{sec_SimulationStudy}. 
The simulations confirm our theoretical results but also reveal some
shortcomings stemming from the sensitivity of our method with regard to the
noise estimation.
We address this by proposing a two-step procedure that combines early stopping
with an additional model selection step.

\begin{minipage}{0.49\textwidth}
  \centering
  \begin{figure}[H]
    \includegraphics[width=0.99\linewidth]{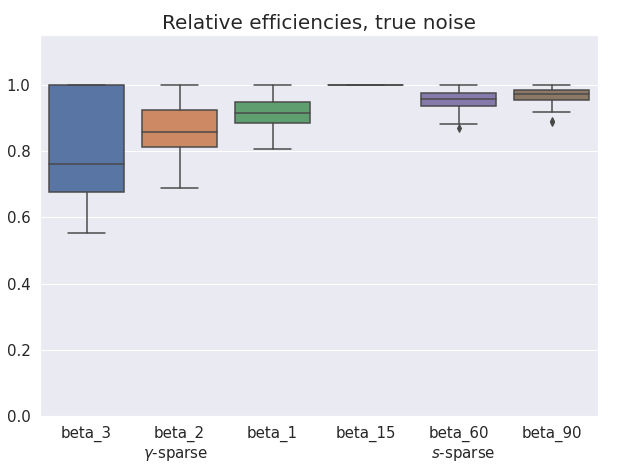}
    \\[-2.0ex]
    \caption{Relative efficiencies for early stopping with the true empirical
             noise level.}
    \label{fig_Uncorrelated_RelativeEfficiencies_TrueNoise}
  \end{figure}
\end{minipage}
\begin{minipage}{0.49\textwidth}
  \centering
  \begin{figure}[H]
    \includegraphics[width=0.99\linewidth]{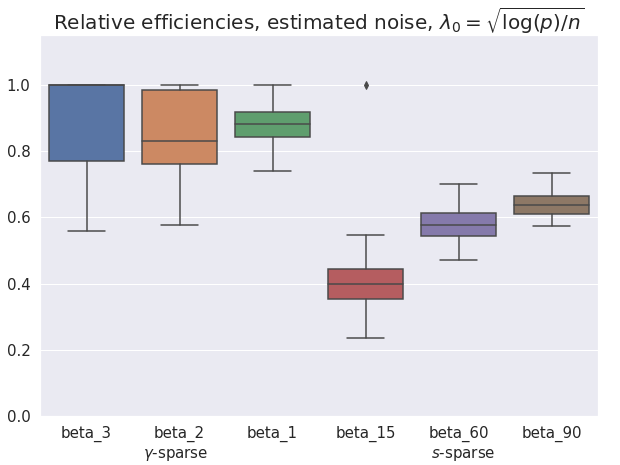}
    \\[-2.0ex]
    \caption{Relative efficiencies for early stopping with the estimated empirical
             noise level for \( \lambda_{0} = \sqrt{ \log(p) / n } \).}
    \label{fig_Uncorrelated_RelativeEfficiencies_lambda0_1}
  \end{figure}
\end{minipage}

\subsection{Numerical simulations for the main results}
\label{ssec_NumericalSimulationsForTheMainResults}

All of our simulations are based on 100 Monte-Carlo runs of a model in which
both sample size \( n \)  and parameter size \( p \) are equal to 1000.
We examine signals \( f \) with coefficients \( \beta \) corresponding to the
two sparsity concepts in Assumption \hyperref[ass_Sparse]{\normalfont
\textbf{{\color{blue} (Sparse)}}}.
We consider the \( s \)-sparse signals
\begin{align}
  \beta^{ (15) }_{j} & = \mathbf{1} \{  1 \le j \le 5 \}  + 0.5 \cdot \mathbf{1} \{  6 \le j \le 10 \} + 0.25 \cdot \mathbf{1} \{ 11 \le j \le 15 \},         \\
  \beta^{ (60) }_{j} & = \mathbf{1} \{  1 \le j \le 20 \} + 0.5 \cdot \mathbf{1} \{ 21 \le j \le 40 \} + 0.25 \cdot \mathbf{1} \{ 41 \le j \le 60 \}, \notag  \\ 
  \beta^{ (90) }_{j} & = \mathbf{1} \{  1 \le j \le 30 \} + 0.5 \cdot \mathbf{1} \{ 31 \le j \le 60 \} + 0.25 \cdot \mathbf{1} \{ 61 \le j \le 90 \} \notag
\end{align}
for \(
  s \in \{ 15, 60, 90 \} 
\) and the \( \gamma \)-sparse signals
\begin{align}
  \beta^{ (3) }_{j}: = j^{ - 3 }, \quad \beta^{ (2) }_{j}: = j^{ - 2 }, \quad \beta^{ (1) }_{j}: = j^{ - 1 }, \qquad j \le p
\end{align}
for \( \gamma \in \{ 3, 2, 1 \} \). 
Note that the definition of Algorithm \ref{alg_OMP} allows to consider
decreasingly ordered coefficients without loss of generality.
In a second step, we normalize all signals to the same \( \ell^{1} \)-norm of
value 10.
Since the Scaled Lasso penalizes the \( \ell^{1} \)-norm, this is necessary to
make the noise estimations comparable between simulations.
For both the covariance structure of the design and the noise terms \(
\varepsilon \), we consider independent standard normal variables.
For the early stopping time \( \tau \) in Equation
\eqref{eq_1_SequentialEarlyStoppingTimeKappaM}, we focus on the noise level
estimate \( \widehat{ \sigma }^{2} \). 
For our theoretical results, \( C_{ \tau } > 0 \) was needed to control the
discretization error \( \Delta( r_{ \tau }^{2} ) \) of the residual norm and to
counteract the fact that Lemma \ref{lem_BoundForTheEmpiricalStochasticError}
does not provide a lower bound of the same size.
Since empirically, both of these aspect do not pose any problems, it seems
warranted to set \( C_{ \tau } = 0 \) and exclude this hyperparameter from our
simulation study.
The simulation in Figure \ref{fig_MethodComparison} of Section
\ref{sec_Introduction} is based on \( \beta^{ (2) } \).
The estimated noise result used a penalty \(
  \lambda_{0} = \sqrt{ \log(p) / n } 
\) for the Scaled Lasso.
The two-step procedure used \(
  \lambda_{0} = \sqrt{ 0.5 \log(p) / n }
\) and \( C_{ \text{AIC} } = 2 \). 
The HDAIC-procedure from Ing \cite{Ing2020ModelSelection} was computed with 
\( C_{ \text{HDAIC} } = 2 \), see also the discussion in Section
\ref{ssec_AnImprovedTwoStepProcedure}.

As a baseline for the potential performance of sequential early stopping, we consider the setting in which we have access to the true empirical noise level and set \( \widehat{ \sigma }^{2} = \| \varepsilon \|_{n}^{2} \). 
As a performance metric for a simulation run, we consider the \emph{relative efficiency}
\begin{align}
  \min_{ m \ge 0 } \| \widehat{F}^{ (m) } - f^{*} \|_{n} / \| \widehat{F}^{ ( \tau ) } - f^{*} \|_{n}, 
\end{align}
\begin{figure}[H]
  \centering
  \includegraphics[width=0.497\textwidth]{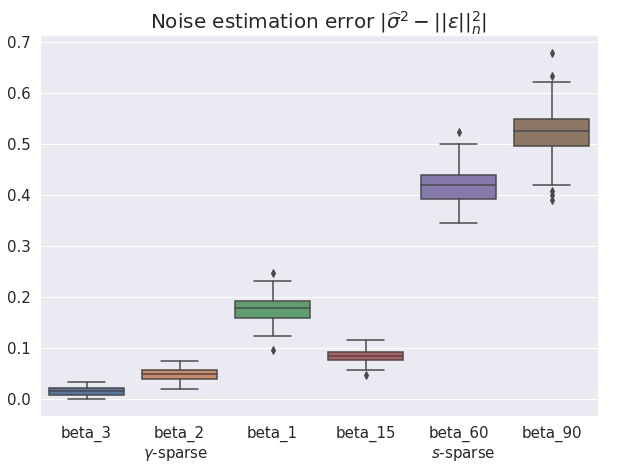}
  \includegraphics[width=0.497\textwidth]{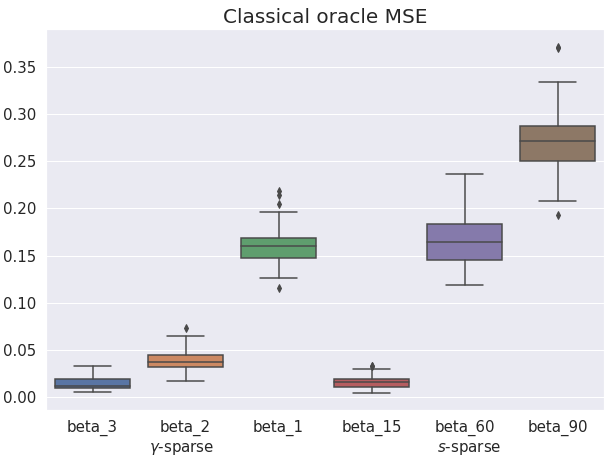}
  \\[-2.0ex]
  \caption{
    Boxplots for the estimation errors for the noise estimation with 
    \( \lambda_{0} = \sqrt{ \log(p) / n } \) together with the classical oracle
    risk.
  }
  \label{fig_Uncorrelated_NoiseEstimation} 
\end{figure}

\noindent which can be interpreted as a proxy for the constant 
\( C_{ \text{Risk} }^{ - 1 / 2 } \) in Theorem
\ref{thm_OptimalAdaptationForThePopulationRisk}. 
We choose this quantity rather than its inverse because it makes for clearer plots.
Values bounded away from zero indicate optimal adaptation up to a constant.
Values closer to one indicate better estimation overall.
We report boxplots of the \emph{relative efficiencies} in Figure
\ref{fig_Uncorrelated_RelativeEfficiencies_TrueNoise}.
The values are clearly bounded away from zero and close to one,
indicating that with access to the true empirical noise level \( \| \varepsilon
\|_{n}^{2} \), the sequential early stopping procedure achieves optimal
adaptation simultaneously over different sparsity levels for both sparsity
concepts from Assumption \hyperref[ass_Sparse]{\normalfont
\textbf{{\color{blue} (Sparse)}}}.
This is expected, from the results in Theorem
\ref{thm_OracleInequalityForTheEmpiricalRisk}, Corollary
\ref{cor_RatesForTheEmpiricalRisk} and Theorem
\ref{thm_OptimalAdaptationForThePopulationRisk}. 
The oracles \( m^{ \mathfrak{o} } \) and \( m^{ \mathfrak{b} } \) vary only
very little over simulation runs. 
Their medians, in the same order as the signals displayed in Figure
\ref{fig_Uncorrelated_RelativeEfficiencies_TrueNoise}, are given by \( (4, 7, 14, 15, 45,
53) \) and \( (5, 10, 31, 15, 51, 66) \) respectively.
This is nearly identically replicated by the median sequential early stopping
times \( (5, 9, 23, 15, 44, 52) \).

In our second simulation, we estimate the empirical noise level using the Scaled
lasso estimator \( \widehat{ \sigma }^{2} \) from Section
\ref{sec_EstimationOfTheEmpiricalNoiseLevel}.
For the penalty parameter, we opt for the choice \(
  \lambda_{0} = \sqrt{ \log(p) / n }, 
\) which tended to have the best performance in the simulation study in Sun
Zhang \cite{SunZhang2012ScaledLasso}.
Note that the choice of \( \lambda_{0} \) in Proposition \ref{prp_FastNoiseEstimation} is scale invariant, see Proposition 1 in Sun and Zhang \cite{SunZhang2012ScaledLasso}.
We report boxplots of the estimation error \( | \widehat{ \sigma }^{2} - \|
\varepsilon \|_{n}^{2} | \) in Figure \ref{fig_Uncorrelated_NoiseEstimation} together with the squared estimation error \( \| \widehat{F}^{ ( m^{ \mathfrak{o} } ) } - f^{*} \|_{n}^{2} \) at the classical oracle.

The results indicate that the two quantities are of the same order, which is the
essential requirement for optimal adaptation in Theorem
\ref{thm_OracleInequalityForTheEmpiricalRisk} and Theorem
\ref{thm_OracleInequalityForTheEmpiricalRisk}. 
This is born out by the relative efficiencies in Figure
\ref{fig_Uncorrelated_RelativeEfficiencies_lambda0_1}, which remain bounded away from zero.
For the signals \( \beta^{ (3) }, \beta^{ (2) }, \beta^{ (1) } \), the quality
of estimation is comparable to that in Figure
\ref{fig_Uncorrelated_RelativeEfficiencies_TrueNoise}.
For the signals \( \beta^{ ( 15 ) }, \beta^{ ( 60 ) }, \beta^{ ( 90 ) } \), the
quality of estimation decreases, which matches the fact that for these signals,
the noise estimation deviates more from the risk at the classical oracle
\( m^{ \mathfrak{o} } \). 
The median stopping times \( (4, 6, 14, 22, 20) \) indicate that for these signals, we tend to stop too early.
Nevertheless, in our simulation, early stopping achieves the optimal estimation
risk up to a constant of at most eight. 

Overall, this confirms the major claim of this paper that it is possible to
achieve optimal adaptation by a fully data-driven, sequential early stopping
procedure. 
The computation times in Table \ref{tab_ComputationTimesForAllSignals} show an
improvement of an order of magnitude in the computational complexity relative
to exhaustive model selection methods as the high-dimensional Akaike criterion
from Ing \cite{Ing2020ModelSelection} or the cross-validated Lasso.

\begin{minipage}{0.49\textwidth}
  \centering
  \begin{figure}[H]
    \includegraphics[width=0.99\linewidth]{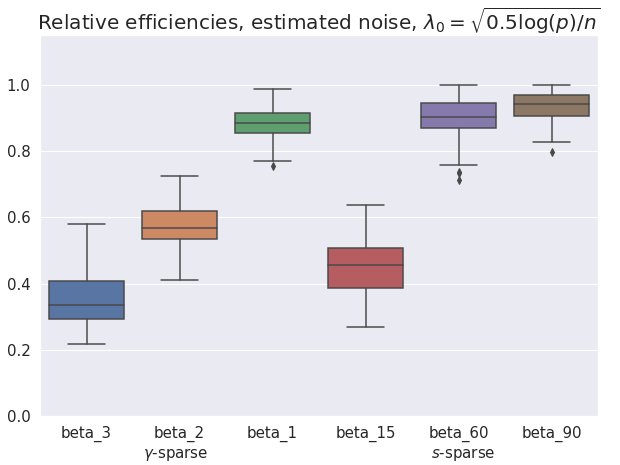}
    \\[-2.0ex]
    \caption{Relative efficiencies for early stopping with estimated empirical
             noise level for \( \lambda_{0} = \sqrt{ 0.5 \log(p) / n } \).}
    \label{fig_Uncorrelated_RelativeEfficiencies_lambda0_05}
  \end{figure}
\end{minipage}
\begin{minipage}{0.49\textwidth}
  \centering
  \begin{figure}[H]
    \includegraphics[width=0.99\linewidth]{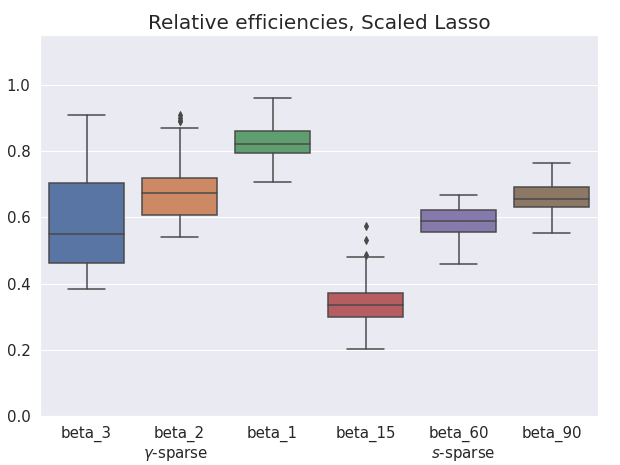}
    \\[-2.0ex]
    \caption{Relative efficiencies for the Scaled Lasso with 
             \( \lambda_{0} = \sqrt{ \log(p) / n } \).}
    \label{fig_Uncorrelated_RelativeEfficiencies_ScaledLasso}
  \end{figure}
\end{minipage}

Experimenting with different simulation setups, however, also reveals some
shortcomings of our methodology.
The performance of the early stopping procedure is fairly sensitive to the
noise estimation. 
This can already be surmised by comparing Figures
\ref{fig_Uncorrelated_RelativeEfficiencies_TrueNoise} and
\ref{fig_Uncorrelated_RelativeEfficiencies_lambda0_1},
and Theorem \ref{thm_OracleInequalityForTheEmpiricalRisk} suggests that the risk
of the estimation method is additive in the estimation error of the empirical
noise level.
In Figure \ref{fig_Uncorrelated_RelativeEfficiencies_lambda0_05}, we
present the relative efficiencies when the empirical noise level is estimated
with the penalty \( \lambda_{0} = \sqrt{ 0.5 \log(p) / n } \). 
The median stopping times \( 
  ( 17, 18, 25, 21, 39, 41 )
\) indicate that the change from \( 1 \) to a factor \( 0.5 \) in 
\( \lambda_{0} \) already makes the difference between stopping slightly too
early and stopping slightly too late.
While the relative efficiencies show that this does not make our
method unusable, ideally this sensitivity should be reduced.

Further, the joint minimization of the Scaled Lasso objective
\eqref{eq_4_ScaledLassoObjective} always includes computing an estimator 
\( \widehat{ \beta } \) of the coefficients. 
In particular, Corollary 1 in Sun and Zhang \cite{SunZhang2012ScaledLasso} also
guarantees optimal adaptation of this estimator at least under 
\( s \)-sparsity.
Ex ante, it is therefore unclear why one should apply our stopped boosting
algorithm on top of the noise estimate rather than just using the Scaled Lasso
estimator of the signal.
In Figure \ref{fig_Uncorrelated_RelativeEfficiencies_ScaledLasso}, we report the
relative efficiencies
\begin{align}
  \min_{ m \ge 0 }
  \| \widehat{F}^{ (m) } - f^{*} \|_{n} /
  \| \mathbf{X} \widehat{ \beta } - f^{*} \|_{n} 
\end{align}
of the Lasso estimator for the same penalty parameters \(
  \lambda_{0} = \sqrt{ \log(p) / n } 
\) which we considered for the initial noise estimation.
Note that this quantity can potentially be larger than one, in case the Lasso
risk is smaller than the risk at the classical oracle boosting iteration 
\( m^{ \mathfrak{o} } \).
Sequential early stopping slightly outperforms the Scaled Lasso estimator,
which we also confirmed in other experiments. 
Naturally, it shares the sensitivity to the choice of \( \lambda_{0} \).
Overall, the stopped boosting algorithm would have to produce results more
stable and closer to the benchmark in Figure
\ref{fig_Uncorrelated_RelativeEfficiencies_TrueNoise} to warrant a clear
preference.
We address these issues in the following section.


\subsection{An improved two step procedure}
\label{ssec_AnImprovedTwoStepProcedure}

We aim to make our methodology more robust to deviations of the estimated
empirical noise level and, at the same time, improve its estimation quality in
order to match the results from Figure
\ref{fig_Uncorrelated_RelativeEfficiencies_TrueNoise} more closely.
Motivated by Blanchard et al. \cite{BlanchardEtal2018a}, we propose a two-step
procedure combining early stopping with an additional model selection

\begin{minipage}{0.5\textwidth}
    \centering
    \begin{figure}[H]
      \includegraphics[width=0.99\linewidth]{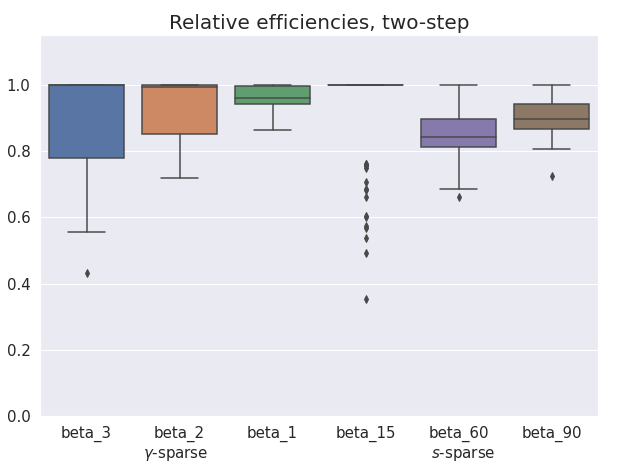}
      \caption{Relative efficiencies for the two-step procedure.}
      \label{fig_Uncorrelated_RelativeEfficiencies_TwoStep} 
    \end{figure}
\end{minipage}
\begin{minipage}{0.5\textwidth}
  \centering
  \begin{table}[H]
    \centering
    \begin{tabular}{lrrr}
      \\
      \\[-0.5ex]
      \\[-1.8ex]\hline
      \\[-1.8ex]\hline\\[-1.8ex]
      \( \beta \)     &    3  &    2  &      1 \\
      \hline\\[-1.8ex]
      True noise      &  12.5 &  19.8 &   47.3 \\
      Est. noise      &  25.3 &  32.0 &   42.8 \\
      Two-step        &  50.5 &  49.6 &   65.3 \\ 
      HDAIC             & 413.7 & 411.6 &  411.9 \\
      Lasso CV        & 133.4 & 164.3 & 1259.5 \\
      \hline\\[-1.8ex]
      \( \beta \)     & 15      &    60  &     90 \\
      \hline\\[-1.8ex]
      True noise      &    28.1 &   79.7 &   90.0 \\ 
      Est. noise      &    40.9 &   57.4 &   57.8 \\
      Two-step        &    56.0 &   90.0 &   92.6 \\ 
      HDAIC             &   410.9 &  412.2 &  407.6 \\
      Lasso CV        &  3741.6 & 3323.7 & 4290.4 \\
      \\[-1.8ex]\hline
      \hline\\[-1.8ex]
      \\
    \end{tabular}
    \caption{Computation times for different methods in seconds.}
    \label{tab_ComputationTimesForAllSignals} 
  \end{table}
\end{minipage}

\begin{minipage}{0.5\textwidth}
  \centering
  \begin{figure}[H]
    \includegraphics[width=0.99\linewidth]{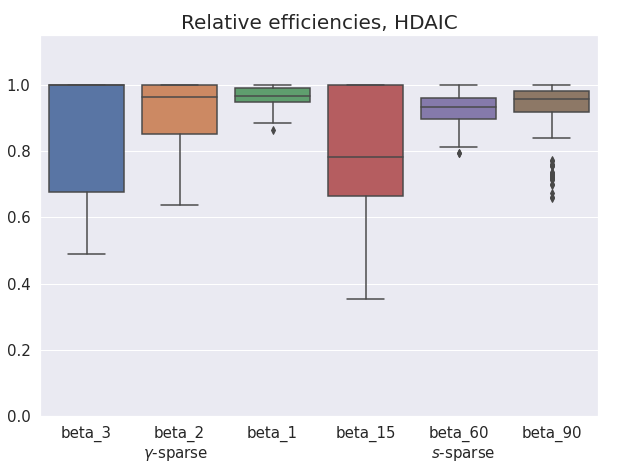}
    \\[-2.0ex]
    \caption{Relative efficiencies for the Akaike criterion from Ing \cite{Ing2020ModelSelection}owith \( C_{ \text{HDAIC} } = 2.25 \).}
    \label{fig_Uncorrelated_RelativeEfficiencies_HDAIC} 
  \end{figure}
\end{minipage}
\begin{minipage}{0.5\textwidth}
  \centering
  \begin{figure}[H]
    \includegraphics[width=0.99\linewidth]{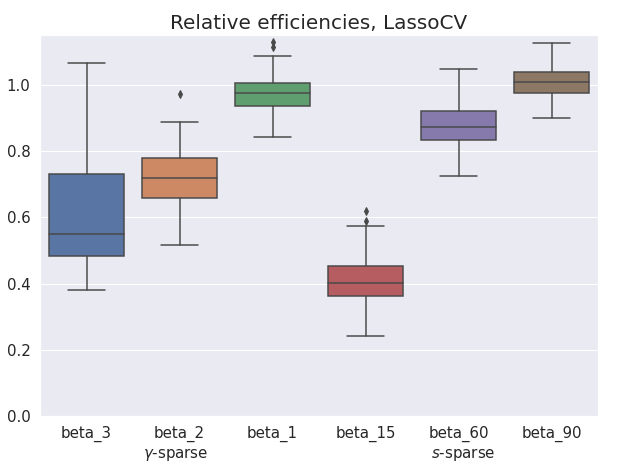}
    \\[-2.0ex]
    \caption{Relative efficiencies for the Lasso based on 5-fold cross-validation.}
    \label{fig_Uncorrelated_RelativeEfficiencies_LassoCV} 
  \end{figure}
\end{minipage}

\noindent step based on the high-dimensional Akaike-information criterion
\begin{align}
  \label{eq_5_NewHDAIC}
  \widehat{m}_{ \text{AIC} }: = \argmin_{ m \ge 0 } \text{AIC}(m) 
  \qquad \text{with} \qquad 
  \text{AIC}(m):
  =
  r_{m}^{2} + \frac{ C_{ \text{AIC} } m \log p }{n}, 
  \quad m \ge 0.
\end{align}
This criterion slightly differs from the one introduced in Ing
\cite{Ing2020ModelSelection}, which is necessary for our setting, see Remark
\ref{rem_TwoStepProcedure}. 
In combination, we select the iteration
\begin{align}
  \tau_{ \text{two-step} }: = \argmin_{ m \le \tau } \text{AIC}(m) 
  \qquad \text{ with } \qquad 
  \tau \text{ from Equation \eqref{eq_1_SequentialEarlyStoppingTimeKappaM}}.
\end{align}
Since this only requires \( \tau \) additional comparisons of 
\( \text{AIC}(m) \) for \( m \le \tau \), the two-step procedure has the same
computational complexity as the estimator \( \widehat{F}^{ ( \tau ) } \).  

The two-step procedure enables us to directly address the sensitivity of 
\( \tau \) to the noise estimation.
Our results for fully sequential early stopping in Theorem
\ref{thm_OracleInequalityForTheEmpiricalRisk}, Corollary
\ref{cor_RatesForTheEmpiricalRisk} and Theorem
\ref{thm_OptimalAdaptationForThePopulationRisk} require estimating the noise
level with the optimal rate \( \mathcal{R}( s, \gamma ) \).
Conversely, assuming that the high-dimensional Akaike criterion selects an
iteration such that its risk is of optimal order among the iterations 
\( m \le \tau \), the two-step procedure only requires an estimate 
\( \widehat{ \sigma }^{2} \) of \( \| \varepsilon \|_{n}^{2} \)
which has a slightly negative bias.
Proposition \ref{prp_NoStoppingTooEarly} then guarantees that \(
  \tau \ge \tilde m_{ s, \gamma, G } 
\) from Equation \eqref{eq_3_TildeMG} for some \( G > 0 \) with probability
converging to one, i.e., the indices \( m \le \tau \) contain an iteration with
risk of order \( \mathcal{R}( s, \gamma ) \). 
Moreover, the second selection step guarantees that as long as this is
satisfied, any imprecision in \( \tau \) only results in a slightly increased or
decreased computation time rather than changes in the estimation risk.
We can establish the following theoretical guarantee: 

\begin{theorem}[Two-step procedure]
  \label{thm_TwoStepProcedure}
  Under Assumptions 
  \hyperref[ass_SubGaussianErrors]{\normalfont \textbf{{\color{blue} (SubGE)}}},
  \hyperref[ass_Sparse]{\normalfont \textbf{{\color{blue} (Sparse)}}},
  \hyperref[ass_SubGaussianDesign]{\normalfont \textbf{{\color{blue} (SubGD)}}}
  and
  \hyperref[ass_CovB]{\normalfont \textbf{{\color{blue} (CovB)}}},
  choose \( \widehat{ \sigma }^{2} \) in Equation
  \eqref{eq_1_SequentialEarlyStoppingTimeKappaM} such that 
  \begin{align*}
        \widehat{ \sigma }^{2}  
    \le 
        \| \varepsilon \|_{n}^{2} + C \mathcal{R}( s, \gamma ) 
  \end{align*}
  with probability converging to one.
  Then, for any choice 
  \(
    C_{ \tau } = c ( \overline{ \sigma }^{2} + \rho^{4} ) 
  \) 
  in \eqref{eq_1_SequentialEarlyStoppingTimeKappaM} with \( c \ge 0 \) and
  \( C_{ \text{AIC} } = C ( \overline{ \sigma }^{2} + \rho^{4} ) \) with \( C > 0 \) large enough,
  the two-step procedure satisfies that with probability converging to one,
  \( 
    \tau_{ \text{two-step} } \ge \tilde m_{ s, \gamma, G }
  \) 
  from Equation \eqref{eq_3_TildeMG} for some \( G > 0 \).
  On the corresponding event,
  \begin{align*}
          \| \widehat{F}^{ ( \tau_{ \text{two-step} } ) } - f^{*} \|_{ L^2 }^{2} 
    & \le 
          C \mathcal{R}( s, \gamma ).
  \end{align*}
\end{theorem}

\noindent Due to our \( \omega \)-pointwise analysis on high probability events, the proof of Theorem \ref{thm_TwoStepProcedure} is simpler than the result for the two-step procedure in Proposition 4.2 of Blanchard et al. \cite{BlanchardEtal2018a}.
In particular, it does not require the analysis of probabilities conditioned on
events \( \{ \tau \le m \} \). 
The details are in Appendix \ref{sec_ProofsForTheMainResults}.
We note two important technical aspects of the two-step procedure:

\begin{remark}[Two-step procedure]
  \label{rem_TwoStepProcedure}
  \
  \begin{enumerate}

    \item[(a)] Under Assumptions 
      \hyperref[ass_SubGaussianErrors]{\normalfont \textbf{{\color{blue} (SubGE)}}},
      \hyperref[ass_Sparse]{\normalfont \textbf{{\color{blue} (Sparse)}}},
      \hyperref[ass_SubGaussianDesign]{\normalfont \textbf{{\color{blue} (SubGD)}}}
      and
      \hyperref[ass_CovB]{\normalfont \textbf{{\color{blue} (CovB)}}}, taken by
      itself, the Akaike criterion in \eqref{eq_5_NewHDAIC} satisfies \( 
            \|
              \widehat{F}^{ ( \widehat{m}_{ \text{AIC} } ) } - f^{*} 
            \|_{ L^{2} }^{2} 
        \le 
            C \mathcal{R}( s, \gamma )
      \).
      The proof of this statement is included in the proof of Theorem
      \ref{thm_TwoStepProcedure}. 
      The criterion in Ing \cite{Ing2020ModelSelection} minimizes
      \begin{align}
        \text{HDAIC}(m): 
        = 
        r_{m}^{2} \Big( 1 + \frac{ C_{ \text{HDAIC} } m \log p  }{n} \Big), 
        \qquad 
        0 \le m \le M_{n}.
      \end{align}
      Including iterations \( m > M_{n} \) potentially makes it unreliable,
      since \( r_{m}^{2} \to 0 \) for \( m \to \infty \).
      In particular, \( \text{HDAIC}(n) = 0  \). 
      The fact that under the assumptions of Theorem \ref{thm_TwoStepProcedure},
      we have no upper bound \( \tau \le M_{n} \) therefore makes it necessary to
      formulate the new criterion in Equation \eqref{eq_5_NewHDAIC}. 

    \item[(b)] Under the assumptions of Theorem \ref{thm_TwoStepProcedure},
      there is no upper bound for \( \tau \).
      For any noise estimate \( \widehat{ \sigma }^{2} \) with \(
            \widehat{ \sigma }^{2} 
        \ge
            \| \varepsilon \|_{n}^{2} - C \mathcal{R}( s, \gamma ) 
      \), however, Proposition \ref{prp_NoStoppingTooLate} guarantees that \(
        \tau \le H m^{*}_{ s, \gamma }  
      \) for some \( H > 0 \) with probability converging to one.
      The two-step procedure therefore retains the computational advantages of
      early stopping.

  \end{enumerate}
  
\end{remark}

For simulations, this suggests choosing the smaller penalty parameter \(
  \lambda_{0} = \sqrt{ 0.5 \log(p) / n } 
\) in the Scaled Lasso objective \eqref{eq_4_ScaledLassoObjective}, which puts a
negative bias on \( \widehat{ \sigma }^{2} \), and then applying the two-step
procedure.
The proof of Theorem \ref{thm_TwoStepProcedure} shows that the penalty term \(
C_{ \text{AIC} } m \log(p) / n \) in Equation \eqref{eq_5_NewHDAIC} essentially
has to dominate the empirical stochastic error \( s_{m} \).
In accordance with Lemma \ref{lem_BoundForTheEmpiricalStochasticError}, we
therefore choose \( C_{ \text{AIC} } = 2 \overline{ \sigma }^{2} = 2 \).
Compared to Figure \ref{fig_Uncorrelated_RelativeEfficiencies_lambda0_05}, the
results in Figure \ref{fig_Uncorrelated_RelativeEfficiencies_TwoStep} show that
the high-dimensional Akaike criterion corrects the instances where \( \tau \)
stops later than the oracle indices.
Empirically, the method attains the risk at the pointwise classical oracle 
\( m^{ ( \mathfrak{o} ) } \) up to a factor \( C_{ \text{Risk} } = 2 \).
The median two-step times are given by \( (4, 7, 12, 15, 37, 37) \).
Overall, performance of the two-step procedure comes very close to the benchmark
results in Figure \ref{fig_Uncorrelated_RelativeEfficiencies_TrueNoise}.
It is much better than that of the Scaled Lasso and at least as good as that of
the full Akaike selection and the default method {\tt LassoCV} from the python
library {\tt scikit-learn} \cite{PedregosaEtal2011ScikitLearn} based on 5-fold
cross-validation, see Figures \ref{fig_Uncorrelated_RelativeEfficiencies_HDAIC}
and \ref{fig_Uncorrelated_RelativeEfficiencies_LassoCV}.
Since we have intentionally biased our noise estimate and iterate slightly
further, the computation times of the two-step procedure are slightly larger
than those of purely sequential early stopping.
Since they are still much lower than those of the full Akaike selection or the
cross-validated Lasso, however, the two-step procedure maintains most of the
advantages from early stopping and yet genuinely achieves the performance of
exhaustive selection criteria.



\bibliographystyle{imsart-number} 
\bibliography{references.bib}       

\begin{thebibliography}{26}

\bibitem{BarronEtal2008Greedy}
\begin{barticle}[author]
\bauthor{\bsnm{Barron},~\bfnm{A.~R.}\binits{A.~R.}},
  \bauthor{\bsnm{Cohen},~\bfnm{A.}\binits{A.}},
  \bauthor{\bsnm{Dahmen},~\bfnm{W.}\binits{W.}} \AND
  \bauthor{\bsnm{DeVore},~\bfnm{R.~A.}\binits{R.~A.}}
(\byear{2008}).
\btitle{Approximation and learning by greedy algorithms}.
\bjournal{The Annals of Statistics}
\bvolume{36}
\bpages{64-94}.
\end{barticle}
\endbibitem

\bibitem{Baxter1962FinitePredictor}
\begin{barticle}[author]
\bauthor{\bsnm{Baxter},~\bfnm{G.}\binits{G.}}
(\byear{1962}).
\btitle{An asymptotic result for the finite predictor}.
\bjournal{Mathematica Scandinavica}
\bvolume{10}
\bpages{137-144}.
\end{barticle}
\endbibitem

\bibitem{BlanchardEtal2018a}
\begin{barticle}[author]
\bauthor{\bsnm{Blanchard},~\bfnm{G.}\binits{G.}},
  \bauthor{\bsnm{Hoffmann},~\bfnm{M.}\binits{M.}} \AND
  \bauthor{\bsnm{Reiß},~\bfnm{M.}\binits{M.}}
(\byear{2018}).
\btitle{Early stopping for statistical inverse problems via truncated {SVD}
  estimation}.
\bjournal{Electronic Journal of Statistics}
\bvolume{12}
\bpages{3204-3231}.
\end{barticle}
\endbibitem

\bibitem{BlanchardEtal2018b}
\begin{barticle}[author]
\bauthor{\bsnm{Blanchard},~\bfnm{G.}\binits{G.}},
  \bauthor{\bsnm{Hoffmann},~\bfnm{M.}\binits{M.}} \AND
  \bauthor{\bsnm{Reiß},~\bfnm{M.}\binits{M.}}
(\byear{2018}).
\btitle{Optimal adaptation for early stopping in statistical inverse problems}.
\bjournal{SIAM/ASA Journal of Uncertainty Quantification}
\bvolume{6}
\bpages{1043-1075}.
\end{barticle}
\endbibitem

\bibitem{BlanchardMathe2012}
\begin{barticle}[author]
\bauthor{\bsnm{Blanchard},~\bfnm{G.}\binits{G.}} \AND
  \bauthor{\bsnm{Mathé},~\bfnm{P.}\binits{P.}}
(\byear{2012}).
\btitle{Discrepancy principle for statistical inverse problems with application
  to conjugate gradient iteration}.
\bjournal{Inverse Problems}
\bvolume{28}
\bpages{115011/1--115011/23}.
\end{barticle}
\endbibitem

\bibitem{Buehlmann2006BoostingForHDLinearModels}
\begin{barticle}[author]
\bauthor{\bsnm{Bühlmann},~\bfnm{P.}\binits{P.}}
(\byear{2006}).
\btitle{Boosting for high-dimensional linear models}.
\bjournal{The annals of statistics}
\bvolume{34}
\bpages{559-583}.
\end{barticle}
\endbibitem

\bibitem{BuehlmannVDGeer2011HDData}
\begin{bbook}[author]
\bauthor{\bsnm{Bühlmann},~\bfnm{P.}\binits{P.}} \AND
  \bauthor{\bparticle{van~de} \bsnm{Geer},~\bfnm{S.}\binits{S.}}
(\byear{2011}).
\btitle{Statistics for High-Dimensional Data}.
\bpublisher{Springer}, \baddress{Heidelberg, Dordrecht, London, New York}.
\end{bbook}
\endbibitem

\bibitem{CelisseWahl2020Discrepancy}
\begin{bmisc}[author]
\bauthor{\bsnm{Celisse},~\bfnm{A.}\binits{A.}} \AND
  \bauthor{\bsnm{Wahl},~\bfnm{M.}\binits{M.}}
(\byear{2020}).
\btitle{Analyzing the discrepancy principle for kernelized spectral filter
  learning algorithms}.
\end{bmisc}
\endbibitem

\bibitem{EnglEtal1996InverseProblems}
\begin{bbook}[author]
\bauthor{\bsnm{Engl},~\bfnm{H.}\binits{H.}},
  \bauthor{\bsnm{Hanke},~\bfnm{M.}\binits{M.}} \AND
  \bauthor{\bsnm{Neubauer},~\bfnm{A.}\binits{A.}}
(\byear{1996}).
\btitle{Regularisation of inverse problems}.
\bseries{Mathematics and its applications}
\bvolume{375}.
\bpublisher{Kluwer Academic Publishers}, \baddress{Dordrecht}.
\end{bbook}
\endbibitem

\bibitem{GaoEtal2013LqHulls}
\begin{barticle}[author]
\bauthor{\bsnm{Gao},~\bfnm{F.}\binits{F.}},
  \bauthor{\bsnm{Ing},~\bfnm{C.}\binits{C.}} \AND
  \bauthor{\bsnm{Yang},~\bfnm{Y.}\binits{Y.}}
(\byear{2013}).
\btitle{Metric entropy and sparse linear approximation of {$\ell^q$}-hulls for
  {$ 0 < q \le 1$}}.
\bjournal{Journal of Approximation Theory}
\bvolume{166}
\bpages{42-55}.
\end{barticle}
\endbibitem

\bibitem{GillLevit1995vanTrees}
\begin{barticle}[author]
\bauthor{\bsnm{Gill},~\bfnm{R.~D.}\binits{R.~D.}} \AND
  \bauthor{\bsnm{B.},~\bfnm{Levit}\binits{L.}}
(\byear{1995}).
\btitle{Application of the van Trees inequality: a Bayesian Cramér-Rao bound}.
\bjournal{Bernoulli}
\bvolume{1}
\bpages{59-79}.
\end{barticle}
\endbibitem

\bibitem{Ing2020ModelSelection}
\begin{barticle}[author]
\bauthor{\bsnm{Ing},~\bfnm{C.}\binits{C.}}
(\byear{2020}).
\btitle{Model selection for high-dimensional linear regression with dependent
  observations}.
\bjournal{The Annals of Statistics}
\bvolume{48}
\bpages{1959--1980}.
\end{barticle}
\endbibitem

\bibitem{IngLai2011ConsistentModelSelection}
\begin{barticle}[author]
\bauthor{\bsnm{Ing},~\bfnm{C.}\binits{C.}} \AND
  \bauthor{\bsnm{Lai},~\bfnm{T.~L.}\binits{T.~L.}}
(\byear{2011}).
\btitle{A stepwise regression method and consistent model selection for
  high-dimensional sparse linear models}.
\bjournal{Statistica Sinica}
\bvolume{21}
\bpages{1473--1513}.
\end{barticle}
\endbibitem

\bibitem{Jahn2021Discrepancy}
\begin{barticle}[author]
\bauthor{\bsnm{Jahn},~\bfnm{T.}\binits{T.}}
(\byear{2022}).
\btitle{Optimal convergence of the discrepancy principle for polynomially and
  exponentially ill-posed operators for statistical inverse problems}.
\bjournal{Numerical Functional Analysis and Optimization}
\bvolume{42}
\bpages{145--167}.
\end{barticle}
\endbibitem

\bibitem{MeyerEtal2015Baxter}
\begin{barticle}[author]
\bauthor{\bsnm{Meyer},~\bfnm{M.}\binits{M.}},
  \bauthor{\bsnm{McMurry},~\bfnm{T.}\binits{T.}} \AND
  \bauthor{\bsnm{Politis},~\bfnm{D.}\binits{D.}}
(\byear{2015}).
\btitle{Baxter's inequality for triangular arrays}.
\bjournal{Mathematical methods of statistics}
\bvolume{24}
\bpages{135-146}.
\end{barticle}
\endbibitem

\bibitem{MikaSzkutnik2021DiscrepancyPrinciple}
\begin{barticle}[author]
\bauthor{\bsnm{Mika},~\bfnm{G.}\binits{G.}} \AND
  \bauthor{\bsnm{Szkutnik},~\bfnm{Z.}\binits{Z.}}
(\byear{2021}).
\btitle{Towards adaptivity via a new discrepancy principle for Poisson inverse
  problems}.
\bjournal{Electronic journal of statistics}
\bvolume{15}
\bpages{2029--2059}.
\end{barticle}
\endbibitem

\bibitem{NeedellVershynin2010RegularizedOMP}
\begin{barticle}[author]
\bauthor{\bsnm{Needell},~\bfnm{D.}\binits{D.}} \AND
  \bauthor{\bsnm{Vershynin},~\bfnm{R.}\binits{R.}}
(\byear{2010}).
\btitle{Signal Recovery From Incomplete and Inaccurate Measurements Via
  Regularized Orthogonal Matching Pursuit}.
\bjournal{Selected Topics in Signal Processing}
\bvolume{4}
\bpages{310-316}.
\end{barticle}
\endbibitem

\bibitem{PedregosaEtal2011ScikitLearn}
\begin{barticle}[author]
\bauthor{\bsnm{Pedregosa},~\bfnm{F.}\binits{F.}} \betal{et~al.}
(\byear{2011}).
\btitle{Scikit-learn: Machine Learning in {P}ython}.
\bjournal{Journal of Machine Learning Research}
\bvolume{12}
\bpages{2825--2830}.
\end{barticle}
\endbibitem

\bibitem{SchapireFreund2012Boosting}
\begin{bbook}[author]
\bauthor{\bsnm{Schapire},~\bfnm{R.~E}\binits{R.~E.}} \AND
  \bauthor{\bsnm{Freund},~\bfnm{Y.}\binits{Y.}}
(\byear{2012}).
\btitle{Boosting: Foundations and algorithms}.
\bpublisher{The MIT Press}, \baddress{Cambridge, Massachusetts}.
\end{bbook}
\endbibitem

\bibitem{Stankewitz2020Smoothed}
\begin{barticle}[author]
\bauthor{\bsnm{Stankewitz},~\bfnm{B.}\binits{B.}}
(\byear{2020}).
\btitle{Smoothed residual stopping for statistical inverse problems via
  truncated SVD estimation}.
\bjournal{Electronic Journal of Statistics}
\bvolume{14}
\bpages{3396-3428}.
\end{barticle}
\endbibitem

\bibitem{SunZhang2012ScaledLasso}
\begin{barticle}[author]
\bauthor{\bsnm{Sun},~\bfnm{T.}\binits{T.}} \AND
  \bauthor{\bsnm{Zhang},~\bfnm{C.~H.}\binits{C.~H.}}
(\byear{2012}).
\btitle{Scaled sparse linear regression}.
\bjournal{Biometrika}
\bvolume{99}
\bpages{879--898}.
\end{barticle}
\endbibitem

\bibitem{Temlyakov2000WeakGreedyAlgorithms}
\begin{barticle}[author]
\bauthor{\bsnm{Temlyakov},~\bfnm{V.~N.}\binits{V.~N.}}
(\byear{2000}).
\btitle{Weak greedy algorithms}.
\bjournal{Advances in Computational Mathematics}
\bvolume{12}
\bpages{213--227}.
\end{barticle}
\endbibitem

\bibitem{Tropp2004Greed}
\begin{barticle}[author]
\bauthor{\bsnm{Tropp},~\bfnm{J.~A.}\binits{J.~A.}}
(\byear{2004}).
\btitle{Greed is good: Algorithmic results for sparse approximation}.
\bjournal{Transactions of information theory}
\bvolume{50}
\bpages{2231-2242}.
\end{barticle}
\endbibitem

\bibitem{TroppGilbert2007OrthogonalMatchingPursuit}
\begin{barticle}[author]
\bauthor{\bsnm{Tropp},~\bfnm{J.~A.}\binits{J.~A.}} \AND
  \bauthor{\bsnm{Gilbert},~\bfnm{A.~C.}\binits{A.~C.}}
(\byear{2007}).
\btitle{Signal Recovery from Random Measurements via Orthogonal Matching
  Pursuit}.
\bjournal{Transactions of Information Theory}
\bvolume{53}
\bpages{4655-4666}.
\end{barticle}
\endbibitem

\bibitem{Vershynin2018HDProbability}
\begin{bbook}[author]
\bauthor{\bsnm{Vershynin},~\bfnm{R.}\binits{R.}}
(\byear{2018}).
\btitle{High-dimensional probability}.
\bseries{Cambridge series in statistical and probabilistic mathematics}.
\bpublisher{Cambridge university press}, \baddress{Cambridge}.
\end{bbook}
\endbibitem

\bibitem{Wainwright2019HDStatistics}
\begin{bbook}[author]
\bauthor{\bsnm{Wainwright},~\bfnm{M.~J.}\binits{M.~J.}}
(\byear{2019}).
\btitle{High-dimensional Statistics: A Non-asymptotic Viewpoint}.
\bpublisher{Cambridge University Press}, \baddress{Cambridge}.
\end{bbook}
\endbibitem

\end{thebibliography}

\newpage

\begin{appendix}

\section{Proofs for the main results}
\label{sec_ProofsForTheMainResults}

  \begin{proof}[Proof of Proposition \ref{prp_NoStoppingTooEarly}(No stopping too early)]
  Proposition \ref{prp_BoundForThePopulationBias} guarantees that 
  \( \tilde m_{ s, \gamma, G} \le m^{*}_{ s, \gamma } \) with probability
  converging to one given that \( G \) is large enough.
  We start by analyzing the left-hand side of the condition
  \eqref{eq_3_StoppingTooEarly}.
  By Lemma \ref{lem_BoundsForTheCrossTermAndTheDiscretizationError} (i), we have
  \begin{align}
    \label{eq_A_NoStoppingTooEarly_LowerBoundLHS_1}
            b_{m}^{2} + 2 c_{m} 
    & \ge 
            b_{m}^{2} 
            \Big( 
              1 
            -
              \sqrt{
                \frac{ 16 \overline{ \sigma }^{2} ( m + 1 ) \log p  }
                     { n b_{m}^{2} } 
              } 
            \Big)
    \qquad \text{ for all } m < \tilde m_{ s, \gamma, G}
  \end{align}
  with probability converging to one.

  We estimate \( b_{m}^{2} \) from below.
  By a standard convexity estimate, we can write
  \begin{align}
    \label{eq_A_NoStoppingTooEarly_LowerBoundEmpiricalBias_1}
            2 b_{m}^{2} 
    & = 
            2 \| ( I - \widehat{ \Pi }_{m} ) f^{*} \|_{n}^{2} 
      \ge 
            \| ( I - \Pi_{m} ) f^{*} \|_{n}^{2} 
          - 
            2 \| ( \widehat{ \Pi }_{m} - \Pi_{m} ) f^{*} \|_{n}^{2}. 
  \end{align}
  For the first term in Equation
  \eqref{eq_A_NoStoppingTooEarly_LowerBoundEmpiricalBias_1},
  we distinguish between the two possible sparsity assumptions.
  Under \( \gamma \)-sparsity, Proposition \ref{prp_FastNormChangeForTheBias}
  and Lemma \ref{lem_UniformBoundsInHighProbability} imply that with
  probability converging to one, for any \(
    m < \tilde m_{ s, \gamma, G} 
  \), 
  \begin{align}
          \| ( I - \Pi_{m} ) f^{*} \|_{n}^{2} 
    & \ge 
          \| ( I - \Pi_{m} ) f^{*} \|_{ L^2 }^{2} 
          \Big( 
            1 
            - 
            C \| ( I - \Pi_{m} ) f^{*} \|_{ L^2 }
                                     ^{ \frac{ - 2 }{ 2 \gamma - 1 } } 
            \sqrt{ \frac{ \rho^{4} \log p }{n} }
          \Big)
    \\
    & \ge 
          \| ( I - \Pi_{m} ) f^{*} \|_{ L^2 }^{2} 
          \Big( 
            1 
            - 
            \frac{C}{ G^{ 1 / ( 2 \gamma - 1 ) } }
            \Big( 
              \frac{ ( \overline{ \sigma }^{2} + \rho^{4} ) \log p }{n} 
            \Big)^{ \frac{1}{2} - \frac{1}{ 2 \gamma } }
          \Big).
    \notag
  \end{align}
  By increasing \( G > 0 \), the term in the outer parentheses becomes 
  larger than \( 1 / 2 \), which yields
  \begin{align}
    \label{eq_A_NoStoppingTooEarly_LowerBoundEmpiricalBias_2}
            \| ( I - \Pi_{m} ) f^{*} \|_{n}^{2} 
    & \ge 
            \frac{G}{2} 
            \Big( 
              \frac{ ( \overline{ \sigma }^{2} + \rho^{4} ) \log p }{n} 
            \Big)^{ 1 - \frac{1}{ 2 \gamma } }.
  \end{align}
  Under \( s \)-sparsity, analogously with probability converging to one, for
  any \(
    m < \tilde m_{ s, \gamma, G} 
  \), 
  \begin{align}
    \label{eq_A_NoStoppingTooEarly_LowerBoundEmpiricalBias_3}
          \| ( I - \Pi_{m} ) f^{*} \|_{n}^{2} 
    & \ge 
          \| ( I - \Pi_{m} ) f^{*} \|_{ L^2 }^{2} 
          \Big( 
            1 
            - 
            C ( s + m ) 
            \sqrt{ \frac{ \rho^{4} \log p }{n} }
          \Big)
    \\
    & \ge 
          \frac{1}{2} \| ( I - \Pi_{m} ) f^{*} \|_{ L^2 }^{2} 
      \ge 
          \frac{ c_{ \lambda } }{2} \underline{ \beta }^{2},
    \notag
  \end{align}
  where we have used that \( s = o( ( n / \log p )^{ 1 / 3 } ) \) and 
  \( m < \tilde m_{ s, \gamma, G} \le m^{*}_{ s, \gamma } \). 

  For the second term in Equation
  \eqref{eq_A_NoStoppingTooEarly_LowerBoundEmpiricalBias_1}, we can write
  \begin{align}
        ( \widehat{ \Pi }_{m} - \Pi_{m} ) f^{*} 
    & = 
        \widehat{ \Pi }_{m} ( I - \Pi_{m} ) f^{*} 
    = 
        ( X^{ ( \widehat{J}_{m} ) } )^{ \top } 
        \widehat{ \Gamma }_{ \widehat{J}_{m} }^{-1} 
        \langle
          ( I - \Pi_{m} ) f^{*}, g_{ \widehat{J}_{m} } 
        \rangle_{n}, 
  \end{align}
  i.e., the coefficients \( 
    \beta ( ( \widehat{ \Pi }_{m} - \Pi_{m} ) f^{*} ) 
  \) are given by \( 
    \widehat{ \Gamma }_{ \widehat{J}_{m} }^{-1} 
    \langle ( I - \Pi_{m} ) f^{*}, g_{ \widehat{J}_{m} } \rangle_{n} 
  \). 
  From Corollary \ref{cor_ReappearingTerms} (i), it then follows that
  \begin{align}
            \| \beta ( ( \widehat{ \Pi }_{m} - \Pi_{m} ) f^{*} ) \|_{2} 
    & \le 
            \| \widehat{ \Gamma }_{ \widehat{J}_{m} }^{-1} \|_{ \text{op} } 
            \sqrt{m} 
            \| 
              \langle ( I - \Pi_{m} ) f^{*}, g_{ \widehat{J}_{m} } \rangle_{n}  
            \|_{1} 
    \\
    & \le 
            C 
            \| \widehat{ \Gamma }_{ \widehat{J}_{m} }^{-1} \|_{ \text{op} } 
            \| \beta^{*} \|_{1}
            \sqrt{m} 
            \sup_{ j, k \le p } 
            | 
              \langle g_{j}, g_{k} \rangle_{n} 
            - 
              \langle g_{j}, g_{k} \rangle_{ L^{2} } 
            |.
    \notag
  \end{align}
  In combination with Lemma \ref{lem_UniformBoundsInHighProbability} we obtain
  that with probability converging to one,
  \begin{align}
    \label{eq_A_NoStoppingTooEarly_LowerBoundEmpiricalBias_4}
        \| ( \widehat{ \Pi }_{m} - \Pi_{m} ) f^{*} \|_{n}^{2} 
    & = 
        \tilde \beta^{ \top } \widehat{ \Gamma }_{ \widehat{J}_{m} } \tilde \beta 
    \le 
        C
        \| \beta^{*} \|_{1}^{2} 
        \frac{ \rho^{4} m \log p }{n} 
    \qquad \text{ for all } m \le m^{*}_{ s, \gamma }.
  \end{align}
  For \( s \)-sparse \( \beta^{*} \), this term converges to zero for 
  \( n \to \infty \). 
  For \( \gamma \)-sparse \( \beta^{*} \),  it is smaller than the rate 
  \( \mathcal{R}( s, \gamma ) \) up to a constant independent of \( G \). 

  By increasing \( G > 0 \) again, Equations
  \eqref{eq_A_NoStoppingTooEarly_LowerBoundEmpiricalBias_1},
  \eqref{eq_A_NoStoppingTooEarly_LowerBoundEmpiricalBias_2},
  \eqref{eq_A_NoStoppingTooEarly_LowerBoundEmpiricalBias_3},
  and
  \eqref{eq_A_NoStoppingTooEarly_LowerBoundEmpiricalBias_4}
  yield
  \begin{align}
    \label{eq_A_NoStoppingTooEarly_LowerBoundEmpiricalBias_5}
            b_{m}^{2} 
    & \ge 
            \begin{dcases}
              \frac{ c_{ \lambda } }{4} \underline{ \beta }^{2},
              & \beta^{*} \ s \text{-sparse}, \\
              \frac{G}{4} 
              \Big( 
                \frac{ ( \overline{ \sigma }^{2} + \rho^{4} ) \log p }{n}
              \Big)^{ 1 - 1 / ( 2 \gamma ) },
              & \beta^{*} \ \gamma \text{-sparse}
            \end{dcases}
            \\
    & \qquad \qquad \qquad \qquad \qquad \qquad
    \text{ for all } m < \tilde m_{ s, \gamma, G }
    \notag 
  \end{align}
  with probability converging to one.
  Plugging this into Equation \eqref{eq_A_NoStoppingTooEarly_LowerBoundLHS_1}, we
  obtain that the left-hand side in condition
  \eqref{eq_3_StoppingTooEarly} satisfies
  \begin{align}
    \label{eq_A_NoStoppingTooEarly_LowerBoundLHS_2}
            b_{m}^{2} + 2 c_{m} 
    & \ge 
            \begin{dcases}
              \frac{ c_{ \lambda } }{8} \underline{ \beta }^{2}, 
              & \beta^{*} \ s \text{-sparse}, \\
              \frac{G}{8} 
              \Big( 
                \frac{ ( \overline{ \sigma }^{2} + \rho^{4} ) \log p }{n}
              \Big)^{ 1 - 1 / ( 2 \gamma ) },
              & \beta^{*} \ \gamma \text{-sparse}
            \end{dcases}
            \\
    & \qquad \qquad \qquad \qquad \qquad \qquad
    \text{ for all } m < \tilde m_{ s, \gamma, G } 
    \notag
  \end{align}
  on an event with probability converging to one for \( G > 0 \) sufficiently
  large.

  At the same time, however, by Lemma \ref{lem_UniformBoundsInHighProbability}
  and our assumption on \( \widehat{ \sigma }^{2} \), the right-hand side in
  condition \eqref{eq_3_StoppingTooEarly} satisfies
  \begin{align}
    \label{eq_A_NoStoppingTooEarly_BoundRHS}
            \widehat{ \sigma }^{2} - \| \varepsilon \|_{n}^{2} 
          +
            C_{ \tau } \frac{ m \log p }{n} 
          + 
            s_{m} 
    & \le 
            C \mathcal{R}( s, \gamma ) 
    \qquad \text{ for all } m < \tilde m_{ s, \gamma, G }
  \end{align}
  with probability converging to one for a constant \( C \) 
  independent of \( G \).
  From Equation \eqref{eq_A_NoStoppingTooEarly_LowerBoundLHS_2} and
  \eqref{eq_A_NoStoppingTooEarly_BoundRHS}, it finally follows that for large 
  \( G > 0 \), condition \eqref{eq_3_StoppingTooEarly} can only be satisfied on
  a event with probability converging to zero.
  This finishes the proof.
\end{proof}

\begin{proof}[Proof of Proposition \ref{prp_NoStoppingTooLate}(No stopping too late)]
  Lemma \ref{lem_BoundsForTheCrossTermAndTheDiscretizationError} (i) yields that
  \begin{align}
            b_{m}^{2} + 2 c_{m} 
    \le 
            b_{m}^{2} 
          + 
            b_{m} 
            \sqrt{ \frac{ 16 \overline{ \sigma }^{2} ( m + 1 ) \log p }{n} } 
    \qquad \text{ for all } m \ge 0 
  \end{align}
  with probability converging to one.
  For \( s \)-sparse \( \beta^{*} \), with probability converging to one, this
  is zero for \( m \ge m^{*}_{ s, \gamma } \) by Proposition
  \ref{prp_BoundForThePopulationBias}. 
  For \( \gamma \)-sparse \( \beta^{*} \), Lemma
  \ref{lem_BoundForTheEmpiricalBias} provides the estimate
  \begin{align}
    & \ \ \ \
    b_{m}^{2} + 2 c_{m} 
    \\
    & \le 
    C \Big( 
      \Big[ 
        \frac{ ( \overline{ \sigma }^{2} + \rho^{4} ) \log p }{n}
      \Big]^{ 1 - \frac{1}{ 2 \gamma } } 
      + 
      \Big[ 
        \frac{ ( \overline{ \sigma }^{2} + \rho^{4} ) \log p  }{n}
      \Big]^{ \frac{1}{2} - \frac{1}{ 4 \gamma } } 
      \sqrt{ \frac{ 16 \overline{ \sigma }^{2} ( m + 1 ) \log p }{n} } 
    \Big) 
    \notag
    \\
    & \qquad \qquad \qquad \qquad \qquad
      \qquad \qquad \qquad \qquad \qquad
      \qquad \qquad 
    \text{ for all } m \ge m^{*}_{ s, \gamma } 
    \notag
  \end{align}
  with probability converging to one.
  For \( m = H m^{*}_{ s, \gamma } \) with \( H > 0 \) large enough, this yields
  \begin{align}
    b_{m}^{2} + 2 c_{m} 
    & \le 
    C \sqrt{H} 
    \Big( 
      \frac{ ( \overline{ \sigma }^{2} + \rho^{4} ) \log p }{n}
    \Big)^{ 1 - \frac{1}{ 2 \gamma } }.
  \end{align}
  At the same time, under the assumption on \( \widehat{ \sigma }^{2} \), the
  right-hand side of condition \eqref{eq_3_ConditionForLateStopping} satisfies
  \begin{align}
            \widehat{ \sigma }^{2} - \| \varepsilon \|_{n}^{2}
          + 
            C_{ \tau } \frac{ m \log p }{n} 
          + 
            s_{m} 
    & \ge 
            c H \mathcal{R}( s, \gamma ).
  \end{align}
  For \( H > 0 \) sufficiently large, condition
  \eqref{eq_3_ConditionForLateStopping} can therefore only be satisfied on an
  event with probability converging to zero.
\end{proof}

\begin{proof}[Proof of Proposition \ref{prp_FastNoiseEstimation} (Fast noise estimation)]
  Theorem 2 in Sun and Zhang \cite{SunZhang2012ScaledLasso} states that on the
  event 
  \begin{align}
      \Omega_{ \text{Lasso} }
    = 
      \Big\{ 
            \sup_{ j \le p } |
            \langle g_{j}, \varepsilon \rangle_{n} |
        \le 
            ( 1 - \alpha^{*} ) 
            \frac{ \xi - 1 }{ \xi + 1 } 
            \| \varepsilon \|_{n} \lambda_{0}
      \Big\},
  \end{align}
  the Scaled Lasso noise estimator \( \widehat{ \sigma } \) satisfies
  \begin{align}
            \max \Big( 
              1 - \frac{ \widehat{ \sigma } }{ \| \varepsilon \|_{n} }, 
              1 - \frac{ \| \varepsilon \|_{n} }{ \widehat{ \sigma } }
            \Big) 
    & = 
            1 - 
            \frac{ \min( \widehat{ \sigma }, \| \varepsilon \|_{n} ) }
                 { \max( \widehat{ \sigma }, \| \varepsilon \|_{n} ) } 
    \le 
            \alpha^{*} 
    = 
            \frac{ 
              \lambda_{0}
              \mu( \| \varepsilon \|_{n} \lambda_{0}, \xi )
            }{ \| \varepsilon \|_{n} }.
  \end{align}
  This implies that on \( \Omega_{ \text{Lasso} } \),
  \begin{align}
    \label{eq_FastNoiseEstimation_AbsoluteErrorBound}
            | \widehat{ \sigma }^{2} - \| \varepsilon \|_{n}^{2} | 
    & = 
            ( \widehat{ \sigma } + \| \varepsilon \|_{n} ) 
            ( 
              \max( \widehat{ \sigma }, \| \varepsilon \|_{n} )
            - \min( \widehat{ \sigma }, \| \varepsilon \|_{n} )
            ) 
    \\
    & \le 
            \Big( 
              1 + \frac{ \max( \widehat{ \sigma }, \| \varepsilon \|_{n} ) }
                       { \min( \widehat{ \sigma }, \| \varepsilon \|_{n} ) }
            \Big) 
            \max( \widehat{ \sigma }, \| \varepsilon \|_{n} )
            \lambda_{0}
            \mu( \| \varepsilon \|_{n} \lambda_{0}, \xi )
    \notag
    \\
    & \le 
            \Big( 1 + \frac{1}{ 1 - \alpha^{*} } \Big) 
            \max( \widehat{ \sigma }, \| \varepsilon \|_{n} )
            \lambda_{0}
            \mu( \| \varepsilon \|_{n} \lambda_{0}, \xi )
    \notag
    \\
    & \le 
            \Big( 
              \frac{1}{   1 - \alpha^{*}       }
            + \frac{1}{ ( 1 - \alpha^{*} )^{2} }
            \Big) 
            \| \varepsilon \|_{n} \lambda_{0}
            \mu( \| \varepsilon \|_{n} \lambda_{0}, \xi )
    \le 
            \frac{ 2 \| \varepsilon \|_{n}^{2} \alpha^{*} }
                 { ( 1 - \alpha^{*} )^{2} },
    \notag
  \end{align}
  where, without loss of generality, we have used that \( \alpha^{*} < 1 \),
  since for \( \alpha^{*} > 1 \), the event \( \Omega_{ \text{Lasso} } \) is
  empty. 
  It remains to be shown that Equation
  \eqref{eq_FastNoiseEstimation_AbsoluteErrorBound} provides a meaningful 
  bound on an event with probability converging to one.

  \textbf{Step 1: Bounding \( \mu( \lambda, \xi ) \).} 
  Set \( \nu = 1 / 2 \).
  For \( s \)-sparse \( \beta^{*} \), the choice \( 
    J = S = \{ j \le p: | \beta^{*}_{j} | \ne 0 \} 
  \) yields the immediate estimate
  \begin{align}
    \label{eq_A_SSparseBoundForMu}
            \mu( \| \varepsilon \|_{n} \lambda_{0}, \xi ) 
    & \le 
            ( \xi + 1 ) 
            \kappa^{ - 2 }( 2 \xi + 1, S ) 
            \| \varepsilon \|_{n} \lambda_{0} s.
  \end{align}

  For \( \gamma \)-sparse \( \beta^{*} \), the choice \( 
    J = J_{ \lambda } 
      = \{ j: | \beta^{*}_{j} | \ge \lambda \}
  \) yields the estimate
  \begin{align}
    \label{eq_FastNoiseEstimation_SimpleBoundForMu}
            \mu( \lambda, \xi ) 
    & \le 
            ( \xi + 1 ) 
            \max( 2, \kappa^{ - 2 }( 2 \xi + 1, J_{ \lambda } ) ) 
            \sum_{ j = 1 }^{p} 
            \min( \lambda, | \beta^{*}_{j} | ). 
  \end{align}
  Without loss of generality, we can assume that the 
  \( 
    ( \beta^{*}_{j} )_{ j \le p }
  \) 
  are decreasingly ordered. 
  We derive that for any \( m \in \mathbb{N} \), it holds that
  \( 
        \sum_{ j > m } | \beta^{*}_{j} |^{2} 
    \le 
        C m^{ 1 - 2 \gamma }.
  \) 
  By the \( \gamma \)-sparsity of \( \beta^{*} \), we have
  \begin{align}
                                \sum_{ j > m } | \beta^{*}_{j} |^{2} 
    & \le 
            | \beta^{*}_{m} |   \sum_{ j > m } | \beta^{*}_{j} |
    \le 
            | \beta^{*}_{m} | C_{ \gamma } 
                              \Big( 
                                \sum_{ j > m } | \beta^{*}_{j} |^{2}
                              \Big)^{ \frac{ \gamma - 1 }{ 2 \gamma - 1 } }.
  \end{align}
  Rearranging yields 
  \( 
    | \beta^{*}_{m} |^{2} \ge C_{ \gamma }^{ - 2 } 
                          (
                            \sum_{ j > m } | \beta^{*}_{j} |^{2}
                          )^{ 2 \gamma / ( 2 \gamma - 1 ) }
  \),
  which implies 
  \begin{align}
            \sum_{ j > m + 1 } | \beta^{*}_{j} |^{2} 
    & = 
            \sum_{ j > m     } | \beta^{*}_{j} |^{2} 
          - 
            | \beta^{*}_{m} |^{2} 
    \le 
            \sum_{ j > m } | \beta^{*}_{j} |^{2} 
            \Big( 
              1 
            - 
              C_{ \gamma }^{ - 2 } 
              \Big( 
                \sum_{ j > m } | \beta^{*}_{j} |^{2}
              \Big)^{ \frac{1}{ 2 \gamma - 1 } } 
            \Big).
  \end{align}
  As in the proof of Proposition \ref{prp_BoundForThePopulationBias}, the
  intermediate  claim now follows from Lemma 1 in Gao et al.
  \cite{GaoEtal2013LqHulls} by setting \( 
    a_{m}: = \sum_{ j > m } | \beta^{*}_{j} |^{2} 
  \). 
  From the above, we obtain that for any \( m \in \mathbb{N} \), 
  \begin{align}
    \label{eq_FastNoiseEstimation_BoundForTheSum}
            \sum_{ j = 1 }^{p} 
            \min( \lambda, | \beta^{*}_{j} | ) 
    & \le 
            m \lambda 
          + 
            \sum_{ j > m } | \beta^{*}_{j} | 
    \le 
            m \lambda 
          + 
            C_{ \gamma }
            \Big( 
              \sum_{ j > m } | \beta^{*}_{j} |^{2}
            \Big)^{ \frac{ \gamma - 1 }{ 2 \gamma - 1 } } 
    \\
    & \le 
            m \lambda + C m^{ 1 - \gamma }.
    \notag
  \end{align}
  For 
  \(
    \lambda = \| \varepsilon \|_{n} \lambda_{0} 
  \)
  and a choice \( m \) of order
  \( 
    ( n / ( \| \varepsilon \|_{n}^{2} \log p )  )^{ 1 / ( 2 \gamma ) }
  \), 
  Equations \eqref{eq_FastNoiseEstimation_SimpleBoundForMu} and
  \eqref{eq_FastNoiseEstimation_BoundForTheSum} translate to the estimate
  \begin{align}
    \label{eq_FastNoiseEstimate_FinalBoundForMu}
            \mu( \lambda, \xi ) 
    & \le 
            C ( \xi + 1 ) 
            \max( 2, \kappa^{ - 2 }( 2 \xi + 1, J_{ \lambda } ) ) 
            \Big( 
              \frac{ \| \varepsilon \|_{n}^{2} \log p }{n}
            \Big)^{ \frac{1}{2} - \frac{1}{ 2 \gamma } }.
  \end{align}

  \textbf{Step 2: Positive compatibility factor.} For the bounds in Equations
  \eqref{eq_A_SSparseBoundForMu} and
  \eqref{eq_FastNoiseEstimate_FinalBoundForMu} to be meaningful, we have to
  guarantee that the compatibility factor is strictly positive.

  For rows \(
    ( X_{i} )_{ i \le n }  \sim N( 0, \Gamma ) 
  \) i.i.d.  of the design matrix \( \mathbf{X} \), Theorem 7.16 in Wainwright
  \cite{Wainwright2019HDStatistics} states that 
  \begin{align}
    \label{eq_FastNoiseEstimation_CompatibilityFactorWainwright}
            \| \mathbf{X} \beta \|_{n}^{2} 
    & \ge 
            \frac{1}{8} 
            \| \sqrt{ \Gamma } \beta \|_{2}^{2} 
          - 
            50 \max_{ j \le p } \Gamma_{ j j } 
            \frac{ \log p }{n} 
            \| \beta \|_{1}^{2} 
    \qquad \text{ for all } \beta \in \mathbb{R}^{p}
  \end{align}
  on an event \( \Omega_{ \text{Comp} } \) with probability at least 
  \( 
    1 - e^{ - n / 32 } / ( 1 - e^{ - n / 32 } )
  \). 
  Since we assume unit variance design, the bound in Equation
  \eqref{eq_FastNoiseEstimation_CompatibilityFactorWainwright} implies 
  \begin{align}
            \| \mathbf{X} \beta \|_{n}^{2} 
    & \ge 
            \frac{\lambda_{ \min }( \Gamma ) }{8} \| \beta \|_{2}^{2} 
          - 
            \frac{ 50 \log p }{n} \| \beta \|_{1}^{2} 
      \ge 
            \frac{ c_{ \lambda } }{ 16 } \| \beta \|_{2}^{2} 
  \end{align}
  for all \( \beta \in \mathbb{R}^{p} \) such that
  \begin{align}
    \label{eq_FastNoiseEstimation_REPrecondition}
        \frac{ 50 \log p }{n} \| \beta \|_{1}^{2} 
    \le 
        \frac{ c_{ \lambda } }{ 16 } \| \beta \|_{2}^{2}. 
  \end{align}
  If 
  \( 
    \| \beta_{ J^{c} } \|_{1} \le \xi \| \beta_{J} \|_{1} 
  \) 
  for some \( \xi > 1 \) and \( J \subset \{ 1, \dots, p \} \), then
  \begin{align}
            \| \beta \|_{1}^{2} 
    & = 
            ( \| \beta_{J} \|_{1} + \| \beta_{ J^{c} }  \|_{1} )^{2} 
    \le 
            ( 1 + \xi )^{2} \| \beta_{J} \|_{1}^{2} 
    \le 
            ( 1 + \xi )^{2} | J | \| \beta \|_{2}^{2}. 
  \end{align}
  Plugging this estimate into the left-hand side of condition
  \eqref{eq_FastNoiseEstimation_REPrecondition} yields that on 
  \( \Omega_{ \text{Comp} } \), \( \mathbf{X} \) satisfies the \emph{restricted
  eigenvalue condition}
  \begin{align}
    \label{eq_FastNoiseEstimation_RECondition}
            \| \mathbf{X} \beta \|_{n}^{2} 
    & \ge 
            \frac{ c_{ \lambda } }{ 16 } 
            \| \beta \|_{2}^{2}, 
    \qquad \text{ for all } \beta \in \mathbb{R}^{p}: 
                            \| \beta_{ J^{c} } \|_{1} \le \xi \| \beta_{J} \|_{1}
  \end{align}
  and all sets \( J \subset \{ 1, \dots, p \} \) with 
  \(
    | J | \le c_{ \lambda } / 800 ( 1 + \xi )^{ - 2 } n / \log p
  \). 
  However, 
  due to the estimate
  \begin{align}
            \| \beta_{J} \|_{1}^{2}
    & \le 
            | J | \| \beta \|_{2}^{2} 
    \le 
            \frac{ 16 | J | }{ c_{ \lambda } } 
            \| \mathbf{X} \beta \|_{n}^{2} 
    \qquad \text{ for all } \beta \in \mathbb{R}^{p}: 
                            \| \beta_{ J^{c} } \|_{1} \le \xi \| \beta_{J} \|_{1},
  \end{align}
  this implies that the compatibility factor 
  \( \kappa^{2}( \xi, J ) \) is larger than \( c_{ \lambda } / 16 \). 

  Under \( s \)-sparsity, the set \(
    S = \{ j \le p: | \beta^{*}_{j} | \ne 0 \}
  \) immediately satisfies the assumption on \( J \) above for \( n \) large
  enough. 
  Under \( \gamma \)-sparsity, let \( \Omega_{ \text{Conv} } \) be the event of
  probability one on which \( \| \varepsilon \|_{n}^{2} \) converges to the
  (unconditional) variance \(
    \text{Var}( \varepsilon_{1} ) = \underline{ \sigma }^{2} > 0 
  \). 
  Since \( \| \beta^{*} \|_{1} \le C \) with some constant \( C > 0 \), for 
  \( \lambda = \| \varepsilon \|_{n} \lambda_{0} \), 
  \begin{align}
            | J_{ \lambda } | 
    & \le 
            \frac{C}{ C_{ \lambda_{0} } } 
            \sqrt{ \frac{n}{ \| \varepsilon \|_{n}^{2} \log p } }
      \le 
            \frac{ c_{ \lambda } }{ 800 ( 1 + \xi )^{2} } 
            \frac{n}{ \log p } 
  \end{align}
  on \( \Omega_{ \text{Conv} } \) for \( n \) sufficiently large.
  We conclude that on 
  \(
    \Omega_{ \text{Conv} } \cap \Omega_{ \text{Comp} } 
  \),
  for any \( \xi > 1 \), the compatibility factor
  \( \kappa^{2}( 2 \xi + 1, J_{ \lambda } )  \) is larger than 
  \( c_{ \lambda } / 16 \) for \( n \) large enough.

  \textbf{Step 3: Bound on the combined event.} 
  Finally, for 
  \(
    \lambda_{0} = C_{ \lambda_{0} } ( \xi + 1 ) / ( \xi - 1 )  
                  \sqrt{ \log(p) / n }, 
  \) 
  we have from Step 2 and Lemma \ref{lem_UniformBoundsInHighProbability} (ii)
  that
  \begin{align}
            \mathbb{P}( \Omega_{ \text{Lasso} }^{c} ) 
    & = 
            \mathbb{P} \Big( 
              \Big\{ 
                \sup_{ j \le p } | \langle g_{j}, \varepsilon \rangle_{n} |
              > 
                ( 1 - \alpha^{*} ) 
                \frac{ \xi - 1 }{ \xi + 1 } 
                \| \varepsilon \|_{n} \lambda_{0}
              \Big\}
              \cap \Omega_{ \text{Conv} } 
              \cap \Omega_{ \text{Comp} } 
            \Big) 
          + 
            o(1) 
    \\
    & \le 
    \notag
            \mathbb{P} \Big( 
              \Big\{ 
                \sup_{ j \le p } | \langle g_{j}, \varepsilon \rangle_{n} |
              > 
                \frac{ \xi - 1 }{ \xi + 1 } 
                \frac{ \underline{ \sigma } }{2} \lambda_{0} 
              \Big\}
              \cap \Omega_{ \text{Conv} } 
              \cap \Omega_{ \text{comp} } 
            \Big) 
          + 
            o(1) 
    \\
    & \le 
            \mathbb{P} \Big( 
              \Big\{ 
                \sup_{ j \le p } | \langle g_{j}, \varepsilon \rangle_{n} |
              > 
                C_{ \varepsilon } 
                \sqrt{ \frac{ \overline{ \sigma }^{2} \log p }{n} } 
              \Big\}
            \Big) 
          + 
            o(1) 
    \xrightarrow[ ]{ n \to \infty } 0. 
    \notag
  \end{align}
  We conclude that
  \( 
         \Omega_{ \text{Lasso} } 
    \cap \Omega_{ \text{Conv} } 
    \cap \Omega_{ \text{Comp} } 
  \) 
  is an event with probability converging to one on which 
  by Equations \eqref{eq_FastNoiseEstimation_AbsoluteErrorBound},
  \eqref{eq_A_SSparseBoundForMu},
  \eqref{eq_FastNoiseEstimate_FinalBoundForMu} and Step 2,
  \begin{align}
            | \widehat{ \sigma }^{2} - \| \varepsilon \|_{n}^{2} | 
    & \le 
            \frac{ 2 \| \varepsilon \|_{n}^{2} \alpha^{*} }
                 { ( 1 - \alpha^{*} )^{2} }
      \le 
            C \| \varepsilon \|_{n}\lambda_{0}
              \mu( \| \varepsilon \|_{n} \lambda_{0}, \xi ) 
    \\
    & \le   
            C
            \begin{dcases}
              \Big( 
                \frac{ \overline{ \sigma }^{2} \log p }{n}
              \Big)^{ 1 - 1 / ( 2 \gamma ) },
              & \beta^{*} \ s \text{-sparse}, \\
              \frac{ \overline{ \sigma }^{2} s \log p }{n},
              & \beta^{*} \ \gamma \text{-sparse}
            \end{dcases}
  \end{align}
  for \( n \) sufficiently large. 
  This finishes the proof.
\end{proof}

\begin{proof}[Proof of Theorem \ref{thm_TwoStepProcedure} (Two-step procedure)]

  The proof follows along the same arguments that we have applied in the
  derivation of Proposition \ref{prp_NoStoppingTooEarly} and Proposition
  \ref{prp_NoStoppingTooLate}.

  From Proposition \ref{prp_NoStoppingTooEarly}, we already know that for some
  \( G > 0 \), the sequential stopping time satisfies \(
    \tau \ge \tilde m_{ s, \gamma, G} 
  \).
  For \( G' > G \) sufficiently large, we now show that \(
    \tau_{ \text{two-step} } \ge \tilde m_{ s, \gamma, G' }
  \) with probability converging to one.
  Assuming that \( \tau_{ \text{two-step} } < \tilde m_{ s, \gamma, G' } \), we
  obtain that
  \begin{align}
    \exists m < \tilde m_{ s, \gamma, G' }: 
        r_{m}^{2} + \frac{ C_{ \text{AIC} } m \log p }{n} 
    \le 
        r_{ \tau }^{2} + \frac{ C_{ \text{AIC} } \tau \log p }{n},
  \end{align}
  which is equivalent to 
  \begin{align}
    \label{eq_A_AICConditionTooEarly}
    \exists m < \tilde m_{ s, \gamma, G' }: 
            b_{m}^{2} + 2 c_{m} - s_{m} 
          + 
            \frac{ C_{ \text{AIC} } m \log p }{n} 
    \le 
            b_{ \tau }^{2} + 2 c_{ \tau } - s_{ \tau } 
          + 
            \frac{ C_{ \text{AIC} } \tau \log p }{n}.
  \end{align}
  Combining the reasoning from the proof of Proposition
  \ref{prp_NoStoppingTooEarly} with the bound from Lemma
  \ref{lem_BoundForTheEmpiricalStochasticError}, the left-hand side of
  condition \eqref{eq_A_AICConditionTooEarly} is larger than
  \begin{align}
            \begin{dcases}
              \frac{ c_{ \lambda } }{ 16 } \underline{ \beta }^{2}, 
              & \beta^{*} \ s \text{-sparse}, \\
              \frac{ G' }{16} 
              \Big( 
                \frac{ ( \overline{ \sigma }^{2} + \rho^{4} ) \log p }{n}
              \Big)^{ 1 - 1 / ( 2 \gamma ) },
              & \beta^{*} \ \gamma \text{-sparse}
            \end{dcases}
  \end{align}
  with probability converging to one and \( G' \) sufficiently large.

  For the right-hand side of condition \eqref{eq_A_AICConditionTooEarly}, we
  can assume that \( \tau < m^{*}_{ s, \gamma } \) from Equation
  \eqref{eq_1_BalancedBoundOracleIndex}. 
  Otherwise, we may replace \( \tau \) with \( m^{*}_{ s, \gamma } \).
  Note that in the setting of Remark \ref{rem_TwoStepProcedure}, we can replace
  \( \tau \) with \( m^{*}_{ s, \gamma } \) from the start, which yields the 
  result stated there.

  Using that \(
    \tilde m_{ s, \gamma, G } 
    \le
    \tau
    <
    \tilde m_{ s, \gamma, G' } 
    \le 
    m_{ s, \gamma }^{*} 
  \) with probability converging to one, together with Lemmas
  \ref{lem_BoundsForTheCrossTermAndTheDiscretizationError}  and
  \ref{lem_BoundForTheEmpiricalStochasticError}, the right-hand side converges
  to zero under \( s \)-sparsity with probability converging to one and is
  smaller than \( C \mathcal{R}( s, \gamma ) \) with probability converging to
  one and \( C \) independent of \( G' \) under \( \gamma \)-sparsity.
  Therefore, \( \tau_{ \text{two-step} } < \tilde m_{ s, \gamma, G' } \) can
  only be true on an event with probability converging to zero.

  Similar to Proposition \ref{prp_NoStoppingTooLate}, we can also show that \(
    \tau_{ \text{two-step} } \le H m^{*}_{ s, \gamma } 
  \) with probability converging to one for \( H > 0 \) large enough.
  If \( \tau_{ \text{two-step} } > H m^{*}_{ s, \gamma } \), analogously to
  condition \eqref{eq_3_ConditionForLateStopping}, we have
  \begin{align}
    \label{eq_A_AICConditionTooLate}
    & 
    \exists m > H m^{*}_{ s, \gamma }:
    \\
    & 
          b_{m}^{2} + 2 c_{m} - s_{m} 
        + 
          \frac{ C_{ \text{AIC} } m \log p }{n} 
    \le 
          b_{ m^{*}_{ s, \gamma } }^{2} + 2 c_{ m^{*}_{ s, \gamma } } 
        -
          s_{ m^{*}_{ s, \gamma } } 
        + 
          \frac{ C_{ \text{AIC} } m^{*}_{ s, \gamma } \log p }{n}.
    \notag
  \end{align}
  Using the bounds from Lemmas \ref{lem_BoundForTheEmpiricalBias},
  \ref{lem_BoundsForTheCrossTermAndTheDiscretizationError} and
  \ref{lem_BoundForTheEmpiricalStochasticError}, on an event with probability
  converging to one, the left-hand side of condition
  \eqref{eq_A_AICConditionTooLate} is larger than \(
    H / 2 \tilde{ \mathcal{R} }( s, \gamma ) 
  \) for \( H \) large enough with 
  \begin{align}
        \tilde{ \mathcal{R} }( s, \gamma ):
    = 
        \begin{dcases}
          \frac{ ( \overline{ \sigma }^{2} + \rho^{4} ) s \log p }{n},
          & \beta^{*} \ s \text{-sparse}, \\
          \Big(
            \frac{ ( \overline{ \sigma }^{2} + \rho^{4} ) \log p }{n}
          \Big)^{ 1 - \frac{1}{ 2 \gamma } },
          & \beta^{*} \ \gamma \text{-sparse}, \\
        \end{dcases}
  \end{align}
  whereas the right-hand side is smaller than \( 
    C \tilde{ \mathcal{R} }( s, \gamma )
  \) with \( C \) independent of \( H \).
  Therefore, \(
    \tau_{ \text{two-step} } > H m^{*}_{ s, \gamma } 
  \) can only be satisfied on an event with probability converging to zero.
  This finishes the proof.
\end{proof}


\section{Proofs for auxiliary results}
\label{sec_ProofsForAuxiliaryResults}

\begin{proof}[Proof of Lemma \ref{lem_BoundsForTheCrossTermAndTheDiscretizationError} (Bounds for the cross term)]
  For (i), without loss of generality, \( b_{m}^{2} > 0 \) for all 
  \( m \ge 0 \).
  We proceed via a supremum-out argument:
  We have
  \begin{align}
          | c_{m} | 
    & =   
          | \langle ( I - \widehat{ \Pi }_{m} ) f^{*}, \varepsilon \rangle_{n} | 
    \le 
          b_{m} \sup_{ h \in \mathcal{H}_{m} } 
                | \langle h, \varepsilon \rangle_{n} | 
  \end{align}
  with 
  \( 
      \mathcal{H}_{m}: 
    =
      \{ 
        \Pi f^{*} / ( \| \Pi f^{*} \|_{n} ): 
        \Pi \text{ is a projection orthogonal to } m \text{ of the } 
        g_{j}, j \le p
      \}
  \).
  Since $ | \mathcal{H}_{m} | \le p^{m} $, we obtain 
  \begin{align}
    & \ \ \ \
            \mathbb{P} \Big\{ 
              \sup_{ m \ge 0 } 
              \frac{ | c_{m} | }{ \sqrt{ ( m + 1 ) b_{m}^{2} } }
              \ge 
              \sqrt{ \frac{ 4 \overline{ \sigma }^{2} \log p }{n} } 
            \Big\} 
    \le 
            \sum_{ m = 0 }^{ \infty }  
            \sum_{ h \in \mathcal{H}_{m} } 
            \mathbb{P} \Big\{ 
              | \langle h, \varepsilon \rangle_{n} | 
              \ge
              \sqrt{ \frac{ 4 \overline{ \sigma }^{2} ( m + 1 ) \log p }{n} } 
            \Big\} 
    \\
    & \le 
            2 \sum_{ m = 0 }^{ \infty }
            p^{m} 
            \exp \Big( 
              \frac{ - 2 n \overline{ \sigma }^{2} ( m + 1 ) \log p }
                   { n \overline{ \sigma }^{2} }
            \Big) 
    \le 
            2 \sum_{ m \ge 0 } 
            p^{ - ( m + 1 ) } 
    =       
            \frac{2}{ 1 - p^{-1} } - 2
            \xrightarrow[ ]{ n \to \infty } 0,
    \notag
  \end{align}
  using a union bound and
  \hyperref[ass_SubGaussianErrors]{\normalfont \textbf{{\color{blue} (SubGE)}}}.

  For (ii), we argue analogously.
  From the definition of Algorithm \ref{alg_OMP} and the Gram-Schmidt
  orthogonalization, we have
  \begin{align}
        r_{ m - 1 }^{2} - r_{m}^{2} 
    & = 
        \Big\langle 
          ( I - \widehat{ \Pi }_{ m - 1 } ) Y,
          \frac{    
            ( I - \widehat{ \Pi }_{ m - 1 } )
            g_{ \widehat{j}_{m} }
          }{ 
            \| 
              ( I - \widehat{ \Pi }_{ m - 1 } ) 
              g_{ \widehat{j}_{m} }
            \|_{n}
          }
        \Big\rangle_{n}^{2}
        \qquad \text{ for all } m \ge 1. 
  \end{align}
  This yields \( 
    \Delta( r_{m}^{2} )  \le 2 b_{ m - 1 } + 2 Z_{ m - 1 }^{2}
  \), where
  \begin{align}
            Z_{m}: 
    & = 
            \Big| \Big\langle 
              \varepsilon, 
              \frac{ 
                ( I - \widehat{ \Pi }_{m} ) g_{ \widehat{j}_{ m + 1 } } 
              }{ 
                \| 
                  ( I - \widehat{ \Pi }_{m} ) g_{ \widehat{j}_{ m + 1 } } 
                \|_{n} 
              }
            \Big\rangle_{n} \Big| 
    \le 
            \sup_{ h \in \tilde{ \mathcal{H} }_{m} } 
            | \langle \varepsilon, h \rangle_{n} |, 
    \qquad m \ge 0 
  \end{align}
  with \( 
    \tilde{ \mathcal{H} }_{m}: = \{ 
      \Pi g_{k} / ( \| \Pi g_{k} \|_{n} ): 
      k \le p,
      \Pi \text{ is a projection orthogonal to } m \text{ of the } 
      g_{j}, j \le p
    \}
  \).
  Since \( | \tilde{ \mathcal{H} }_{m} | \le p^{ m + 1 } \), 
  \begin{align}
    & \ \ \ \
            \mathbb{P} \Big\{ 
              \sup_{ m \ge 1 } 
              \frac{ Z_{m} }{ \sqrt{ m + 1 } } 
              \ge 
              \sqrt{ \frac{ 4 \overline{ \sigma }^{2} \log p }{n} }
            \Big\} 
    \le 
            \sum_{ m = 0 }^{ \infty }  
            \sum_{ h \in \mathcal{H}_{m} } 
            \mathbb{P} \Big\{ 
              | \langle \varepsilon, h \rangle_{n} |
              \ge 
              \sqrt{ \frac{ 4 \overline{ \sigma }^{2} ( m + 1 ) \log p }{n} }
            \Big\} 
    \\
    & \le 
            2
            \sum_{ m = 0 }^{ \infty } 
            p^{ m + 1 } 
            \exp \Big( 
                      \frac{ - 2 n \overline{ \sigma }^{2} ( m + 1 ) \log p }
                           {     n \overline{ \sigma }^{2}                  } 
                 \Big) 
    \le 
            2
            \sum_{ m = 0 }^{ \infty } 
            p^{ - ( m + 1 ) } 
    =     
            \frac{2}{ 1 - p^{-1} } - 2
    \xrightarrow[ ]{ n \to \infty } 0
    \notag
  \end{align}
  as in (i).
  This finishes the proof.
\end{proof}

\begin{proof}[Proof of Lemma \ref{lem_UniformBoundsInHighProbability} (Uniform bounds in high probability)]

\begin{enumerate}

  \item[(i)] From assumption 
    \hyperref[ass_SubGaussianDesign]{\textbf{{\color{blue} (SubGD)}}},
    it is immediate that the \(
      X_{1}^{ (j) }, j \le p
    \) are subgaussian with parameter \( \rho^{2} \).
    Therefore,
    \begin{align}
            \langle g_{j}, g_{k} \rangle_{n} - \langle g_{j}, g_{k} \rangle_{ L^{2} }  
      & = 
            \frac{1}{n} \sum_{ i = 1 }^{n} 
            \big( 
                           X_{i}^{ (j) } X_{i}^{ (k) }
            -
              \mathbb{E} ( X_{i}^{ (j) } X_{i}^{ (k) } )
            \big) 
      \qquad j, k \in \mathbb{N}
    \end{align}
    is an average of centered subexponential variables with parameters 
    \( ( C \rho^{4}, C \rho^{2} ) \), i.e., for \( 
        Z: 
      =
                       X_{i}^{ (j) } X_{i}^{ (k) }
          - \mathbb{E} X_{i}^{ (j) } X_{i}^{ (k) }
    \), 
    \begin{align}
            \mathbb{E} e^{ u Z } 
      & \le 
            e^{ u^{2} C \rho^{4} / 2 } 
      \qquad \text{ for all } | u | \le \frac{1}{ C \rho^{2} }. 
    \end{align}
    From Bernstein's inequality, see Theorem 2.8.1 in Vershynin
    \cite{Vershynin2018HDProbability}, we obtain that for $ t > 0 $, 
    \begin{align}
      & \ \ \ \
            \mathbb{P} \Big\{ 
              \sup_{ j, k \le p } 
              | 
                \langle g_{j}, g_{k} \rangle_{n} 
              - \langle g_{j}, g_{k} \rangle_{ L^{2} } 
              | 
              \ge t
            \Big\} 
      \le 
            \sum_{ j, k \le p } 
            \mathbb{P} \Big\{ 
              | 
                \langle g_{j}, g_{k} \rangle_{n} 
              - \langle g_{j}, g_{k} \rangle_{ L^{2} }  
              | 
              \ge t
            \Big\} 
      \\ \notag 
      & \le 
            2 p^{2} 
            \exp \Big( 
              - c n
              \min \Big[
                \frac{ t^{2} }{ \rho^{4} }, 
                \frac{ t     }{ \rho^{2} }
              \Big]
            \Big) 
      =  
            2 \exp \Big( 
              2 \log p  
            -
              c n \min \Big[
                \frac{ t^{2} }{ \rho^{4} }, 
                \frac{ t     }{ \rho^{2} }
              \Big]
            \Big).
      \notag
    \end{align}
    Setting \( 
      t = C_{g} \sqrt{ \rho^{4} \log(p) / n } 
    \) with \( C_{g} > 0 \) sufficiently large yields the statement in (i),
    since we have assumed that \( \log p = o(n) \). 

  \item[(ii)] By (i), we have that via a union bound,
    \begin{align}
      & \ \ \ \
                \mathbb{P} \Big\{ 
                  \sup_{ j \le p } 
                  | \langle \varepsilon, g_{j} \rangle_{n} | 
                  \ge t
                \Big\} 
      \\ 
      & \le 
                \mathbb{P} \Big\{ 
                  \sup_{ j \le p } 
                  | \langle \varepsilon, g_{j} \rangle_{n} | 
                  \ge t,
                  \sup_{ j \le p } \| g_{j} \|_{n} < \frac{3}{2} 
                \Big\} 
              + 
                \mathbb{P} \Big\{ 
                  \sup_{ j \le p } \| g_{j} \|_{n} \ge \frac{3}{2} 
                \Big\} 
      \notag 
      \\
      & \le 
                \sum_{ j = 1 }^{p} 
                \mathbb{P} \Big\{ 
                  | \langle \varepsilon, g_{j} \rangle_{n} | 
                  \ge t,
                  \sup_{ j \le p } \| g_{j} \|_{n} < \frac{3}{2} 
                \Big\} 
              + 
                o(1) 
      \notag
      \\
      & \le 
                2 p \exp \Big( \frac{ - c n t^{2} }{ \overline{ \sigma }^{2} } \Big) 
              + 
                o(1) 
          = 
                2 
                \exp \Big( 
                  \log p - \frac{ c n t^{2} }{ \overline{ \sigma }^{2} }
                \Big) 
              + 
                o(1), 
      \notag
    \end{align}
    where the last inequality follows from
    \hyperref[ass_SubGaussianErrors]{\textbf{{\color{blue} (SubGE)}}} 
    by conditioning on the design, applying Hoeffding's
    inequality, see Theorem 2.6.2 in Vershynin
    \cite{Vershynin2018HDProbability}, and estimating \( 
      \| g_{j} \|_{n} < 3 / 2
    \) in the denominator of the exponential.
    By choosing \( 
      t = C_{ \varepsilon } \sqrt{ \overline{ \sigma }^{2} \log(p) / n } 
    \) with \( C_{ \varepsilon } > 0 \) large enough, we then obtain the
    statement in (ii).

  \item[(iii)]
    For \( t > 0 \), a union bound yields
    \begin{align}
              \mathbb{P} \Big\{ 
                \sup_{ | J | \le c_{ \text{iter} } n / \log p } 
                \frac{
                  \| \widehat{ \Gamma }_{J} - \Gamma_{J} \|_{ \text{op}}  
                }{
                  \rho^{2}
                } 
                \ge t
              \Big\} 
      & \le 
              \sum_{ | J | \le c_{ \text{iter} } n / \log p } 
              \mathbb{P} \Big\{ 
                \frac{
                  \| \widehat{ \Gamma }_{J} - \Gamma_{J} \|_{ \text{op}}  
                }{
                  \rho^{2}
                } 
                \ge t
              \Big\}. 
    \end{align}
    For any fixed \( J \) with \(
      | J | =  m \le c_{ \text{iter} } n / \log p, 
    \) we can choose a \( 1 / 4 \)-net \( \mathcal{N} \) of the unit ball in \(
    \mathbb{R}^{m} \) with \( | \mathcal{N} | \le 9^{m} \), see Corollary 4.2.13
    in Vershynin \cite{Vershynin2018HDProbability}.
    By an approximation argument, 
    \begin{align}
              \| \widehat{ \Gamma }_{J} - \Gamma_{J} \|_{ \text{op} } 
      & \le 
              2 \max_{ v \in \mathcal{N} } 
              |
                \langle 
                  ( \widehat{ \Gamma }_{J} - \Gamma_{J} ) v, v
                \rangle 
              |
      \\
      & = 
              2 \max_{ v \in \mathcal{N} } 
              | 
                \frac{1}{n} 
                \sum_{ i = 1 }^{n} 
                \langle X_{i}^{ (J) }, v  \rangle^{2} 
              - 
                \mathbb{E} \langle X_{i}^{ (J) }, v  \rangle^{2} 
              |, 
      \notag
    \end{align}
    with \(
      X_{i}^{ (J) } =   ( X_{i}^{ (j) } )_{ j \in J } 
                        \in \mathbb{R}^{ | J | }.
    \) 
    As in (i), the \( 
                   \langle X_{i}^{ (J) }, v \rangle^{2} 
      - 
        \mathbb{E} \langle X_{i}^{ (J) }, v \rangle^{2} 
    \), \( i = 1, \dots, n \), are independent subexponential random variables
    with
    parameters \( ( C \rho^{4}, C \rho^{2} ) \), i.e., by a union bound and
    Bernstein's inequality,
    \begin{align}
              \mathbb{P} \Big\{ 
                \frac{
                  \| \widehat{ \Gamma }_{J} - \Gamma_{J} \|_{ \text{op}}  
                }{
                  \rho^{2}
                } 
                \ge t
              \Big\}
      & \le 
              \sum_{ v \in \mathcal{N} } 
              \mathbb{P} \Big\{ 
                | 
                  \frac{1}{n} 
                  \sum_{ i = 1 }^{n} 
                             \langle X_{i}^{ (J) }, v \rangle 
                  - 
                  \mathbb{E} \langle X_{i}^{ (J) }, v \rangle 
                | 
                \ge \frac{ t^{2} \rho^{2} }{2}
              \Big\} 
      \\
      & \le 
              2 \cdot 9^{m} 
              \exp \big( - c n \min( t^{2}, t ) \big).
      \notag
    \end{align}
    Together, this yields
    \begin{align}
              \mathbb{P} \Big\{ 
                \sup_{ | J | \le c_{ \text{iter} } n / \log p } 
                \frac{
                  \| \widehat{ \Gamma }_{J} - \Gamma_{J} \|_{ \text{op}}  
                }{
                  \rho^{2}
                } 
                \ge t
              \Big\} 
      & \le 
              \sum_{ m = 1 }^{ \lfloor c_{ \text{iter} } n / \log p \rfloor } 
              \sum_{ J: | J | = m } 
              \mathbb{P} \Big\{ 
                \frac{
                  \| \widehat{ \Gamma }_{J} - \Gamma_{J} \|_{ \text{op}}  
                }{
                  \rho^{2}
                } 
                \ge t
              \Big\} 
      \\
      & \le 
              \sum_{ m = 1 }^{ \lfloor c_{ \text{iter} } n / \log p \rfloor } 
              \binom{p}{m} 
              2 \cdot 9^{m} 
              \exp \big( - c n \min( t^{2}, t ) \big) 
      \notag 
      \\
      & \le 
              \frac{ 2 c_{ \text{iter} } n }{ \log p } 
              \exp \Big( 
                \frac{ c_{ \text{iter} } n }{ \log p }
                \log( 9 p )
                -
                  c n \min( t^{2}, t )
              \Big). 
      \notag
    \end{align}
    Setting \(
      t = c_{ \text{iter} } C_{ \Gamma } 
    \) with \( C_{ \Gamma } > 0 \) large enough yields the result.

  \item[(iv)] 
    Set \(
      c_{ \text{iter} } < c_{ \lambda } / ( C_{ \Gamma } \rho^{2} ) 
    \), with \( c_{ \lambda } \) from Assumption 
    \hyperref[ass_CovB]{\normalfont \textbf{{\color{blue} (CovB)}}},
    and consider
    $ 
      Q_{n}: = \big\{ 
                 \forall | J | \le c_{ \text{iter} } n / \log p: 
                 \widehat{ \Gamma }_{J}^{-1} \text{ exists} 
               \big\}.
    $
    For $ n $ large enough, we have
    \begin{align}
              \mathbb{P}( Q_{n} ) 
      & = 
              \mathbb{P} \Big\{ 
                \inf_{ | J | \le c_{ \text{iter} } n / \log p } 
                \lambda_{ \min }( \widehat{ \Gamma }_{J} ) 
                > 0
              \Big\} 
      \\
      & \ge  
              \mathbb{P} \Big\{ 
                \inf_{ | J | \le c_{ \text{iter} } n / \log p } 
                \Big( 
                  \lambda_{ \min }( \Gamma_{J} ) 
                - \| \widehat{ \Gamma }_{J} - \Gamma_{J} \|_{ \text{op} }
                \Big)
                > 0 
              \Big\} 
      \notag
      \\
      & \ge 
              \mathbb{P} \Big\{ 
                \inf_{ | J | \le c_{ \text{iter} } n / \log p } 
                \Big( 
                  c_{ \text{iter} } C_{ \Gamma } \rho^{2} 
                - \| \widehat{ \Gamma }_{J} - \Gamma_{J} \|_{ \text{op} } 
                \Big)
                > 0 
              \Big\} 
              \xrightarrow[ ]{ n \to \infty } 1,
      \notag
    \end{align}
    where we have used (iii) and
    \begin{align}
      \lambda_{ \min }( \widehat{ \Gamma }_{J} ) 
      & \ge 
      \lambda_{ \min }( \Gamma_{J} )
      - 
      \| \widehat{ \Gamma }_{J} - \Gamma_{J} \|_{ \text{op} } 
    \end{align}
    by Weyl's inequality.  Now, let 
    $ 
      F_{n}: 
      = 
      \{ 
        \sup_{ | J | \le c_{ \text{iter} } n / \log p }  
        \| \widehat{ \Gamma }_{J} - \Gamma_{J} \|_{ \text{op} } 
        \le c_{ \text{iter} } C_{ \Gamma } \rho^{2} 
      \}
    $
    be the event from (iii).
    For $ n $ large enough and
    \( 
      C_{ \Gamma^{-1} } > 1 / ( 
                                c_{ \lambda } 
                              - c_{ \text{iter} } C_{ \Gamma } \rho^{2}
                              ),
    \)
    we then have
    \begin{align}
      & \ \ \ \
      \mathbb{P} \Big( 
        \Big\{ 
          \sup_{ | J | \le c_{ \text{iter} } n / \log p } 
          \| \widehat{ \Gamma }_{J}^{-1} \| \le C_{ \Gamma^{-1} }
        \Big\} 
        \cap F_{n} \cap Q_{n}
      \Big) 
      \\
      & \ge 
      \mathbb{P} \Big( 
        \Big\{ 
          \Big( 
            1 
          - 
            \frac{ c_{ \text{iter} } C_{ \Gamma } \rho^{2}  }{ c_{ \lambda } }
          \Big) 
          \sup_{ | J | \le c_{ \text{iter} } n / \log p } 
          \| \widehat{ \Gamma }_{J}^{-1} \|_{ \text{op} } 
          \le 
          \frac{1}{ c_{ \lambda } }
        \Big\}
        \cap F_{n} \cap Q_{n}
      \Big) 
      \notag
      \\
      & \ge 
      \mathbb{P} \Big( 
        \Big\{ 
          \forall | J | \le c_{ \text{iter} } n / \log p:
          \big( 
            1 
          - 
            \| \widehat{ \Gamma }_{J} - \Gamma_{J} \|_{ \text{op} } 
            \| \Gamma_{J}^{-1} \|_{ \text{op} }
          \big) 
          \| \widehat{ \Gamma }_{J}^{-1} \|_{ \text{op} } 
          \le 
          \| \Gamma_{J}^{-1} \|_{ \text{op} }  
        \Big\}
        \cap F_{n} \cap Q_{n}
      \Big) 
      \notag
      \\
      & \ge 
      \mathbb{P}( F_{n} \cap Q_{n} ) 
      \xrightarrow[ ]{ n \to \infty } 1,
      \notag
    \end{align}
    where we have used Banach's Lemma for the inverse in the last inequality,
    which yields that for fixed \( J \), 
    \begin{align}
      \| \widehat{ \Gamma }_{J}^{-1}  \|_{ \text{op} } 
      & = 
      \| 
        \widehat{ \Gamma }_{J}^{-1}
        - 
        \Gamma_{J}^{-1} + \Gamma_{J}^{-1}
      \|_{ \text{op} } 
      \le 
      \frac{ 
        \| \Gamma_{J}^{-1} \|_{ \text{op} }
      }{ 
        1 - \| 
              \Gamma_{J}^{-1} 
              ( \widehat{ \Gamma }_{J} - \Gamma_{J} ) 
            \|_{ \text{op} }  
      } 
    \end{align}
    as long as 
    \( 
      \|
        \Gamma_{J}^{-1} ( \widehat{ \Gamma }_{J} - \Gamma_{J} )
      \|_{ \text{op} }
      < 
      1.
    \)
    Otherwise, the inequality
    \begin{align}
        \big( 
          1 
        - 
          \| \widehat{ \Gamma }_{J} - \Gamma_{J} \|_{ \text{op} } 
          \| \Gamma_{J}^{-1} \|_{ \text{op} }
        \big) 
        \| \widehat{ \Gamma }_{J}^{-1} \|_{ \text{op} } 
        \le 
        \| \Gamma_{J}^{-1} \|_{ \text{op} }  
    \end{align}
    is trivially true, since the left-hand side is negative.

\end{enumerate}

\end{proof}

\begin{corollary}[Reappearing terms]
  \label{cor_ReappearingTerms}
  Under Assumptions
  \hyperref[ass_SubGaussianErrors]{\textbf{{\color{blue} (SubGE)}}},
  \hyperref[ass_SubGaussianDesign]{\textbf{{\color{blue} (SubGD)}}}
  and \hyperref[ass_CovB]{\textbf{{\color{blue} (CovB)}}},
  the following statements hold:
  \begin{enumerate}

    \item[(i)]
      For any $ J \subset \{ 1, \dots, p \} $, $ j' \in J $, $ k \not \in J $, 
      we have 
      \begin{align*}
                | \langle ( I - \Pi_{J} ) g_{k}, g_{ j' } \rangle_{n} | 
        & \le 
                ( 1 + C_{ \text{Cov} } ) \sup_{ j, k \le p } 
                | 
                  \langle g_{j}, g_{k} \rangle_{n} 
                - \langle g_{j}, g_{k} \rangle_{ L^{2} }
                |. 
      \end{align*}
      For \( k \in J \), the left-hand side vanishes.

    \item[(ii)]
      With probability converging to one, we have
      \begin{align*}
            \sup_{ | J | \le M_{n}, k \not \in J } 
            \langle \varepsilon, ( I - \widehat{ \Pi }_{J} ) g_{k} \rangle_{n} 
        \le 
            C
            \sqrt{ \frac{ ( \overline{ \sigma }^{2} + \rho^{4} ) \log p }{n} }.
      \end{align*}

  \end{enumerate}
\end{corollary}

\begin{proof}[Proof]
  \begin{enumerate}

    \item[(i)] Note that for 
      $ 
        j' \in J, k \not \in J, 
      $ 
      by the characterization of the projections in Equation
      \eqref{eq_1_Projections} and Assumption
      \hyperref[ass_CovB]{\textbf{{\color{blue} (CovB)}}}, we have that
      \begin{align}
        & \ \ \ \
                | \langle g_{k} - \Pi_{J} g_{k}, g_{ j' } \rangle_{n} | 
        = 
                \Big| 
                  \Big\langle 
                    g_{k} 
                  - \sum_{ j \in J }
                    \big( 
                      \Gamma_{J}^{-1}
                      \langle g_{k}, g_{J} \rangle_{ L^{2} } 
                    \big)_{j} 
                    g_{j}, 
                    g_{ j' }
                  \Big\rangle_{n} 
                \Big| 
                \\
        & = 
                \Big| 
                  \Big\langle 
                    g_{k} 
                  - \sum_{ j \in J }
                    \big( 
                      \Gamma_{J}^{-1} 
                      \langle g_{k}, g_{J} \rangle_{ L^{2} } 
                    \big)_{j} 
                    g_{j}, 
                    g_{ j' }
                  \Big\rangle_{n} 
        \notag
        \\
        & - 
                  \underbrace{
                  \Big(
                  \langle g_{k}, g_{ j' } \rangle_{ L^{2} } 
                  - \sum_{ j \in J }
                    \big( 
                      \Gamma_{J}^{-1} 
                      \langle g_{k}, g_{J} \rangle_{ L^{2} } 
                    \big)_{j} 
                  \langle g_{j}, g_{ j' } \rangle_{ L^{2} }  
                  \Big)
                  }_{
                    = \langle 
                        ( I - \Pi_{J} ) g_{k}, g_{ j' } 
                      \rangle_{ L^{2} } 
                    = 0 
                  }
                \Big| 
        \notag
        \\
        & \le 
                \| 
                  ( 
                    1,
                    \Gamma_{J}^{-1} 
                    \langle g_{k}, g_{J} \rangle_{ L^{2} } 
                  )^{ \top }
                \|_{1} 
                \sup_{ j, k \le p } 
                | 
                  \langle g_{j}, g_{k} \rangle_{n} 
                - \langle g_{j}, g_{k} \rangle_{ L^{2} }
                | 
        \notag
        \\
        & \le 
                ( 1 + C_{ \text{Cov} } ) 
                \sup_{ j, k \le p } 
                | 
                  \langle g_{j}, g_{k} \rangle_{n} 
                - \langle g_{j}, g_{k} \rangle_{ L^{2} }
                |, 
        \notag 
      \end{align}
      where the second equality follows by the properties of the projection.

    \item[(ii)]
      For $ k \not \in J $, we have
      \begin{align}
                | 
                  \langle
                    \varepsilon, ( I - \widehat{ \Pi }_{J} ) g_{k}  
                  \rangle_{n}
                |
        & \le 
                | \langle \varepsilon, g_{k} \rangle_{n} |
              + 
                | \langle 
                    \varepsilon, ( \widehat{ \Pi }_{J} - \Pi_{J} ) g_{k} 
                \rangle_{n} |
              + 
                | \langle \varepsilon, \Pi_{J} g_{k} \rangle_{n} | 
        \notag 
        \\
        \label{eq_ReappearingTerms_SplitUp}
        & = 
                | \langle \varepsilon, g_{k} \rangle_{n} |
              + 
                |
                \langle 
                    \varepsilon, \widehat{ \Pi }_{J} ( I - \Pi_{J} ) g_{k} 
                \rangle_{n}
                |
              + 
                | \langle \varepsilon, \Pi_{J} g_{k} \rangle_{n} |. 
      \end{align}
      The supremum over the first term in Equation
      \eqref{eq_ReappearingTerms_SplitUp} can be treated immediately by Lemma
      \ref{lem_UniformBoundsInHighProbability} (ii).
      The same is true for the supremum over the last term, since
      \begin{align}
                \sup_{ | J | \le M_{n}, k \not \in J } 
                | \langle \varepsilon, \Pi_{J} g_{k} \rangle_{n} |
        & = 
                \sup_{ | J | \le M_{n}, k \not \in J } 
                |
                  \langle
                    \varepsilon,
                    g_{J}^{ \top } 
                    \Gamma_{J}^{-1}
                    \langle g_{k} , g_{J} \rangle_{ L^{2} } 
                  \rangle_{n}
                |
        \\
        & \le 
                \sup_{ | J | \le M_{n}, k \not \in J } 
                \| \Gamma_{J}^{-1} \langle g_{k}, g_{J} \rangle \|_{1} 
                \sup_{ j \le p } 
                | \langle \varepsilon, g_{j} \rangle_{n} |
                \notag \\
        & \le 
                C_{ \text{Cov} } \sup_{ j \le p } 
                | \langle \varepsilon, g_{j} \rangle_{n} |
        \notag 
      \end{align}
      by the characterization of $ \Pi_{J} $ in Equation \eqref{eq_1_Projections}
      and Assumption \hyperref[ass_CovB]{\textbf{{\color{blue} CovB}}}.
      Finally, the supremum over the middle term in Equation
      \eqref{eq_ReappearingTerms_SplitUp} can be written as
      \begin{align}
        & \ \ \ \
                \sup_{ | J | \le M_{n}, k \not \in J } 
                \langle \varepsilon, g_{J} \rangle_{n}^{ \top } 
                \widehat{ \Gamma }_{J}^{-1} 
                \langle ( I - \Pi ) g_{k}, g_{J} \rangle_{n} 
                \\
        & \le 
                \sup_{ | J | \le M_{n} } 
                \| \widehat{ \Gamma }_{J}^{-1} \|_{ \text{op} } 
                \sup_{ | J | \le M_{n}, k \not \in J } 
                \| \langle ( I - \Pi_{J} ) g_{k},  g_{J} \rangle_{n} \|_{2} 
                \sup_{ j, k \le p } 
                \| \langle \varepsilon, g_{J} \rangle_{n} \|_{2} 
        \notag 
        \\
        & \le 
                \sup_{ | J | \le M_{n} } 
                \| \widehat{ \Gamma }_{J}^{-1} \|_{ \text{op} } 
                M_{n} ( 1 + C_{ \text{Cov} } ) 
                \sup_{ j, k \le p } 
                | 
                  \langle g_{j}, g_{k} \rangle_{n}
                - \langle g_{j}, g_{k} \rangle_{ L^{2} }
                | 
                \sup_{ j \le p } 
                | \langle \varepsilon, g_{j} \rangle_{n} |,
        \notag 
      \end{align}
      where we have used (i) for the second inequality. This term can now also
      be treated by Lemma \ref{lem_UniformBoundsInHighProbability} (i), (ii)
      and (iv).

  \end{enumerate}
\end{proof}

\begin{lemma}[Lower bound for order statistics]
  \label{lem_LowerBoundForOrderStatistics}
  Let \( Z_{1}, \dots, Z_{p} \sim N( 0, \sigma^{2} ) \) i.i.d.
  Then, the order statistic \( Z_{ ( p - m + 1 ) } \) satisfies \( 
        Z_{ ( p - m + 1 ) }
    \ge 
        c \sqrt{ \sigma^{2} \log p }
  \) with probability converging to one for \( p \to \infty \), as long as 
  \( m \le c p^{ \varrho } \) for some \( \varrho \in ( 0, 1 ) \). 
\end{lemma}

\begin{proof}[Proof]
  Without loss of generality, let \( p / m \in \mathbb{N} \). 
  Split the sample into \( m \) groups of size \( p / m \) and for a natural number \( k \le m \), let \( Z^{ (k) } \) denote the maximal value of the \( Z_{j}, j \le p \) which belong to the \( k \)-th group. 
  Then, by a union bound,  
  \begin{align}
    & \ \ \ \
    \mathbb{P} \Big\{ 
      Z_{ ( p - m + 1 ) } \ge \sqrt{ \sigma^{2} \log \Big( \frac{p}{m} \Big) }
    \Big\} 
    \ge 
    \mathbb{P} \Big\{ 
      \min_{ k \le m } Z^{ (k) } 
      \ge
      \sqrt{ \sigma^{2} \log \Big( \frac{p}{m} \Big) }
    \Big\} 
    \\
    & \ge
    1 - 
    m
    \mathbb{P} \Big\{ 
      Z^{ (1) } < \sqrt{ \sigma^{2} \log \Big( \frac{p}{m} \Big) }
    \Big\} 
    = 
    1 - 
    m
    \Big( 
      1 - 
      \mathbb{P} \Big\{ 
        Z^{ (1) } \ge \sqrt{ \sigma^{2} \log \Big( \frac{p}{m} \Big) }
      \Big\} 
    \Big).
    \notag
  \end{align}
  By independence, we can further estimate 
  \begin{align}
    & \ \ \ \
    \mathbb{P} \Big\{ 
      Z^{ (1) } \ge \sqrt{ \sigma^{2} \log \Big( \frac{p}{m} \Big) }
    \Big\} 
    = 
    1 - 
    \Big( 
      1 - 
      \mathbb{P} \Big\{ 
        Z_{1} \ge \sqrt{ \sigma^{2} \log \Big( \frac{p}{m} \Big) }
      \Big\} 
    \Big)^{ \frac{p}{m} } 
    \\
    & \ge 
    1 - 
    \Big( 
      1 - \frac{c}{ \sqrt{ ( p / m ) \log( p / m ) } }
    \Big)^{ \frac{p}{m} } 
    \ge 
    1 - e^{ - c \sqrt{ \frac{p}{m} / \log( \frac{p}{m} ) } }, 
    \notag
  \end{align}
  using the lower Gaussian tail bound from, e.g.,  Proposition 2.1.2 in
  Vershynin \cite{Vershynin2018HDProbability}.
  Together, this yields
  \begin{align}
    \mathbb{P} \Big\{ 
      Z_{ ( p - m + 1 ) } \ge \sqrt{ \sigma^{2} \log \Big( \frac{p}{m} \Big) }
    \Big\} 
    & \ge 
    1 - m e^{ - c \sqrt{ \frac{p}{m} / \log( \frac{p}{m} ) } } 
    \xrightarrow[ ]{ p \to \infty } 1.
  \end{align}
  On the event on the left-hand side, the order statistic satisfies \( 
    Z_{ ( p - m + 1 ) } 
    \ge 
    c \sqrt{ ( 1 - \varrho ) \sigma^{2} \log p }
  \). 
\end{proof}


\section{Proofs for supplementary results}
\label{sec_ProofsForSupplementaryResults}

\begin{proof}[Proof of Lemma \ref{lem_BoundForThePopulationStochasticError} (Bound for the population stochastic error)]
  We define
  \begin{align}
    Q_{m} = \langle Y - \Pi_{m} f^{*}, g_{ \widehat{J}_{m} } \rangle_{n} 
          = \langle \varepsilon, g_{ \widehat{J}_{m} } \rangle_{n} 
          + \langle ( I - \Pi_{m} ) f^{*}, g_{ \widehat{J}_{m} } \rangle_{n}.
  \end{align}
  From Corollary \ref{cor_ReappearingTerms} (i), it follows that 
  \begin{align}
    \label{eq_C_BoundForQm}
          \| Q_{m} \|_{2} ^{2} 
    & \le  
          2
          \Big( 
            m \sup_{ j \le p } 
            \langle \varepsilon, g_{j} \rangle_{n}^{2}
          + 
            ( 1 + C_{ \text{Cov} } )^{2} \| \beta^{*} \|_{1}^{2} 
            m
            \sup_{ j, k \le p } 
            | 
              \langle g_{j}, g_{k} \rangle_{n} 
            - \langle g_{j}, g_{k} \rangle_{ L^{2}( \mathbb{P} ) } 
            |^{2} 
          \Big)
      \in \mathbb{R}^{m} 
  \end{align}
  Note that under \( s \)-sparsity, we can ignore the second term in the
  parentheses when \( m > \tilde m \), since then, \( 
    ( I - \Pi_{ m } ) f^{*} = 0
  \).

  We can express the function $ \widehat{F}^{ (m) } - \Pi_{m} f^{*} $ in terms
  of $ Q_{m} $ via
  \begin{align}
      \widehat{F}^{ (m) } - \Pi_{m} f^{*} 
    =
      g_{ \widehat{J}_{m} }^{ \top } 
      \widehat{ \Gamma }_{ \widehat{J}_{m} }^{-1} 
      \langle Y - \Pi_{m} f^{*}, g_{ \widehat{J}_{m} } \rangle_{n} 
    = 
      ( \widehat{ \Gamma }_{ \widehat{J}_{m} }^{-1} Q_{m} )^{ \top } 
      g_{\widehat{J}_{m}}.
  \end{align}
  Due to the mean zero design, we have
  \begin{align}
    \label{eq_C_FinalBoundForTheStochasticError}
            \| \widehat{F}^{ (m) } - \Pi_{m} f^{*} \|^{2}_{ L^{2} }
    & = 
            \text{Cov}( \widehat{F}^{ (m) } - \Pi_{m} f^{*} ) 
      = 
            Q_{m}^{ \top } 
            \widehat{ \Gamma }_{ \widehat{J}_{m} }^{-1}
            \Gamma_{ \widehat{J}_{m} } 
            \widehat{ \Gamma }_{ \widehat{J}_{m} }^{-1}
            Q_{m}  
    \\
    & = 
            Q_{m}^{ \top } 
            \widehat{ \Gamma }_{ \widehat{J}_{m} }^{-1}
            (
              \widehat{ \Gamma } _{ \widehat{J}_{m} }
            - \widehat{ \Gamma } _{ \widehat{J}_{m} }
            + \Gamma_{ \widehat{J}_{m} }
            )
            \widehat{ \Gamma }_{ \widehat{J}_{m} }^{-1}
            Q_{m}  
    \notag 
    \\
    & \le 
            \| Q_{m} \|_{2}^{2} 
            \| \widehat{ \Gamma }_{ \widehat{J}_{m} }^{-1} \|_{ \text{op} } 
          + \| Q_{m} \|_{2}^{2} 
            \| \widehat{ \Gamma }_{ \widehat{J}_{m} }^{-1} \|_{ \text{op} }^{2} 
            \| 
              \widehat{ \Gamma }_{ \widehat{J}_{m} }
            - \Gamma_{ \widehat{J}_{m} }
            \|_{ \text{op} }.
    \notag 
  \end{align}
  Together, Equations \eqref{eq_C_BoundForQm} and
  \eqref{eq_C_FinalBoundForTheStochasticError} yield the desired result on the
  intersection of the events from Lemma
  \ref{lem_UniformBoundsInHighProbability} (i)-(iv), the probability of which
  converges to one.
\end{proof}

\begin{proof}[Proof of Proposition \ref{prp_BoundForThePopulationBias} (Bound for the population bias)]
  We present the proof under Assumption \hyperref[ass_Sparse]{\normalfont
  \textbf{{\color{blue} (Sparse)}}} (ii).
  Under \hyperref[ass_Sparse]{\normalfont \textbf{{\color{blue} (Sparse)}}}
  (i), the reasoning is analogous.
  The details are discussed in Step 5.

  \textbf{Step 1: Sketch of the arguments.}
  For \( 1 \le k \le p \), we consider the two residual dot product terms
  \begin{align}
    \label{eq_BoundForThePopulationBias_DotProductTerms}
            \mu_{ J, k }: 
    & =
            \langle ( I - \Pi_{J} ) f^{*}, g_{k} \rangle_{ L^{2} } 
      = 
            \sum_{ j \not \in J } 
            \beta_{j}
            \langle g_{j}, ( I - \Pi_{J} ) g_{k} \rangle_{ L^{2} }, 
    \\
            \widehat{ \mu }_{ J, k }: 
    & = 
            \Big\langle
              ( I - \widehat{ \Pi }_{J} ) Y, \frac{ g_{k} }{ \| g_{k} \|_{n} }
            \Big\rangle_{n} 
      = 
            \Big\langle 
              \varepsilon, \frac{ ( I - \widehat{ \Pi }_{J} ) g_{k} } 
                                { \| g_{k} \|_{n}                   } 
            \Big\rangle_{n} 
          + 
            \Big\langle 
              f^{*},
              \frac{ ( I - \widehat{ \Pi }_{J} ) g_{k} }
                   { \| g_{k} \|_{n} } 
            \Big\rangle_{n}. 
    \notag 
  \end{align}
  Note that the choice of the next component in Algorithm \ref{alg_OMP} is based on \( \widehat{ \mu }_{ J, k } \) and \( \mu_{ J, k } \) is its population counterpart.

  In Step 4, we show that
  \begin{align}
    \mathbb{P}( A_{n}( M_{n} )^{c} ) & \xrightarrow[ ]{ n \to \infty } 0, 
    \text{ where } 
    \\ 
        A_{n}(m)
    & = 
        \Big\{ 
          \sup_{ | J | \le m, j \le p } 
          | \widehat{ \mu }_{ J, j } - \mu_{ J, j } | 
          \le \tilde C \sqrt{ \frac{ ( \sigma^{2} + \rho^{4} ) \log p }{n} }
        \Big\}
        \notag 
  \end{align}
  for some constant \( C > 0 \). 
  Then, for any $ \xi \in ( 0, 1 ) $, we set
  \begin{align}
        B_{n}(m): 
    = 
        \Big\{ 
          \inf_{ l \le m } \sup_{ j \le p } 
          | \mu_{ \widehat{J}_{l}, j } | 
          > \frac{ 2 \tilde C }{ 1 - \xi }
            \sqrt{ \frac{ ( \sigma^{2} + \rho^{4} ) \log p }{n} } 
        \Big\}. 
  \end{align}
  In Step 2 of the proof, we show that on $ A_{n}(m) \cap B_{n}(m) $,
  \begin{align}
        | \mu_{ \widehat{J}_{l}, \widehat{j}_{ l + 1 } } | 
    \ge 
        \xi \sup_{ j \le p } 
        | \mu_{ \widehat{J}_{l}, j } | 
    \qquad \text{ for all } l \le m,
  \end{align}
  which implies 
  \begin{align}
    \label{eq_BoundForThePopulationBias_RecursiveBoundForTheBias}
          \| ( I - \Pi_{ l + 1 } ) f^{*} \|_{ L^2( \mathbb{P} ) }^{2}
    & \le
          \| ( I - \Pi_{l} ) f^{*} \|_{ L^2( \mathbb{P} ) }^{2}
          \Big( 
            1 
            -
            c \big( 
              \| ( I - \Pi_{l} ) f^{*} \|_{ L^2( \mathbb{P} ) }^{2}
            \big)^{ 1 / ( 2 \gamma - 1 ) } 
          \Big) 
    \\
    &     \qquad \qquad \qquad \qquad \qquad \qquad \qquad \qquad 
          \text{ for all } l \le m - 1.
    \notag
  \end{align}
  Lemma 1 from Gao et al. \cite{GaoEtal2013LqHulls} states that any sequence $ ( a_{m} )_{ m \in \mathbb{N}_{0} } $ for which there exist $ A, c > 0 $ 
  and $ \alpha \in ( 0, 1 ] $ with 
  \begin{align}
    a_{0} \le A 
    \qquad \text{ and } \qquad 
    a_{ m + 1 } \le a_{m} ( 1 - c a_{m}^{ \alpha } )
    \qquad \text{ for all } m = 0, 1, 2, \dots 
  \end{align}
  satisfies 
  \begin{align}
            a_{m} 
    & \le 
            \max \{ 2^{ 1 / \alpha^{2} } ( c \alpha )^{ - 1 / \alpha }, A \} 
            m^{ - 1 / \alpha }
    \qquad \text{ for all } m = 0, 1, 2, \dots 
  \end{align}
  Applying this to 
  \(
    ( \| ( I - \Pi_{m} ) f \|_{ L^{2} }^{2} )_{ m \in \mathbb{N}_{0} } 
  \)
  then yields 
  \begin{align}
        \| ( I - \Pi_{m} ) f^{*} \|_{ L^2( \mathbb{P} ) }^{2}
    \le 
        C m^{ 1 - 2 \gamma }
    \qquad \text{ on } A_{n}(m) \cap B_{n}(m). 
  \end{align}
  In Step 3, we independently show that
  \begin{align}
          \| ( I - \Pi_{m} ) f^{*} \|_{ L^2( \mathbb{P} ) }^{2} 
    & \le 
          C 
          \Big(
            \frac{ ( \overline{ \sigma }^{2} + \rho^{4}  ) \log p }{n} 
          \Big)^{ 1 - 1 / ( 2 \gamma ) }
          \qquad \text{on } B_{n}(m)^{c}. 
  \end{align}
  This establishes that on \( A_{n}( M_{n} ) \), 
  \begin{align}
    \label{eq_BoundForThePopulationBias_ResultForMLeSqrtKn}
          \| ( I - \Pi_{m} ) f^{*} \|_{ L^2( \mathbb{P} ) }^{2}
    & \le 
          C
          \Big( 
            m^{ 1 - 2 \gamma } 
          + 
            \Big(
              \frac{ ( \overline{ \sigma }^{2} + \rho^{4}  ) \log p }{n} 
            \Big)^{ 1 - 1 / ( 2 \gamma ) }
          \Big)
          \qquad \text{ for all } m \le M_{n}. 
  \end{align}
  Finally, the monotonicity of 
  \( 
    m \mapsto \| ( I - \Pi_{m} ) f^{*} \|_{ L^2( \mathbb{P} ) }^{2}
  \) 
  yields that the estimate in Equation
  \eqref{eq_BoundForThePopulationBias_ResultForMLeSqrtKn} also holds for 
  \( m > M_{n} \) as \( m^{ 1 - 2 \gamma } \) then becomes a lower
  order term.

  \textbf{Step 2: Analysis on $ A_{n}(m) \cap B_{n}(m) $.} 
  On $ A_{n}(m) \cap B_{n}( m  ) $, for any $ l \le m $, we have
  \begin{align}
    \label{eq_BoundForThePopulationBias_WOGACondition}
            | \mu_{ \widehat{J}_{l}, \widehat{j}_{l + 1} } | 
    & \ge 
            | \widehat{ \mu }_{ \widehat{J}_{l}, \widehat{j}_{l + 1} } | 
          - 
            | 
              \widehat{ \mu }_{ \widehat{J}_{l}, \widehat{j}_{l + 1} }
            - \mu_{ \widehat{J}_{l}, \widehat{j}_{l + 1} }
            | 
    \ge 
            | \widehat{ \mu }_{ \widehat{J}_{l}, \widehat{j}_{l + 1} } | 
          - 
            \sup_{ | J | \le m, j \le p } 
            | \widehat{ \mu }_{ J, j } - \mu_{ J, j } | 
            \\
    & \ge 
            \sup_{ j \le p } 
            | \widehat{ \mu }_{ \widehat{J}_{j}, j } | 
          - \tilde C \sqrt{ \frac{ ( \sigma^{2} + \rho^{4} ) \log p }{n} }  
    \ge 
            \sup_{ j \le p } 
            | \mu_{ \widehat{J}_{j}, j } | 
          - 2 \tilde C \sqrt{ \frac{ ( \sigma^{2} + \rho^{4} ) \log p }{n} }  
    \notag 
    \\
    & \ge   
          \xi \sup_{ j \le p } 
          | \mu_{ \widehat{J}_{j}, j } |, 
    \notag 
  \end{align}
  where for the third inequality, we have used the definition of Algorithm
  \ref{alg_OMP} and $ A_{n}(m) $.
  For the final inequality, we have used the definition of $ B_{n}(m) $. 

  Using the orthogonality of the projection \( \Pi_{m} \), we can always estimate
  \begin{align}
    \label{eq_BoundForThePopulationBias_OrthogonalBiasEstimate}
        \| ( I - \Pi_{l} ) f^{*} \|_{ L^{2} }^{2} 
    & = 
        \Big\langle 
          ( I - \Pi_{l} ) f^{*}, 
          \sum_{ j = 1 }^{p} \beta^{*}_{j} g_{j}
        \Big\rangle_{ L^{2} } 
    = 
        \sum_{ j \not \in \widehat{J}_{l} } 
        \beta^{*}_{j} 
        \langle ( I - \Pi_{l} ) f^{*}, g_{j} \rangle_{ L^{2} } 
    \\
    & \le 
        \sup_{ j \le p } 
        \langle ( I - \Pi_{l} ) f^{*}, g_{j} \rangle_{ L^{2} } 
        \sum_{ j \not \in \widehat{J}_{l} } 
        | \beta^{*}_{j} | 
    =
        \sup_{ j \le p } | \mu_{ \widehat{J}_{l}, j } | 
        \sum_{ j \not \in \widehat{J}_{l} } 
        | \beta^{*}_{j} |. 
    \notag
  \end{align}
  Further,
  \begin{align}
    \label{eq_BoundForThePopulationBias_LowerBiasBound}
          \| ( I - \Pi_{l} ) f^{*} \|_{ L^{2} }^{2} 
    & = 
          ( \beta^{*} - \beta( \Pi_{l} f^{*} ) )^{ \top } 
          \Gamma 
          ( \beta^{*} - \beta( \Pi_{l} f^{*} ) ) 
    \\
    & \ge 
          \lambda_{ \min }( \Gamma ) 
          \| ( \beta^{*} - \beta( \Pi_{l} f^{*} ) ) \|_{2}^{2} 
      \ge 
          c_{ \lambda }
          \sum_{ j \not \in \widehat{J}_{l} } 
          | \beta^{*}_{j} |^{2}, 
    \notag
  \end{align}
  where in the last step, we have used that 
  \( 
    \beta( \Pi_{l} f^{*} )_{l} = 0
  \) 
  for all \( l \not \in \widehat{J}_{l} \).
  Together with the \( \gamma \)-sparsity Assumption, Equations
  \eqref{eq_BoundForThePopulationBias_OrthogonalBiasEstimate} and
  \eqref{eq_BoundForThePopulationBias_LowerBiasBound} yield 
  \begin{align}
    \label{eq_BoundForThePopulationBias_CombinationBound}
          \| ( I - \Pi_{l} ) f^{*} \|_{ L^{2} }^{2} 
    & \le 
          \sup_{ j \le p } | \mu_{ \widehat{J}_{l}, j } | 
          \sum_{ j \not \in \widehat{J}_{l} } 
          | \beta_{j}^{*} | 
    \le 
          \sup_{ j \le p } | \mu_{ \widehat{J}_{l}, j } | 
          \Big( 
            \sum_{ j \not \in \widehat{J}_{l} } 
            | \beta_{j}^{*} |^{2}
          \Big)^{ ( \gamma - 1 ) / ( 2 \gamma - 1 ) } 
    \\
    & \le 
          C \sup_{ j \le p } | \mu_{ \widehat{J}_{l}, j } | 
          \big( 
            \| ( I - \Pi_{l} ) f^{*} \|_{ L^{2} }^{2} 
          \big)^{ ( \gamma - 1 ) / ( 2 \gamma - 1 ) }.
    \notag
  \end{align}

  Since the inequality in \eqref{eq_BoundForThePopulationBias_WOGACondition} holds on 
  \( A_{n}(m) \cap B_{n}(m) \), we have that for any \( l \le m - 1 \), 
  \begin{align}
          \| ( I - \Pi_{ l + 1 } ) f^{*} \|_{ L^{2} }^{2} 
    & \le 
          \| 
            ( I - \Pi_{l} ) f^{*} 
          - 
            \mu_{ \widehat{J}_{m}, \widehat{j}_{ l + 1 } } 
            g_{ \widehat{j}_{ l + 1 } } 
          \|_{ L^{2} }^{2} 
          =  
          \| ( I - \Pi_{l} ) f^{*} \|_{ L^{2} }^{2} 
          - 
          \mu_{ \widehat{J}_{l}, \widehat{j}_{ l + 1 } }^{2} 
    \notag
    \\
    & \le 
          \| ( I - \Pi_{l} ) f^{*} \|_{ L^{2} }^{2} 
          - 
          \xi^{2}
          \sup_{ j \le p } \mu_{ \widehat{J}_{m}, j }^{2} 
    \\
    & \le 
          \| ( I - \Pi_{l} ) f^{*} \|_{ L^{2} }^{2} 
          - 
          c \big( 
            \| ( I - \Pi_{l} ) f^{*} \|_{ L^{2} }^{2} 
          \big)^{ 2 \gamma / ( 2 \gamma - 1 ) } 
    \notag 
    \\
    & \le 
          \| ( I - \Pi_{l} ) f^{*} \|_{ L^{2} }^{2} 
          \Big( 
            1 
            - 
            c \big( 
              \| ( I - \Pi_{l} ) f^{*} \|_{ L^{2} }^{2}
            \big)^{ 1 / ( 2 \gamma - 1 ) }
          \Big). 
    \notag
  \end{align}

  \textbf{Step 3: Analysis on \( A_{n}(m) \cap B_{n}(m)^{c} \).} 
  Equation \eqref{eq_BoundForThePopulationBias_CombinationBound} implies that
  \begin{align}
          \| ( I - \Pi_{l} ) f^{*} \|_{ L^{2} }^{2} 
    & \le 
          C \Big( 
            \sup_{ j \le p } | \mu_{ \widehat{J}_{l}, j } |
          \Big)^{ 2 - 1 / \gamma }. 
  \end{align}
  On $ B_{n}(m)^{c} $, we therefore obtain by the monotonicity of 
  \( 
    m \mapsto \| ( I - \Pi_{m} ) f^{*} \|_{ L^{2} }^{2}
  \) 
  that
  \begin{align}
    & \ \ \ \
            \| ( I - \Pi_{m} ) f^{*} \|_{ L^{2} }^{2} 
      = 
            \inf_{ l \le m } 
            \| ( I - \Pi_{l} ) f^{*} \|_{ L^{2} }^{2} 
    \\
    & \le 
            C \Big( 
              \min_{ l \le m } 
              \sup_{ j \le p } 
              | \mu_{ \widehat{J}_{l}, j } |
            \Big)^{ 2 - 1 / \gamma } 
    \le     
            C 
            \Big( 
              \frac{ ( \sigma^{2} + \rho^{4} ) \log p }{n}
            \Big)^{ 1 - 1 / ( 2 \gamma ) }.
    \notag
  \end{align}

  \textbf{Step 4: $ \mathbb{P}( A_{n}( M_{n} )^{c} ) \to 0 $.} 
  Since for \( k \in J \), \( \widehat{ \mu }_{ J, k } = \mu_{ J, k } = 0, \) we only need to consider the case \( k \not \in J \).
  For \( | J | \le M_{n} \), we may write
  \begin{align}
    \label{eq_BoundForThePopulationBias_DecompositionOfTheDotProductDifference}
          \widehat{ \mu }_{ J, k } - \mu_{ J, k }
    & = 
          \Big\langle 
            \varepsilon, 
            \frac{ ( I - \widehat{ \Pi }_{J} ) g_{k} }
                 { \| g_{k} \|_{n} } 
          \Big\rangle_{n} 
          \\
    & + 
          \sum_{ j \not \in J } 
          \beta_{j} 
          \Big[ 
            \frac{ 
                   \langle
                     g_{j}, ( I - \widehat{ \Pi }_{J} ) g_{k} 
                   \rangle_{n}
                 }
                 { \| g_{k} \|_{n} }
          - \langle g_{j}, ( I - \Pi_{J} ) g_{k} \rangle_{ L^{2} } 
          \Big]. 
    \notag 
  \end{align}

  From Lemma \ref{lem_UniformBoundsInHighProbability} (i) and Corollary
  \ref{cor_ReappearingTerms} (ii), we obtain an event with probability
  converging to one on which
  \begin{align}
            \Big\langle 
              \varepsilon, 
              \frac{ ( I - \widehat{ \Pi }_{J} ) g_{k} }
                   { \| g_{k} \| } 
            \Big\rangle_{n} 
    & \le 
            2 
            \langle 
              \varepsilon, ( I - \widehat{ \Pi }_{J} ) g_{k}
            \rangle_{n} 
      \le 
            C \sqrt{ \frac{ ( \sigma^{2} + \rho^{4} ) \log p }{n} }. 
  \end{align}

  Further, for \( j, k \not \in J \), we can estimate
  \begin{align}
    & \ \ \ \
            \frac{ 
                   \langle
                     g_{j}, ( I - \widehat{ \Pi }_{J} ) g_{k} 
                   \rangle_{n}
                 }
                 { \| g_{k} \|_{n} }
          - 
            \langle g_{j}, ( I - \Pi_{J} ) g_{k} \rangle_{ L^{2} } 
    \\
    & \le 
            \Big| \frac{1}{ \| g_{k} \|_{n} } - 1 \Big| 
            \| g_{j} \|_{n} \| g_{k} \|_{n} 
          + 
            | 
              \langle g_{j}, ( I - \widehat{ \Pi }_{J} ) g_{k} \rangle_{n}
              - 
              \langle
                g_{j}, ( I - \Pi_{J} ) g_{k}
              \rangle_{ L^{2} } 
            |. 
    \notag
  \end{align}
  From Lemma \ref{lem_UniformBoundsInHighProbability} (i), we obtain an event with probability converging to one on which
  \begin{align}
    & \ \ \ \
            \Big| \frac{1}{ \| g_{k} \|_{n} } - 1 \Big| 
            \| g_{j} \|_{n} \| g_{k} \|_{n} 
      \le 
            | \| g_{k} \|_{n} - 1  | \| g_{j} \|_{n}
    \\
    & \le 
            | \| g_{k} \|_{n}^{2} - 1  | \| g_{j} \|_{n}
    \le 
            C \sqrt{ \frac{ \rho^{4} \log p }{n} }. 
    \notag
  \end{align}
  Additionally, 
  \begin{align}
    & \ \ \ \
            | 
              \langle g_{j}, ( I - \widehat{ \Pi }_{J}  ) g_{k} \rangle_{n}
            - \langle g_{j}, ( I - \Pi_{J}              ) g_{k} \rangle_{ L^2 } 
            | 
    \\
    & \le 
            | 
              \langle g_{j}, g_{k} \rangle_{n} 
            - \langle g_{j}, g_{k} \rangle_{ L^2 } 
            | 
          + 
            | \langle 
              g_{j}, ( \widehat{ \Pi }_{J} - \Pi_{J} ) g_{k} 
            \rangle_{n} | 
          + 
            | 
              \langle g_{j}, \Pi_{J} g_{k} \rangle_{n} 
            - \langle g_{j}, \Pi_{J} g_{k} \rangle_{ L^2 } 
            | 
    \notag 
    \\
    & \le 
            ( C_{ \text{Cov} } + 1 ) 
            \sup_{ j, k \le p } 
            | 
              \langle g_{j}, g_{k} \rangle_{n} 
            - \langle g_{j}, g_{k} \rangle
            | 
          +  
            | \langle
              g_{j}, \widehat{ \Pi }_{J} ( I - \Pi_{J} ) g_{k} 
            \rangle_{n} | 
    \notag
  \end{align}
  by Assumption \hyperref[ass_CovB]{\textbf{{\color{blue} (CovB)}}}.
  The first term can be treated by Lemma
  \ref{lem_UniformBoundsInHighProbability} (i) again.
  Using the representation of \( \widehat{ \Pi }_{J} \) in Equation
  \eqref{eq_1_Projections}, the second term can be estimated against
  \begin{align}
    & \ \ \ \
            | \langle
              ( I - \Pi_{J} ) g_{j},
              \widehat{ \Pi }_{J} ( I - \Pi_{J} ) g_{k} 
            \rangle_{n} | 
          + 
            | \langle 
              \Pi_{J} g_{j}, 
              \widehat{ \Pi }_{J} ( I - \Pi_{J} ) g_{k} 
            \rangle_{n} | 
    \\
    & = 
            | 
              \langle ( I - \Pi_{J} ) g_{j}, g_{J} \rangle_{n}^{ \top } 
              \widehat{ \Gamma }_{J}^{-1} 
              \langle ( I - \Pi_{J} ) g_{k}, g_{J} \rangle_{n}
            | 
          + 
            | \langle 
              \Pi_{J} g_{j}, ( I - \Pi_{J} ) g_{k} 
            \rangle_{n} | 
    \notag
    \\
    & \le 
            \| \widehat{ \Gamma }_{J}^{-1} \|_{ \text{op} } 
            \| 
              \langle ( I - \Pi_{J} ) g_{j}, g_{J} \rangle_{n}^{ \top } 
            \|_{2}^{2} 
          + 
            | 
              \langle 
                g_{J}^{ \top } \Gamma_{J}^{-1} \langle g_{j}, g_{J}\rangle, 
                ( I - \Pi_{J} ) g_{k} 
              \rangle_{n} 
            | 
    \notag
    \\
    & \le 
            \| \widehat{ \Gamma }_{J}^{-1} \|_{ \text{op} } M_{n}
            \sup_{ j \in J, k \not \in J } 
            | \langle ( I - \Pi_{J} ) g_{k}, g_{j} \rangle_{n} |^{2} 
          + 
            C_{ \text{Cov} } 
            \sup_{ j \in J, k \not \in J } 
            | \langle ( I - \Pi_{J} ) g_{k}, g_{j} \rangle_{n} |.
    \notag
  \end{align}
  Analogously to before, the remaining terms can now be treated by Lemma
  \ref{lem_UniformBoundsInHighProbability} (iv) and Corollary
  \ref{cor_ReappearingTerms} (i).  The result now follows by intersecting all
  the events and taking the supremum in Equation
  \eqref{eq_BoundForThePopulationBias_DecompositionOfTheDotProductDifference}. 

  \textbf{Step 5: \( s \)-sparse setting.}
  Under Assumption \hyperref[ass_Sparse]{\normalfont \textbf{{\color{blue}
  (Sparse)}}} (i), the general argument from Step 1 is the same. 
  However, we need to consider the events
  \begin{align}
        A_{n}(m): 
    & = 
        \Big\{ 
          \sup_{ | J | \le m, k \le p } 
          | \widehat{ \mu }_{ J, k } - \mu_{ J, k }  | 
          \le 
          C \sqrt{ 
            \frac{
              ( \overline{ \sigma }^{2} + \| \beta^{*} \|_{1} \rho^{4} ) 
              s \log p
            }{n}
          }
        \Big\},
    \\
        B_{n}(m): 
    & = 
        \Big\{ 
          \inf_{ l \le m } \sup_{ l \le p } 
          | \mu_{ \widehat{J}_{l}, j } | 
          > 
          \frac{ 2 C }{ 1 - \xi } 
          \sqrt{ 
            \frac{
              ( \overline{ \sigma }^{2} + \| \beta^{*} \|_{1} \rho^{4} ) 
              s \log p
            }{n}
          }
        \Big\}.
    \notag
  \end{align}
  On \( A_{n}(m) \cap B_{n}(m)^{c} \), the analysis is exactly the same as before.
  In Step 4, we can also argue the same. 
  We merely have to account for the coefficients of \( f^{*} \) by a factor
  \( \| \beta^{*} \|_{1} \), instead of shifting them into the constant.

  On \( A_{n}(m) \cap B_{n}(m) \), we obtain that for any \( l \le m \),
  \begin{align}
            \| ( I - \Pi_{l} ) f^{*} \|_{ L^{2} }^{2} 
    & \le 
            \sup_{ | J | \le l, j \le p } | \mu_{ J, j } |
            \sum_{ j \not \in \widehat{J}_{l} } 
            | \beta^{*}_{j} | 
    \le 
            \sup_{ | J | \le l, j \le p } | \mu_{ J, j } |
            \Big( 
              s
              \sum_{ j \not \in \widehat{J}_{l} } 
              | \beta^{*}_{j} |^{2}
            \Big)^{ \frac{1}{2} }. 
  \end{align}
  Since additionally, \( 
        \sum_{ j \not \in \widehat{J}_{l} } 
        | \beta^{*}_{j} |^{2}
    \le 
        c \| ( I - \Pi_{l} ) f^{*} \|_{ L^{2} }^{2}
  \), we have
  \begin{align}
          \sup_{ | J | \le l, j \le p } \mu_{ J, j }^{2}
    & \ge
          \frac{c}{s} 
          \| ( I - \Pi_{l} ) f^{*} \|_{ L^{2} }^{2}.
  \end{align}
  Recursively, this yields
  \begin{align}
          \| ( I - \Pi_{m} ) f^{*} \|_{ L^{2} }^{2} 
    & \le 
          \| 
            ( I - \Pi_{ m - 1 } ) f^{*} 
          - 
            \mu_{ \widehat{J}_{ m - 1 }, \widehat{j}_{m} } 
            g_{ \widehat{j}_{m} } 
          \|_{ L^{2} }^{2} 
        =
          \| ( I - \Pi_{ m - 1 } ) f^{*} \|_{ L^{2} }^{2} 
          - 
          \mu_{ \widehat{J}_{ m - 1 }, \widehat{j}_{m} }^{2} 
    \notag
    \\
    & \le 
          \| ( I - \Pi_{ m - 1 } ) f^{*} \|_{ L^{2} }^{2} 
          - 
          \xi^{2}
          \sup_{ j \le p } \mu_{ \widehat{J}_{ m - 1 }, j }^{2} 
      \le 
          \| ( I - \Pi_{ m - 1 } ) f^{*} \|_{ L^{2} }^{2} 
          \Big( 1 - \frac{c}{s} \Big) 
    \\
    & \le 
          \| f^{*} \|_{ L^{2} }^{2} 
          \exp \Big( \frac{ - c m }{s} \Big).
    \notag
  \end{align}

  Finally, as long as \( S \not \subset \widehat{J}_{m} \), 
  \begin{align}
            \| ( I - \Pi_{m} ) f^{*} \|_{ L^{2} }^{2} 
    & \ge 
            \beta( ( I - \Pi_{m} ) f^{*} )^{ \top }
            \Gamma 
            \beta( ( I - \Pi_{m} ) f^{*} )
      \ge 
            c_{ \lambda } \underline{ \beta }^{2}.
  \end{align}
  However, for \( m = C s \) with some \( C > 0 \) and \( n \in \mathbb{N} \)
  large enough, 
  \begin{align}
            \| ( I - \Pi_{m} ) f^{*} \|_{ L^{2} }^{2} 
    & \le 
            \frac{ c_{ \lambda } }{4} \underline{ \beta }^{2} 
          + 
            \sqrt{ 
              \frac{
                ( \overline{ \sigma }^{2} + \| \beta^{*} \|_{1} \rho^{4} ) 
                s \log p
              }{n}
            }
    \le 
            \frac{ c_{ \lambda } }{2} \underline{ \beta }^{2} 
  \end{align}
  with probability converging to one.
  This yields the last claim of Proposition \ref{prp_BoundForThePopulationBias}. 
\end{proof}

\begin{proof}[Proof of Proposition \ref{prp_FastNormChangeForTheBias} (Fast norm change for the bias)]
  For a fixed \( m \le M_{n} \), let \( \tilde \beta \) be the coefficients of
  \( ( I - \Pi_{m} ) f^{*} \). 
  For \( \gamma \)-sparse \( \beta^{*} \), we then have
  \begin{align}
    \label{eq_FastNormChangeForTheBias1}
    & \ \ \ \
            | 
              \| ( I - \Pi_{m} ) f^{*} \|_{n}^{2} 
              - 
              \| ( I - \Pi_{m} ) f^{*} \|_{ L^2 }^{2}
            | 
    = 
            | 
              \sum_{ j, k = 1 }^{p} 
              \tilde \beta_{j} 
              \tilde \beta_{k} 
              ( 
                \langle g_{j}, g_{k} \rangle_{n} 
                - 
                \langle g_{j}, g_{k} \rangle_{ L^2 }
              )
            | 
    \\
    & \le 
            \| \tilde \beta \|_{1}^{2}
            \sup_{ j, k \le p } 
            | 
              \langle g_{j}, g_{k} \rangle_{n} 
              - 
              \langle g_{j}, g_{k} \rangle_{ L^2 }
            | 
    \notag
    \\
    & \le 
            ( C_{ \text{Cov} } + 1 )
            \Big( 
              \sum_{ j \not \in \widehat{J}_{m} } | \beta^{*}_{j} | 
            \Big)^{2} 
            \sup_{ j, k \le p } 
            | 
              \langle g_{j}, g_{k} \rangle_{n} 
            - 
              \langle g_{j}, g_{k} \rangle_{ L^2 }
            | 
    \notag
    \\
    & \le 
            ( C_{ \text{Cov} } + 1 )
            C_{ \gamma } 
            \Big( 
              \sum_{ j \not \in \widehat{J}_{m} } | \beta^{*}_{j} |^{2} 
            \Big)^{ \frac{ 2 \gamma - 2 }{ 2 \gamma - 1 } }
            \sup_{ j, k \le p } 
            | 
              \langle g_{j}, g_{k} \rangle_{n} 
            - 
              \langle g_{j}, g_{k} \rangle_{ L^2 }
            |, 
    \notag 
  \end{align}
  where the second inequality follows from the uniform Baxter inequality in
  \eqref{eq_1_UniformBaxtersInequality}.
  Additionally, we have
  \begin{align}
            \| ( I - \Pi_{m} ) f^{*} \|_{ L^2 }^{2} 
    & = 
            \tilde \beta^{ \top } \Gamma \tilde \beta 
    \ge 
            \lambda_{ \min }( \Gamma ) 
            \| \tilde \beta \|_{2}^{2} 
    \ge 
            c_{ \lambda } 
            \sum_{ j \not \in \widehat{J}_{m} } 
            | \beta^{*}_{j} |^{2}.
  \end{align}
  Plugging this inequality into Equation \eqref{eq_FastNormChangeForTheBias1}
  yields the result.
  For \( s \)-sparse \( \beta^{*} \), the statement in Proposition
  \ref{prp_FastNormChangeForTheBias} is obtained analogously by using the
  inequality
  \( 
    \| \tilde \beta \|_{1}^{2} 
    \le 
    ( s + m ) \| \tilde \beta \|_{2}^{2}
  \). 
  It follows from the fact that the projection \( ( I - \Pi_{m} ) \) adds at
  most \( m \) components to the support of \( \beta^{*} \).
\end{proof}

\begin{lemma}[Uniform Baxter's inequality, Ing \cite{Ing2020ModelSelection}]
  \label{lem_UniformBaxtersInequality}
  Under Assumption \hyperref[ass_CovB]{\textbf{{\color{blue} (CovB)}}},
  for any $ J \subset \{ 1, \dots, p \} $ with 
  \( | J | \le M_{n} \), we have
  \begin{align*}
            \| \beta( ( I - \Pi_{J} ) f^{*} )  \|_{1}
    & \le 
            ( C_{ \text{Cov} } + 1 ) 
            \sum_{ j \not \in J } | \beta^{*}_{j} |,
  \end{align*}
  where \( \tilde \beta_{J} \) denotes the coefficients of the population 
  residual term \( ( I - \Pi_{J} ) f^{*} \).
\end{lemma}

\begin{proof}[Proof]
  For any \( J \subset \{ 1, \dots, p \} \) as above, we have
  \begin{align}
            \| \beta( ( I - \Pi_{J} ) f^{*} ) \|_{1}
    & = 
            \| \beta^{*} - \beta( \Pi_{J} f^{*} ) \|_{1} 
    \le 
            \| \beta^{*}_{J} - \beta( \Pi_{J} f^{*} ) \|_{1} 
          + 
            \sum_{ j \not \in J } | \beta^{*}_{j} | 
    \\
    & \le 
            ( C_{ \text{Cov} } + 1 ) 
            \sum_{ j \not \in J } | \beta^{*}_{j} |, 
    \notag
  \end{align}
  where for the last inequality, we have used that
  \begin{align}
    & \ \ \ \
            \beta^{*}_{J} - \beta( \Pi_{J} f^{*} )
    = 
            \Gamma_{J}^{-1} \Gamma_{J} 
            ( 
              \beta^{*}_{J} 
            -
              \Gamma_{J}^{-1} \langle f^{*}, g_{J} \rangle_{ L^{2} } 
            ) 
    \\
    & = 
            \Gamma_{J}^{-1} 
            \Big( 
              \Gamma_{J} \beta^{*}_{J} 
            -
              \sum_{ j = 1 }^{p} 
              \beta^{*}_{j} 
              \langle g_{j}, g_{J} \rangle_{ L^{2} } 
            \Big) 
    = 
          - \sum_{ j \not \in J } 
            \beta^{*}_{j} 
            \Gamma_{J}^{-1} 
            \langle g_{j}, g_{J} \rangle_{ L^{2} } 
    \notag
  \end{align}
  and condition \eqref{eq_1_CovB_2} in its formulation from Section
  \ref{ssec_FurtherNotation}. 
\end{proof}


\section{Simulation study}
\label{sec_SimulationStudy}

In this section, we provide additional simulation results.
We begin by displaying the boxplots of the stopping times for the different
scenarios from Section \ref{sec_NumericalSimulationsAndATwoStepProcedure}. 
In order to indicate whether stopping happened before or after the classical
oracle \( 
  m^{ \mathfrak{o} } 
  = 
  \argmin_{ m \ge 0 } \| \widehat{F}^{ (m) } - f^{*} \|_{n}^{2}
\), we report the difference \( 
  \tau - m^{ \mathfrak{o} }
\) or \( 
  \tau_{ \text{two-step} } - m^{ \mathfrak{o} }
\).
Figures
\ref{fig_Uncorrelated_StoppingTimes_TrueNoise}-\ref{fig_Uncorrelated_StoppingTimes_TwoStep}
correspond to the Figures \ref{fig_Uncorrelated_RelativeEfficiencies_TrueNoise},
\ref{fig_Uncorrelated_RelativeEfficiencies_lambda0_1},
\ref{fig_Uncorrelated_RelativeEfficiencies_lambda0_05} and
\ref{fig_Uncorrelated_RelativeEfficiencies_TwoStep} respectively.
The results clearly indicate that for the true empirical noise level, the
sequential early stopping time matches the classical oracle very closely.
For the estimated noise level with \( 
  \lambda_{0} = \sqrt{ \log(p) / n }
\), this is still true for the very sparse signals.
For the estimated noise level with \( 
  \lambda_{0} = \sqrt{ 0.5 \log(p) / n }
\), the stopping times systematically overestimate the classical oracle in the 
sparse signals, which is then corrected by the two-step procedure.

\begin{minipage}{0.5\textwidth}
  \centering
  \begin{figure}[H]
    \includegraphics[width=0.99\linewidth]{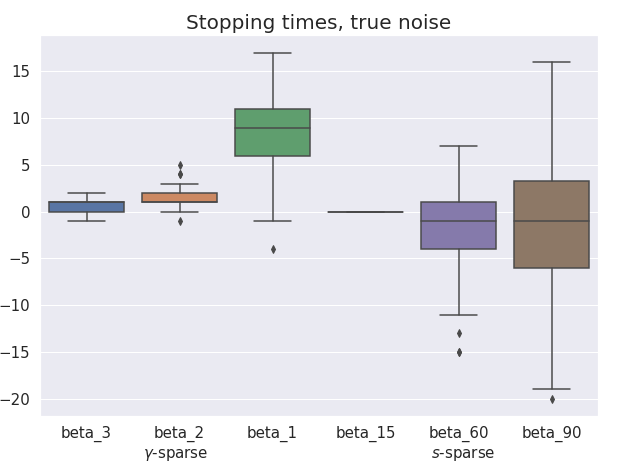}
    \\[-1.0ex]
    \caption{
      Stopping times for early stopping with the true empirical noise level.
    }
    \label{fig_Uncorrelated_StoppingTimes_TrueNoise} 
  \end{figure}
\end{minipage}
\begin{minipage}{0.5\textwidth}
  \centering
  \begin{figure}[H]
    \includegraphics[width=0.99\linewidth]{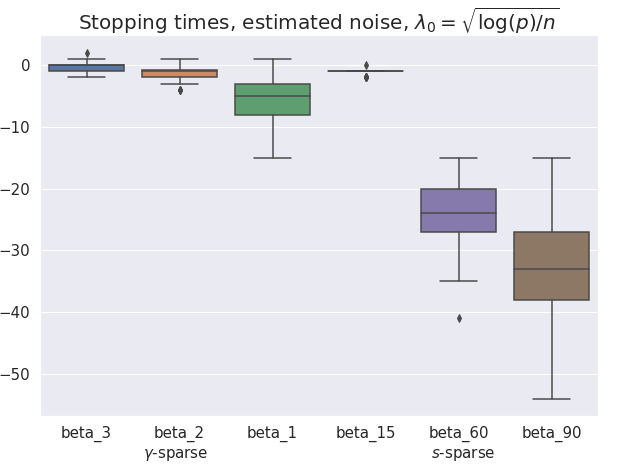}
    \\[-1.0ex]
    \caption{
      Stopping times for early stopping with the estimated empirical noise level
      for \( \lambda_{0} = \sqrt{ \log(p) / n } \).
    }
    \label{fig_Uncorrelated_StoppingTimes_lambda0_1} 
  \end{figure}
\end{minipage}

\begin{minipage}{0.5\textwidth}
  \centering
  \begin{figure}[H]
    \includegraphics[width=0.99\linewidth]{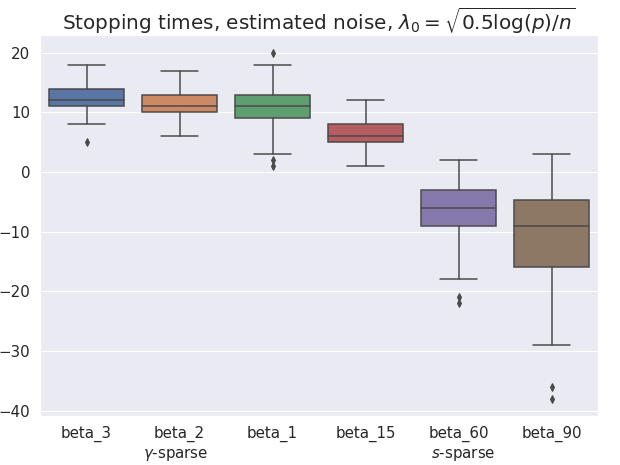}
    \\[-1.0ex]
    \caption{
      Stopping times for early stopping with the estimated empirical noise level
      for \( \lambda_{0} = \sqrt{ 0.5 \log(p) / n } \).
    }
    \label{fig_Uncorrelated_StoppingTimes_lambda_05} 
  \end{figure}
\end{minipage}
\begin{minipage}{0.5\textwidth}
  \centering
  \begin{figure}[H]
    \includegraphics[width=0.99\linewidth]{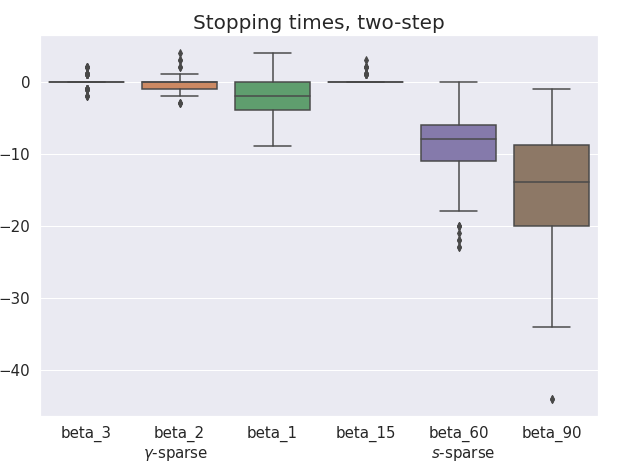}
    \\[-1.0ex]
    \caption{
      Stopping times for the two-step procedure.
    }
    \hspace{1cm}
    \label{fig_Uncorrelated_StoppingTimes_TwoStep} 
  \end{figure}
\end{minipage}

\begin{figure}[H]
  \centering
  \includegraphics[width=0.497\textwidth]{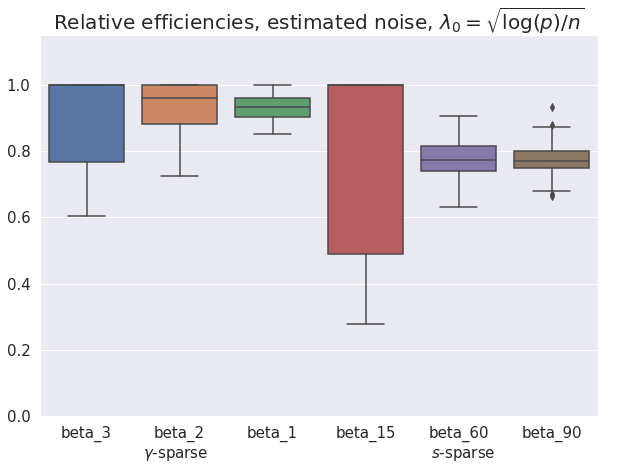}
  \includegraphics[width=0.497\textwidth]{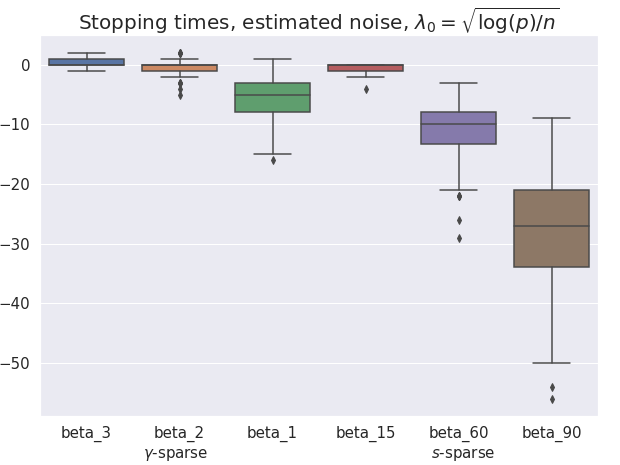}
  \\[-1.0ex]
  \caption{
    Boxplots for the relative efficiencies and the deviation \(
      \tau - m^{ \mathfrak{o} } 
    \) of the stopping time from the classical oracle for the estimated noise
    level with \( 
      \lambda_{0} = \sqrt{ \log(p) / n }
    \) in the correlated design setting.
  }
  \label{fig_Correlated_lambda0_1} 
\end{figure}

\begin{figure}[H]
  \centering
  \includegraphics[width=0.497\textwidth]{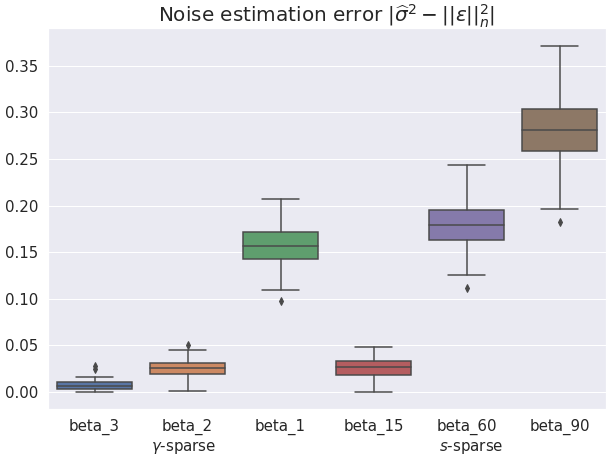}
  \includegraphics[width=0.497\textwidth]{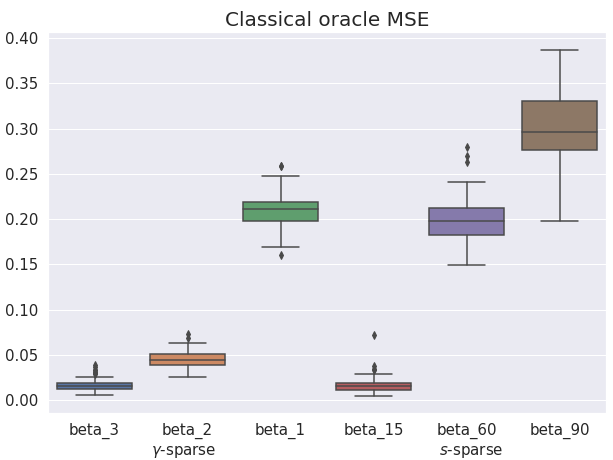}
  \\[-2.0ex]
  \caption{
    Boxplots for the noise estimation errors with \( 
      \lambda_{0} = \sqrt{ \log(p) / n }
    \) together with the classical oracle risk in the correlated design setting.
  }
\end{figure}

For a correlated design simulation, we use the same setting as in Section
\ref{sec_NumericalSimulationsAndATwoStepProcedure} but instead of \( 
  \Gamma = I_{p}
\), we consider the covariance matrix
\begin{align}
  \Gamma: 
  = 
  \begin{pNiceMatrix}[columns-width=.5cm]
    1 & a       & b      &        &   \\
    a & \Ddots  & \Ddots & \Ddots &   \\
    b & \Ddots  &        &        & b \\
      & \Ddots  &        &        & a \\
      &         & b      & a      & 1 \\
  \end{pNiceMatrix}
  \in \mathbb{R}^{ p \times p } 
  \qquad \text{ for } a = 0.4, b = 0.1.
\end{align}

This allows for substantial serial correlation over the whole set of covariates.
Since \( a + b \le 1 / 2 \), the estimate
\begin{align}
      v^{ \top } \Gamma v
  & = 
      \sum_{ j = 1 }^{p} v_{j}^{2} 
      + 
      2 a \sum_{ j = 1 }^{ p - 1 } v_{j} v_{ j + 1 } 
      + 
      2 b \sum_{ j = 1 }^{ p - 2 } v_{j} v_{ j + 2 } 
      > 
      ( 1 - 2 ( a + b ) ) \| v \|_{2}^{2}
\end{align}
for all \( v \in \mathbb{R}^{p} \) guarantees that \( \Gamma \) is a well 
defined positive definite covariance matrix.
Coincidentally, this also guarantees that the cumulative coherence satisfies \( 
  \mu(m) \le 1 / 2
\) for all \( m \ge 1 \). 
By Example \ref{expl_BoundedCovariance} (b), we can therefore assume that
Assumption \hyperref[ass_CovB]{\textbf{{\color{blue} (CovB)}}} is satisfied.
The medians of the classical oracles \( m^{ \mathfrak{o} } \) are given by \( 
  ( 5, 10, 28, 15, 53, 72 )
\) in the same order as the signals are displayed in Figure
\ref{fig_Correlated_lambda0_1}.
The medians of the balanced oracle \( m^{ \mathfrak{b} } \) are given by \( 
  ( 6, 13, 54, 15, 55, 80 )
\). 
Here, both the early stopping procedure with the noise estimate for \( 
  \lambda_{0} = \sqrt{ \log(p) / n }
\) and the two-step procedure match the benchmark results for the full Akaike
selection from Ing \cite{Ing2020ModelSelection} and the 5-fold cross-validated
{\tt LassoCV} from {\tt Scikit-learn} \cite{PedregosaEtal2011ScikitLearn}.

\begin{figure}[H]
  \centering
  \includegraphics[width=0.497\textwidth]{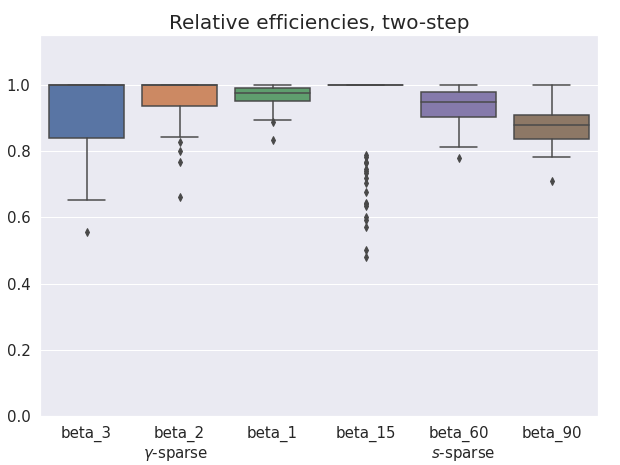}
  \includegraphics[width=0.497\textwidth]{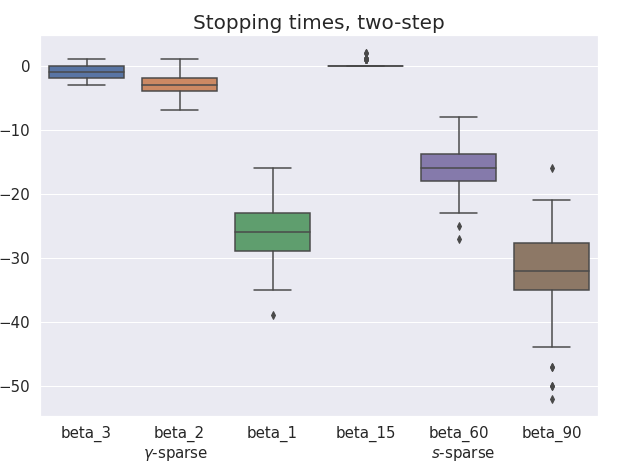}
  \\[-1.0ex]
  \caption{
    Boxplots for the relative efficiencies and the deviation \( 
      \tau_{ \text{two-step} }- m^{ \mathfrak{o} }
    \) of the two-step procedure from the classical oracle for an estimated
    noise level with \(
      \lambda_{0} = \sqrt{ 0.5 \log(p) / n }
    \) in the correlated design setting.
  }
\end{figure}

\begin{figure}[H]
  \centering
  \includegraphics[width=0.497\textwidth]{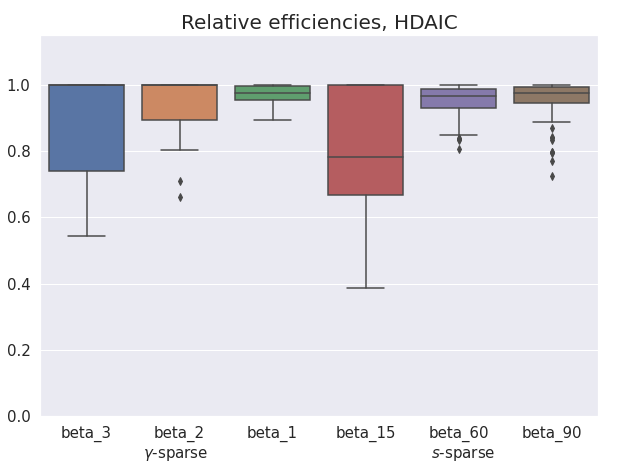}
  \includegraphics[width=0.497\textwidth]{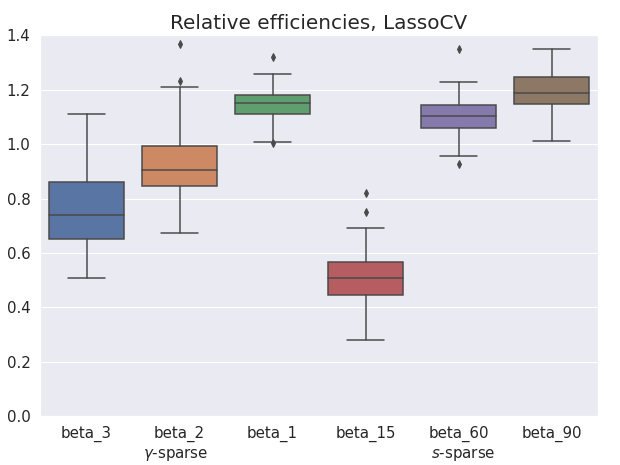}
  \\[-1.0ex]
  \caption{
    Boxplots for the relative efficiencies for the Akaike criterion from
    Ing \cite{Ing2020ModelSelection} with \( 
      C_{ \text{HDAIC} } = 2
    \) and the Lasso based on 5-fold cross-validation in the correlated design
    setting.
  }
\end{figure}

\begin{figure}[H]
  \centering
  \includegraphics[width=0.497\textwidth]{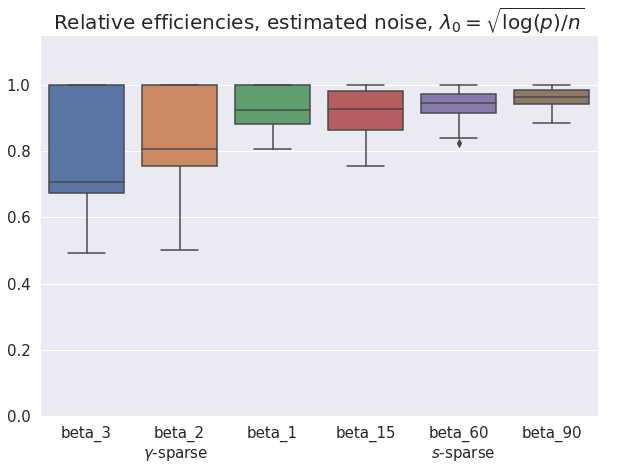}
  \includegraphics[width=0.497\textwidth]{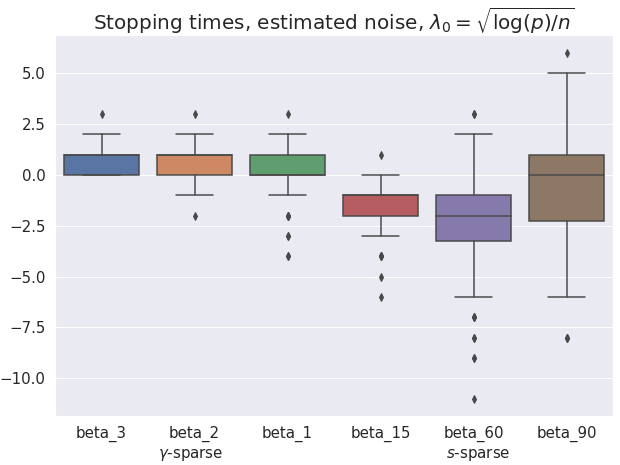}
  \\[-1.0ex]
  \caption{
    Boxplots for the relative efficiencies and the deviation \(
      \tau - m^{ \mathfrak{o} } 
    \) of the stopping time from the classical oracle for the estimated noise
    level with \( 
      \lambda_{0} = \sqrt{ \log(p) / n }
    \) in the classification setting.
  }
\end{figure}

\begin{figure}[H]
  \centering
  \includegraphics[width=0.497\textwidth]{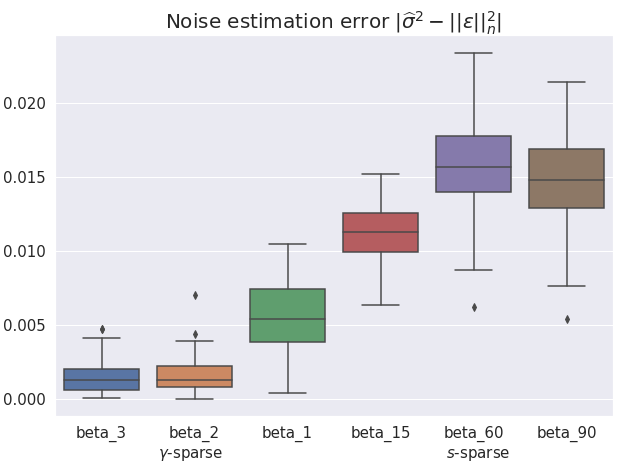}
  \includegraphics[width=0.497\textwidth]{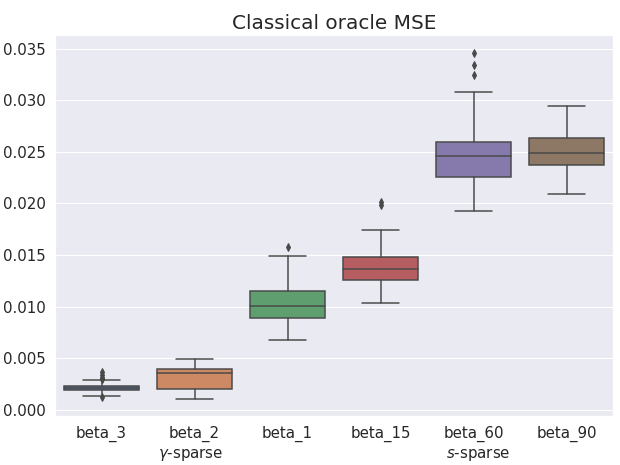}
  \\[-1.0ex]
  \caption{
    Boxplots for the noise estimation errors with \( 
      \lambda_{0} = \sqrt{ \log(p) / n }
    \) together with the classical oracle risk in the classification setting.
  }
\end{figure}

For the classification setting in Example \ref{expl_RegressionAndClassification}
(b), we maintain the correlated design and sample Bernoulli distributed labels
\( Y_{i}, i = 1, \dots, n \), with 
\begin{align}
  \mathbb{P} \{ Y_{i} = 1 \} 
  = 
  \max \Big( 
    \min \Big( 
        0.5 
      +
        \sum_{ j = 1 }^{p} 
        \tilde \beta^{ (r) }_{j} X_{i}^{ (j) }, 
      1 
    \Big),
    0
  \Big), 
  \qquad r \in \{ 3, 2, 1, 15, 60, 90 \},
\end{align}
where the \( \tilde \beta^{ (r) } \) are rescaled versions of the
coefficients \( \beta^{ (r) } \) from Section
\ref{sec_NumericalSimulationsAndATwoStepProcedure}. 
In particular,
\begin{align}
  \tilde \beta^{ (3) } = 0.03 \beta^{ (3) }, \qquad 
  \tilde \beta^{ (2) } = 0.03 \beta^{ (2) }, \qquad 
  \tilde \beta^{ (1) } = 0.1  \beta^{ (1) }, 
  \\
  \tilde \beta^{ ( 15 ) } = 0.1  \beta^{ ( 15 ) }, \qquad 
  \tilde \beta^{ ( 60 ) } = 0.1  \beta^{ ( 60 ) }, \qquad 
  \tilde \beta^{ ( 90 ) } = 0.1  \beta^{ ( 90 ) }. 
  \notag 
\end{align}
The rescaling guarantees that with high probability, the values of the linear
signals are in between \( [ 0, 1 ] \) and the linear model is indeed a good 
approximation for the simulated data.
In Algorithm \ref{alg_OMP}, we add an intercept column 
\( X^{0} = 1 \in \mathbb{R}^{n} \) to the design.
In this setting, estimating the noise level becomes essential, since it depends
on the coefficients themselves. 
The medians of the empirical noise levels \( \| \varepsilon \|_{n}^{2} \) are 
given by \( 
  ( 0.15, 0.18, 0.28, 0.12, 0.19, 0.21 )
\) \\ 

\begin{figure}[H]
  \centering
  \includegraphics[width=0.497\textwidth]{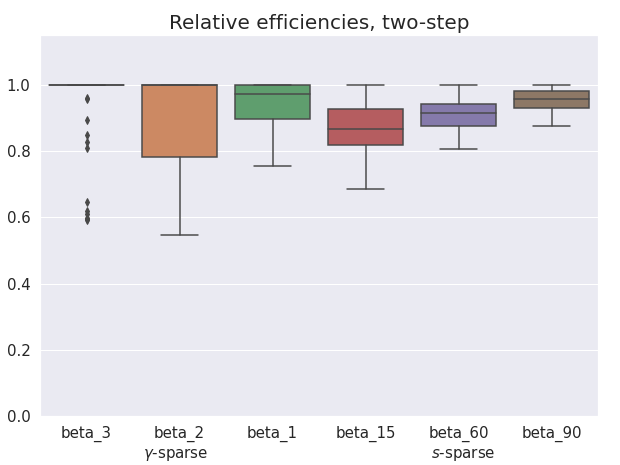}
  \includegraphics[width=0.497\textwidth]{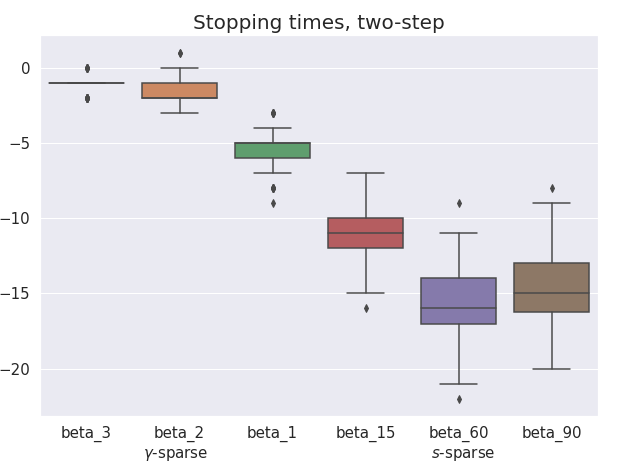}
  \\[-2.0ex]
  \caption{
    Boxplots for the relative efficiencies and the deviation \( 
      \tau_{ \text{two-step} } - m^{ \mathfrak{o} } 
    \) of the two-step procedure from the classical oracle for an estimated
    noise level with \(
      \lambda_{0} = \sqrt{ 0.5 \log(p) / n }
    \) in the classification setting.
  }
\end{figure}

\begin{figure}[H]
  \centering
  \includegraphics[width=0.497\textwidth]{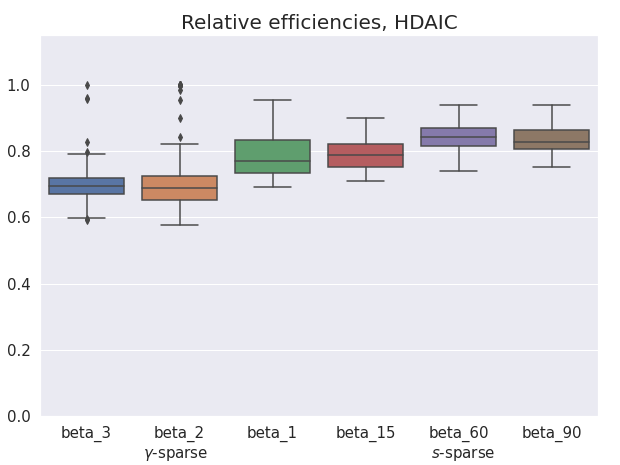}
  \includegraphics[width=0.497\textwidth]{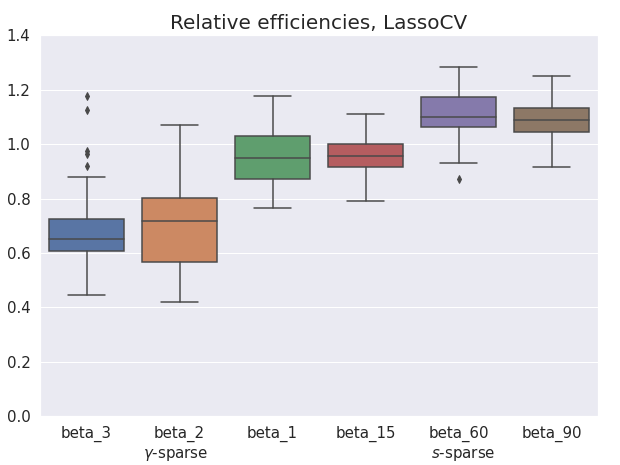}
  \\[-2.0ex]
  \caption{
    Boxplots for the relative efficiencies for the Akaike criterion from
    Ing \cite{Ing2020ModelSelection} with \( 
      C_{ \text{HDAIC} } = 0.5
    \) and the Lasso based on 5-fold cross-validation in the classification
    setting.
  }
\end{figure}

\noindent in the order in which the signals are displayed. 
We present the same plots as in the correlated regression setting.
The median oracles \( m^{ ( \mathfrak{o} ) } \) and \( m^{ ( \mathfrak{b} )
} \) are given by \( 
  ( 2, 3, 6, 13, 14, 8 )
\) and 
\( 
  ( 7, 5, 12, 29, 32, 28 )
\) respectively. 
For both the two-step procedure and the full Akaike selection, we use a constant
\( C_{ \text{AIC} } = C_{ HDAIC } = 0.5 \), which is two times the maximal
variance of a Bernoulli variable.


\end{appendix}

\begin{acks}[Acknowledgments]
  The author is very grateful for the discussions with Markus Reiß and Martin
  Wahl that were indispensable during the preparation of this paper.
\end{acks}
\begin{funding}
  The research of the author has been partially funded by the Deutsche
  Forschungsgemeinschaft (DFG) -- Project-ID 318763901-SFB1294.
%
\end{funding}




\end{document}